\documentclass[10pt,onecolumn]{article}
\usepackage{latexsym}
\usepackage{amssymb, graphicx, amsmath}
\usepackage{booktabs}
\usepackage{cases}
\usepackage[normalem]{ulem}
\usepackage{amsmath,amsthm,amssymb,amsfonts}
\usepackage{thmtools}

\makeatletter
\@ifundefined{newcounteralias}{}{%
  \renewcommand\thmt@autorefsetup{%
    \@xa\def\csname\thmt@envname autorefname\@xa\endcsname\@xa{\thmt@thmname}%
  }%
}
\makeatother

\usepackage{graphicx}
\usepackage{epstopdf}
\usepackage{multirow}
\usepackage{subcaption}
\usepackage{url}
\usepackage{listings}
\usepackage{float}
\usepackage{ulem}

\usepackage{empheq}
\usepackage{hyperref}
\usepackage{cleveref}
\usepackage{titlesec}

\usepackage{enumitem}

\usepackage{bbm}

\titlespacing*{\section}{0pt}{1.6ex}{1ex}
\titlespacing*{\subsection}{0pt}{1.2ex}{0.8ex}
\titlespacing*{\subsubsection}{0pt}{1.0ex}{0.6ex}

\usepackage{titlesec}

\titleformat{\section}       
{\normalfont\large\bfseries} 
{\thesection}{1em}{}
\titlespacing*{\section}
{0pt}{2.5ex plus 0.5ex minus 0.2ex}{1.2ex plus 0.2ex}

\titleformat{\subsection}    
{\normalfont\normalsize\bfseries} 
{\thesubsection}{1em}{}
\titlespacing*{\subsection}
{0pt}{2ex plus 0.5ex minus 0.2ex}{0.8ex plus 0.2ex}

\titleformat{\subsubsection} 
{\normalfont\normalsize\bfseries} 
{\thesubsubsection}{1em}{}
\titlespacing*{\subsubsection}
{0pt}{1.5ex plus 0.5ex minus 0.2ex}{0.5ex plus 0.2ex}

\usepackage[dvipsnames]{xcolor}
\definecolor{sintefblue}{RGB/cmyk}{0,60,101/1,.57,0,.4}

\definecolor{sintefdarkgreen}{RGB/cmyk}{0, 70, 40/.93,.43,.92,.52}

\definecolor{sintefred}{RGB/cmyk}{190, 60,55/19, 86,77, 8}
\definecolor{sinteflilla}{RGB/cmyk}{120,  0,80/48,100,27,31}

\usepackage{threeparttable}

\newcommand{\RNum}[1]{\uppercase\expandafter{\romannumeral #1\relax}}

\def \[{\begin{equation}}
\def \]{\end{equation}}
\def \]{\end{equation}}

\newtheorem{theorem}{Theorem}[section]

\newtheorem{lemma}[theorem]{Lemma}
\newtheorem{proposition}[theorem]{Proposition}
\newtheorem{corollary}{Corollary}[theorem]

\newtheorem{assumption}[theorem]{Assumption}
\newtheorem{remark}[theorem]{Remark}

\theoremstyle{definition}
\newtheorem{definition}[theorem]{Definition}

\crefname{assumption}{assumption}{assumptions}

\makeatletter
\@addtoreset{equation}{section}
\@addtoreset{figure}{section}
\@addtoreset{table}{section}
\makeatother

\usepackage[dvipsnames]{xcolor}
\usepackage{caption}
\usepackage{subcaption}
\usepackage{booktabs}
\usepackage{multirow}

\usepackage{tikz}
\usepackage{anyfontsize}
\usetikzlibrary{positioning,calc}
\tikzset{global scale/.style={
		scale=#1,
		every node/.append style={scale=#1}
	}
}
\tikzset{eaxis/.style={->,>=stealth}}
\tikzset{elegant/.style={smooth,thick,samples=50,cyan}}
\newcommand*{\num}{pi}

\numberwithin{equation}{section}

\newcommand{\bR}{\mathbb{R}}
\newcommand{\cQ}{\mathcal{Q}}
\newcommand{\cS}{\mathcal{S}}
\newcommand{\cA}{\mathcal{A}}
\newcommand{\cG}{\mathcal{G}}

\newcommand{\cF}{\mathcal{F}}
\newcommand{\cK}{\mathcal{K}}
\newcommand{\bT}{\mathbb{T}}
\newcommand{\bZ}{\mathbb{Z}}
\newcommand{\cE}{\mathcal{E}}
\newcommand{\cR}{\mathcal{R}}
\newcommand{\cT}{\mathcal{T}}
\newcommand{\cN}{\mathcal{N}}

\DeclareMathOperator{\glimsup}{g-lim\,sup}
\newcommand{\glimsupn}{\underset{n \to \infty}{\glimsup\,}}
\DeclareMathOperator{\gliminf}{g-lim\,inf}
\newcommand{\gliminfn}{\underset{n \to \infty}{\gliminf\,}}
\newcommand{\hatsetd}{ d \hat{\kern -.15em l }}
\newcommand{\setd}{{ d \kern -.15em l}}

\definecolor{newblue}{rgb}{0.03, 0.27, 0.49}
\usepackage{hyperref}
\hypersetup{
	colorlinks,
	citecolor=sintefblue,
	filecolor=newblue,
	linkcolor=sintefred,
	urlcolor=newblue
}

\begin{document}

\begin{center}
\renewcommand{\thefootnote}{\fnsymbol{footnote}}

{\Large  \bf Learning to Control: The iUzawa-Net for Nonsmooth Optimal Control of Linear PDEs}\\

\bigskip

 {\bf Yongcun Song}\footnote{\parbox[t]{16cm}{
 Division of Mathematical Sciences, School of Physical and Mathematical Sciences, Nanyang Technological University, 21
Nanyang Link, 637371, Singapore. This
author was supported by the NTU Start-Up Grant. Email: yongcun.song@ntu.edu.sg}\vspace{0.5em}}
\quad
{\bf Xiaoming Yuan}\footnote{\parbox[t]{16cm}{
Department of Mathematics, The University of Hong Kong, Pok Fu Lam, Hong Kong SAR, China. This author was supported by the National Key Research and Development Program of China (2024YFC3711804) and the Croucher Senior Research Fellowship. Email: xmyuan@hku.hk
}\vspace{0.5em}}
\quad
{\bf Hangrui Yue}\footnote{\parbox[t]{16cm}{
School of Mathematical Sciences, Nankai University, Tianjin 300071, China. This author was supported by the National Natural Science Foundation of China (No. 12301399).  Email: yuehangrui@gmail.com}\vspace{0.5em}}
\quad
{\bf Tianyou Zeng}\footnote{\parbox[t]{16cm}{
Department of Mathematics, The University of Hong Kong, Pok Fu Lam, Hong Kong SAR, China. This author was supported by Hong Kong PhD Fellowship Scheme.  Email: logic@connect.hku.hk
}}

\bigskip
October 26, 2026

\renewcommand{\thefootnote}{\arabic{footnote}}
\setcounter{footnote}{0}


\end{center}

\bigskip

{\small

\parbox{0.95\hsize}{

\hrule

\medskip

{\bf Abstract.}
We propose an optimization-informed deep neural network approach, named iUzawa-Net, aiming for the first solver that enables real-time solutions for a class of nonsmooth optimal control problems of linear partial differential equations (PDEs). The iUzawa-Net unrolls an inexact Uzawa method for saddle point problems, replacing classical preconditioners and PDE solvers with specifically designed learnable neural networks. We prove universal approximation properties and establish the asymptotic $\varepsilon$-optimality for the iUzawa-Net, and validate its promising numerical efficiency through nonsmooth elliptic and parabolic optimal control problems. Our techniques offer a versatile framework for designing and analyzing various optimization-informed deep learning approaches to optimal control and other PDE-constrained optimization problems. The proposed learning-to-control approach synergizes model-based optimization algorithms and data-driven deep learning techniques, inheriting the merits of both methodologies.
\medskip

\medskip
 
\medskip

\hrule

}}

\bigskip

\section{Introduction}

Optimal control models governed by partial differential equations (PDEs) are crucial for studying complex systems in various areas. 
Typical examples include optimal design in aerodynamics \cite{martins2022aerodynamic}, precision control in chemical processes \cite{christofides2002nonlinear}, the development of targeted drug delivery systems in medicine \cite{chakrabarty2005optimal}, performance optimization in digital twins \cite{antil2024mathematical}, data assimilation in weather forecasting \cite{fisher2009data}, and control of flexible robotic manipulators \cite{zhang2005pde}, just to name a few. 
In practice, nonsmoothness naturally arises in optimal control models, see e.g. \cite{hinze2009optimization,lions1971optimal,manzoni2021optimal,troltzsch2010optimal}, as constraints and nonsmooth regularizations are often imposed on the control variables to ensure certain desired properties, such as boundedness \cite{lions1971optimal}, sparsity \cite{stadler2009elliptic}, and discontinuity \cite{song2024admm}.
The nonsmooth structure and PDE constraints preclude direct application of standard optimization algorithms, such as conjugate gradient and quasi-Newton methods.
Furthermore, emerging applications increasingly demand efficient or even real-time solutions.
In this context, classical numerical methods struggle due to the high-dimensional and ill-conditioned algebraic systems arising from mesh-based discretizations, which are computationally expensive to solve.
These challenges necessitate carefully designed and problem-specific numerical methods for nonsmooth optimal control problems, especially in real-time scenarios.

\subsection{Model}
Let $\Omega\subset\mathbb{R}^d$  with $d \geq 1$ be a bounded domain with Lipschitz boundary.
We consider the following general optimal control problem initiated in \cite{lions1971optimal} with a linear PDE constraint:
\begin{equation}\label{eq:opt-ctrl}
    \begin{aligned}
         \min_{u \in U, \ y \in Y} ~ \frac{1}{2} \left\| y - y_d \right\|_Y^2 + \frac{1}{2} \left\langle u, Nu \right\rangle_U + \theta(u),
        \quad \ \ \mathrm{s.t.} ~ y = S(u+f),
    \end{aligned}
\end{equation}
where $U$ and $Y$ are Hilbert spaces; $u \in U$ and $y \in Y$ are the control and state variables, respectively; and $N: U \to U$ is a linear, bounded, self-adjoint, and positive definite operator with $\langle u, Nu \rangle_U \geq c_0\|u\|_U^2$ for some constant $c_0 > 0$. 
The constraint $y = S(u + f)$ represents a well-posed linear PDE defined on $\Omega$ with $S: U \to Y$ being the corresponding linear solution operator and $f \in U$ being an uncontrollable source term\footnote{We adopt the conventional formulation in which $u$ and $f$ are treated as distinct variables, representing the controllable and uncontrollable parts of the PDE parameter, rather than absorbing $f$ into $u$ through a change of variables.
}.
We denote by  $S^*: Y \to U$ the adjoint operator of $S$, which is assumed to be well-defined.
The function $y_d \in Y$ represents a given desired state, and the functional $\theta: U \to \bR \cup \{+\infty\}$ is assumed to be convex, lower semicontinuous, proper but \emph{nonsmooth}, which imposes a constraint or regularization on the control variable. Throughout the remainder of this work, we fix $U = L^2(\Omega)$ and $Y = L^2(\Omega)$ for simplicity of presentation.

We treat $y_d$ and $f$ as \emph{parameters} of the problem \eqref{eq:opt-ctrl}, and focus on the challenging scenario where \eqref{eq:opt-ctrl} needs to be solved \emph{repeatedly} for \emph{many different} parameter instances, and thus real-time solution of each case of \eqref{eq:opt-ctrl} with fixed $(y_d, f)^\top$ is required\footnote{For simplicity, we focus on the parameters $y_d$ and $f$. Our methodology can be easily generalized if other parameters in \eqref{eq:opt-ctrl} are considered, e.g., the boundary values and initial values.}.
This scenario arises in numerous important applications, including predictive control of PDEs \cite{dubljevic2006predictive,dubljevic2005predictive}, real-time optimal control \cite{behrens2014real,biegler2007real}, and numerical studies of optimal control problems~\cite{berggren1996computational,glowinski1994exact}.
In particular, emerging applications such as digital twins \cite{antil2024mathematical,grieves2016digital,grieves2023digital,xiu2025computational} impose stringent real-time requirements on the solution of the underlying virtual models, of which~\eqref{eq:opt-ctrl} constitutes a representative example.
Clearly, achieving such demands is typically beyond the capacity of traditional iterative methods, whose common aim is to solve \eqref{eq:opt-ctrl} with a fixed $(y_d, f)^\top$ via a sequence of iterations while each one usually costs considerable time.
This necessitates the development of efficient algorithms capable of solving~\eqref{eq:opt-ctrl} in real time for varying problem parameters $(y_d, f)^\top$.

\subsection{Classical Numerical Methods for Problem (\ref{eq:opt-ctrl})}

In the literature, numerous numerical methods have been developed for solving various specific cases of \eqref{eq:opt-ctrl}.
Algorithms like semismooth Newton (SSN) methods \cite{ulbrich2002semismooth}, primal-dual active set methods \cite{kunisch2002primal}, and interior point methods \cite{ulbrich2009primal} generally achieve superlinear convergence rate, but entail solving complex subproblems at each iteration, which are often computationally expensive.
More recently, first-order algorithms have been applied to solving \eqref{eq:opt-ctrl}, including the alternating direction method of multipliers (ADMM) \cite{glowinski2022application}, the primal-dual method \cite{song2023accelerated}, the proximal gradient method \cite{schindele2017proximal,schindele2016proximal}, the Peaceman–Rachford method \cite{glowinski1994exact}, and the inexact Uzawa method \cite{song2019inexact}.
These algorithms typically incur lower computational cost per iteration, which makes them attractive for large-scale problems. Note that their convergence rate is at most linear, meaning that achieving high-accuracy solutions may require a large number of iterations.

All the above algorithms rely heavily on mesh-based discretization schemes for implementation, e.g., finite element methods (FEMs) and finite difference methods (FDMs). Such discretizations give rise to large-scale algebraic systems or optimization subproblems at each iteration, which typically require iterative solvers, resulting in nested (multi-layer) iterations. Moreover, the condition numbers of the resulting subproblems often deteriorate as the mesh is refined, making their numerical solutions computationally demanding or even prohibitive. 
To mitigate the aforementioned difficulties, various preconditioned numerical methods have been developed in the literature, see, e.g., \cite{herzog2010preconditioned,pearson2012regularization,porcelli2015preconditioning,schiela2014operator,song2019inexact,stoll2014one,stoll2012preconditioning}. The effectiveness and efficiency of these algorithms critically depend on the appropriate choice of preconditioners. In particular, the design of effective preconditioners requires a delicate balance between computational cost and convergence enhancement. It must account for the specific mathematical structure of the problem at hand, often necessitating case-by-case analysis and sophisticated problem-dependent design. 

It is noteworthy that classical numerical methods are capable of computing highly accurate solutions. However, they are typically designed for the problem \eqref{eq:opt-ctrl} with fixed parameters $y_d$ and $f$, and any change in the parameters necessitates solving \eqref{eq:opt-ctrl} from scratch. In such cases, classical methods become computationally expensive and are unable to produce real-time solutions. To reduce the computational cost of solving \eqref{eq:opt-ctrl} repeatedly with different parameters, model order reduction methods have been developed in the literature, such as the proper orthogonal decomposition (POD) methods \cite{hinze2005proper,troltzsch2009pod} and the reduced basis methods \cite{bader2016certified,negri2013reduced}; see, e.g., \cite{benner2015survey,hesthaven2016certified} for comprehensive reviews. These methods construct a low-dimensional subspace from high-fidelity solutions in an offline phase, and then solve a reduced problem of \eqref{eq:opt-ctrl} on this subspace in the online phase, which significantly reduces the online computational cost. Nevertheless, the construction of the reduced problems generally requires the discretized PDE operators, and the resulting reduced models are tied to the discretization on which they are constructed.

\subsection{Physics-Informed Neural Networks and Operator Learning for Problem (\ref{eq:opt-ctrl})}

Several deep learning approaches applicable to solving \eqref{eq:opt-ctrl} have been proposed, including physics-informed neural networks (PINNs) \cite{gao2025prox,lai2025hard,song2024admm} and operator-learning-based methods \cite{song2023accelerated,song2024operator}, which differ from classical numerical methods in that they are generally mesh-free, easy to implement, and flexible to adapt to various concrete examples of \eqref{eq:opt-ctrl}.
In \cite{lai2025hard}, the hard-constraint PINNs were proposed for solving control-constrained optimal control problems with interfaces.
In \cite{song2024admm}, PINNs were incorporated within the ADMM for solving the resulting optimal control subproblem in each ADMM iteration.
In \cite{gao2025prox}, a Prox-PINNs framework was proposed for solving general elliptic variational inequalities, which covers \eqref{eq:opt-ctrl} as a special case.
Several operator-learning-enhanced primal-dual approaches were proposed in \cite{song2023accelerated,song2024operator} for solving \eqref{eq:opt-ctrl} with $N = \alpha I$ for some $\alpha > 0$, with \cite{song2023accelerated} focusing on linear PDE constraints and \cite{song2024operator} addressing nonlinear PDE constraints.
They embed deep-learning-based PDE solvers within traditional primal-dual algorithms, accelerating the solution of PDE subproblems by replacing classical iterative PDE solvers with forward passes of deep neural networks.
The combination of computational efficiency, accuracy, and generalization ability makes operator learning particularly suitable for tasks requiring frequent, computationally intensive evaluations.

Notably, the PINNs-based methods discussed above require training a neural network from scratch for solving \eqref{eq:opt-ctrl} with each different problem parameter $(y_d, f)^\top$.
For operator learning methods, current approaches embed neural surrogate models as subroutines for solving PDE subproblems into traditional optimization algorithms, which still require manual design of the stepsizes and preconditioners.
The resulting algorithms usually require a large number of iterations to converge, as indicated in \cite{lai2025hard,song2024operator}. Additionally, 
operator learning methods for PDEs, such as Fourier Neural Operators (FNOs) \cite{li2021fourier}, could in principle be used to learn the mapping from the problem parameters \( (y_d, f)^\top \) to the optimal control $u^*$.  However, these operator learning methods are purely data-driven and do not exploit the mathematical structure of \eqref{eq:opt-ctrl}. As a result, they often lack interpretability and achieve only limited accuracy in solving \eqref{eq:opt-ctrl}, as validated in \Cref{sec:num-exp}.

\subsection{Optimization-Informed Neural Networks}

Optimization-informed neural networks, also known as algorithm unrolling for finite-dimensional models, originated in \cite{gregor2010learning} as data-driven methods for solving optimization problems.
The core idea of algorithm unrolling is to design a neural network where each layer mirrors an iteration of a traditional optimization algorithm.
The resulting neural network can be viewed as being unrolled from traditional optimization algorithms; hence, the interpretability and robustness of its architecture are generally increased.
In particular, all the hyperparameters in the blueprint optimization algorithms, such as stepsizes and preconditioners, are treated as learnable parameters in the neural networks and are automatically optimized during the training process.

Various algorithm-unrolled neural networks have been proposed for finite-dimensional optimization problems, including the learned iterative shrinkage-thresholding algorithm (LISTA) \cite{gregor2010learning}, unrolled proximal gradient descent methods \cite{hosseini2020dense}, unrolled ADMM \cite{sun2016deep,yang2018admm}, and unrolled primal-dual methods \cite{adler2018learned,li2024pdhg}.
Notably, in these neural networks, only a few layers are required to achieve satisfactory numerical accuracy, in contrast to their blueprint optimization algorithms that often require relatively many iterations.
We refer to \cite{hauptmann2025learned,shlezinger2023modelbook,shlezinger2023model} for a comprehensive review of the recent advances in algorithm unrolling for finite-dimensional problems and its various applications.

A straightforward application of the above works to \eqref{eq:opt-ctrl} requires discretizing the underlying infinite-dimensional input and output function spaces into finite-dimensional ones. However, neural network models trained on a specific discretization generally do not generalize well to other discretizations, which can significantly degrade the numerical accuracy. Consequently, developing optimization-informed neural networks that are capable of operating directly in infinite-dimensional settings with PDE constraints becomes a problem of significant interest and importance, yet limited work has been done in this direction. 
In \cite{kratsios2025generative}, the authors proposed a generative neural operator for approximating mappings from the space of convex and G\^{a}teaux differentiable functions to Hilbert spaces, which can be applied to solve parameterized unconstrained convex optimization problems.
However, this method requires evaluating the objective functional and performing adaptive sampling in each forward pass.
In \eqref{eq:opt-ctrl}, evaluating the objective functional entails solving PDEs, which makes such forward passes prohibitively expensive and renders this approach impractical.
To the best of our knowledge, there is currently no optimization-informed neural network architecture that is practically effective for solving the PDE-constrained optimal control problem~\eqref{eq:opt-ctrl}.

\subsection{Methodology}\label{se:method}

We propose an optimization-informed neural network architecture, termed \texttt{iUzawa-Net}, for efficiently solving the optimal control problem \eqref{eq:opt-ctrl}. Structurally, the \texttt{iUzawa-Net} is composed by stacking layers that mirror an inexact Uzawa method for solving \eqref{eq:opt-ctrl}. This approach generalizes the unrolling technique introduced in~\cite{gregor2010learning} to infinite-dimensional settings, while incorporating recent advances in deep learning for approximating mappings between function spaces. As a result, the \texttt{iUzawa-Net} serves as a neural surrogate model for the \emph{solution operator} of \eqref{eq:opt-ctrl}, defined as the parameter-to-control mapping
\begin{equation}\label{eq:param-to-sol-oper}
	T: Y \times U \to U, \quad (y_d, f)^\top \mapsto u^*,
\end{equation}
where $u^*$ is the optimal control of \eqref{eq:opt-ctrl} corresponding to the problem parameter $(y_d, f)^\top$.
After training, we perform an efficient forward pass of the trained surrogate model to approximate $T(y_d, f)$ for each pair of input parameters $y_d$ and $f$, hence numerically solving \eqref{eq:opt-ctrl} in real-time.
In the following, we briefly outline the design of the \texttt{iUzawa-Net} and detail its implementation in \Cref{sec:iuzawa-net}.

\subsubsection{An inexact Uzawa method for (\ref{eq:opt-ctrl})}
We start with the saddle-point problem reformulation of \eqref{eq:opt-ctrl}:
\begin{equation}\label{eq:opt-ctrl-pd}
	\min_{u \in U} \, \max_{p \in Y} ~ \frac{1}{2} \left\langle u, Nu \right\rangle_U + \theta(u) + \left\langle p, Su \right\rangle_Y - \frac{1}{2} \left\| p \right\|^2_Y - \left\langle p, y_d - Sf \right\rangle_Y,
\end{equation}
where $p \in Y$ is the associated Fenchel dual variable.
It is well-known that, for any given $(y_d, f)^\top \in Y \times U$, \eqref{eq:opt-ctrl-pd} admits a solution $(u^*, p^*) \in U \times Y$ by the Fenchel-Rockafellar duality \cite[Section 15.3]{bauschke2017convex}, and the solution is unique due to the strict convexity and concavity of \eqref{eq:opt-ctrl-pd} with respect to $u$ and $p$, respectively. Hence,  the solution operator $T$ in \eqref{eq:param-to-sol-oper} is well-defined.
Furthermore, \eqref{eq:opt-ctrl-pd} admits the following necessary and sufficient optimality condition:
\begin{equation}\label{eq:opt-ctrl-opt-cond}
	\begin{aligned}
		 0 \in Nu^* + \partial \theta(u^*) + S^* p^*, \quad
		 0 = Su^* - p^* - (y_d - Sf). 
	\end{aligned}
\end{equation}
Inspired by \cite{bramble1997analysis,elman1994inexact,song2019inexact}, we consider the following inexact Uzawa algorithm for \eqref{eq:opt-ctrl}:
\begin{equation}\label{eq:inexact-uzawa}
	\left\{
	\begin{aligned}
		& u^{k+1} = \left(Q_A + \partial \theta\right)^{-1} \left( \left( Q_A - N \right)  u^k - S^* p^k\right), \\
		& p^{k+1} = p^k + Q_S^{-1} (S u^{k+1} - p^k + Sf - y_d),
	\end{aligned}
	\right.
\end{equation}
where $Q_A: U \to U$ and $Q_S: Y \to Y$ are two self-adjoint positive definite operators preconditioning on $N$ and $SN^{-1}S^* + I$, respectively.
The convergence of \eqref{eq:inexact-uzawa} is guaranteed for any choice of
preconditioners satisfying \( Q_A \succeq N \) and
\( Q_S \succeq S N^{-1} S^* + I \); this follows from a straightforward
extension of the analysis in~\cite{song2019inexact}.

To implement \eqref{eq:inexact-uzawa}, one needs to construct the preconditioners \( Q_A \) and \( Q_S \) and solve two PDEs $z^k = S^*p^k$ and $y^{k+1} = Su^{k+1}$ at each iteration. With suitable choices
of \( Q_A \) and \( Q_S \), the inexact Uzawa method~\eqref{eq:inexact-uzawa}
efficiently solves \eqref{eq:opt-ctrl} with fixed \( y_d \)
and \( f \).
In practice, the scheme \eqref{eq:inexact-uzawa} involves a
trade-off between the algebraic measurement of convergence speed in terms of the outer iterations and the computational cost of each inner iteration. Balancing these effects relies on carefully designing the preconditioners \( Q_A \) and \( Q_S \), which
typically requires problem-dependent tuning.

\subsubsection{iUzawa-Net: an overview}
The \texttt{iUzawa-Net} is designed by converting the iterative structure of \eqref{eq:inexact-uzawa} into a multi-layer neural network, which allows the numerically challenging operators $S^*$ and $S$, as well as the preconditioners $Q_A$ and $Q_S$, to be parameterized as trainable neural sub-networks.
To this end, we specify $Q_A = N + \tau I$ with some $\tau \geq 0$, and the iterative scheme \eqref{eq:inexact-uzawa} reduces to
\begin{equation}\label{eq:inexact-uzawa-qa-spec}
	\left\{
	\begin{aligned}
		& u^{k+1} = \left(N + \tau I + \partial \theta\right)^{-1} \left(\tau u^k - S^* p^k\right), \\
		& p^{k+1} = p^k + Q_S^{-1} (S u^{k+1} - p^k + Sf - y_d).
	\end{aligned}
	\right.
\end{equation}
Then, we employ two learnable neural networks $\cS^k: U \to Y$ and $\cA^k: Y \to U$ to approximate the solution operator $S$ and its adjoint $S^*$, respectively, in the $k$-th iteration of \eqref{eq:inexact-uzawa-qa-spec}. 
Simultaneously, we construct two learnable deep surrogate models $\cQ_A^k: U \to U$ and $\cQ_S^k: Y \to Y$ to approximate the operators $(N + \tau I + \partial \theta)^{-1}$ and $Q_S^{-1}$.
By treating the $k$-th iteration of \eqref{eq:inexact-uzawa-qa-spec} with the surrogate operators as the $k$-th layer of the \texttt{iUzawa-Net}, and stacking the first $L$ layers together, we obtain the following neural network architecture:
\begin{subequations}\label{eq:inexact-uzawa-duf}
	\begin{empheq}[left=\empheqlbrace]{align}
		& \cT(y_d, f; \theta_{\cT}) = u^L, \\
        & \left\{\begin{aligned}
        & u^{k+1} = \cQ_{A}^k (\tau u^k - \cA^k p^k), \\
        & p^{k+1} = p^k + \cQ_S^k (\cS^k (u^{k+1} + f) - p^k - y_d), 
        \end{aligned}\right. \quad \forall k = 0, \ldots, L-1, \label{eq:inexact-uzawa-duf-iter} \\
        & u^0 = 0, \quad p^0 = 0.
	\end{empheq}
\end{subequations}
Here, $\cT(y_d, f; \theta_{\cT})$ denotes the proposed \texttt{iUzawa-Net}; $\theta_{\cT}$ denotes the collection of trainable parameters in $\cS$, $\cA$, $\cQ_A$ and $\cQ_S$; and $u^L$ provides an approximation to the optimal control $u^*$ associated with the parameters $y_d$ and $f$.
For notational simplicity, we denote by $u^{k+1}$ and $p^{k+1}$
both the iterates generated by the inexact Uzawa method in \eqref{eq:inexact-uzawa-qa-spec} and the outputs of the 
$k$-th layer of the \texttt{iUzawa-Net} in \eqref{eq:inexact-uzawa-duf}. The abuse of notation should not cause confusion, as the meaning is clear from the context.

In implementation, the parameter $\tau>0$ can be chosen arbitrarily by users. We employ neural networks in operator learning for PDEs (e.g., FNO or DeepONet \cite{lu2021learning}) as the architectures for the PDE solution operators $\mathcal S^k$ and $\mathcal A^k$. The module $\mathcal Q_S^k$ is specifically constructed to approximate a symmetric positive definite preconditioner. For $\mathcal Q_A^k$, we consider the important case in which $(N+\tau I+\partial\theta)^{-1}$ acts pointwise on its input\footnote{This setting covers a broad range of applications. In many of these cases, closed-form expressions of $(N+\tau I+\partial\theta)^{-1}$ are not available, which motivates the use of a learned neural network to approximate it.}, and we design the architecture of $\mathcal Q_A^k$ accordingly to capture this structural property. The specific architectures of all modules are described in detail in Section~\ref{sec:iuzawa-net}.
A schematic of the full network architecture of the \texttt{iUzawa-Net} is provided in \Cref{fig:iuzawa-net}.

\subsubsection{Theoretical results of the iUzawa-Net}

We first follow the established literature on neural network approximation theory (e.g., \cite{kovachki2021universal,pinkus1999approximation}) and show that, under mild conditions, the solution operator 
$T$ defined in \eqref{eq:param-to-sol-oper} can be approximated to arbitrary accuracy on any compact set by an \texttt{iUzawa-Net} with two layers.
This result implies that the \texttt{iUzawa-Net} serves as a \emph{universal approximator} for the class of solution operators of \eqref{eq:opt-ctrl}.
Nevertheless, it does not guarantee two properties that are expected to be satisfied by optimization-informed neural networks, namely \emph{algorithm tracking} and \emph{weight tying}, whose definitions will be detailed in \Cref{def:alg-track,def:weight-tying}.
Informally, algorithm tracking implies that the layer outputs of the \texttt{iUzawa-Net} can be interpreted as inexact iterates of the inexact Uzawa method \eqref{eq:inexact-uzawa}, and
weight tying means that the layers in \eqref{eq:inexact-uzawa-duf} share identical parameters. Both properties are naturally expected in an optimization-informed neural network based on \eqref{eq:inexact-uzawa} with fixed stepsize and preconditioners.
The lack of these properties reveals that the above approximation result does not guarantee that the network preserves the underlying optimization structure.

To bridge this gap, we develop an analysis framework that unifies tools from optimization and neural network approximation theory, and establish the following results.
First, we show that there exists an \texttt{iUzawa-Net} that is algorithm tracking and weight tying.
Then, from an optimization perspective, we establish the \emph{asymptotic $\varepsilon$-optimality} of the \texttt{iUzawa-Net}s in the following sense: for any $\varepsilon>0$ and sufficiently large $L \in \mathbb{N}$, there exists an $L$-layer \texttt{iUzawa-Net} such that all the sufficiently deep layer outputs $(u^k,p^k)^\top$ remain in an $\varepsilon$-neighbourhood of the optimal pair $(u^\ast,p^\ast)^\top$ in $U \times Y$.
This, in turn, yields a universal approximation result of the \texttt{iUzawa-Net}s \emph{under the algorithm tracking and weight tying requirements}.
It provides theoretical justification for the module-sharing design adopted in implementation, which substantially reduces model complexity while still achieving satisfactory numerical accuracy.
We further show that these results remain valid for \emph{bounded} subsets of a Sobolev space of positive order, which is of crucial advantage as classical compactness assumptions are restrictive in infinite-dimensional spaces.

\subsection{Contributions}

The \texttt{iUzawa-Net} demonstrated in \Cref{se:method} provides a unified framework that integrates classical iterative methods with modern deep learning techniques for solving \eqref{eq:opt-ctrl}.
Computationally, the \texttt{iUzawa-Net} involves only a single offline training phase; during inference, obtaining a numerical solution of \eqref{eq:opt-ctrl} for a new set of input parameters requires only a forward pass of the neural network. This facilitates real-time solutions for \eqref{eq:opt-ctrl} by avoiding nested iterations in traditional numerical methods.
Moreover, the \texttt{iUzawa-Net} enables data-driven preconditioner design within a single end-to-end trainable manner, which offers accelerated convergence and improved adaptability across different problem settings.
Notably, this brings the additional benefit that given sufficient data, the solution operator of \eqref{eq:opt-ctrl} can be learned even without the explicit knowledge of $S$ and $\theta$.

Theoretically, we prove several universal approximation theorems of the proposed \texttt{iUzawa-Net} under mild assumptions. In particular, one of the results allows approximating the solution operator of \eqref{eq:opt-ctrl} to arbitrary accuracy with shared layer parameters on bounded subsets of Sobolev spaces, avoiding the compactness requirement in conventional universal approximation results.
Furthermore, compared with existing works, our analysis framework unifies tools from optimization and neural network approximation theory, which is of independent interest and provides a framework for designing and analyzing optimization-informed neural networks in infinite-dimensional spaces.

In summary, our approach differs from existing methods for \eqref{eq:opt-ctrl} by simultaneously achieving numerical efficiency, ease of implementation, structural interpretability, and rigorous theoretical guarantees.

\subsection{Organization}

The rest of this paper is organized as follows.
In \Cref{sec:iuzawa-net}, we present the design of the neural network architectures for the learnable modules in the \texttt{iUzawa-Net}.
We then prove a universal approximation theorem for the \texttt{iUzawa-Net} in \Cref{sec:universal-approx}.
In \Cref{sec:stability-analysis}, we show that the outputs of the \texttt{iUzawa-Net} with shared layer parameters approximate the optimal control of \eqref{eq:opt-ctrl} with guaranteed accuracy after passing through sufficiently many layers, which results in a new universal approximation result for the \texttt{iUzawa-Net}.
Several numerical experiments are presented in \Cref{sec:num-exp} to demonstrate the efficiency, accuracy, and generalization ability of the proposed \texttt{iUzawa-Net}.
In particular, we compare the results of the \texttt{iUzawa-Net} with the reference ones obtained by FEM-based traditional algorithms, a POD-based reduced-order method, and other deep learning methods.
We conclude in \Cref{sec:conclusion}.

\section{iUzawa-Net: Modules Design and Training}\label{sec:iuzawa-net}

In this section, we detail the implementation of the \texttt{iUzawa-Net} defined in \eqref{eq:inexact-uzawa-duf}. We first elaborate on the constructions of neural network architectures for the learnable modules $\cS^k$, $\cA^k$, $\cQ_S^k$, and $\cQ_A^k$ in \eqref{eq:inexact-uzawa-duf}. Subsequently, we present the full architecture and a training framework for the \texttt{iUzawa-Net}.

\subsection{Neural Network Architecture for \texorpdfstring{$\cS^k$}{Sk} and \texorpdfstring{$\cA^k$}{Ak}}\label{se:Ak}

For illustration, we apply the FNO \cite{li2021fourier} as the structure of the neural networks $\cS^k$ and $\cA^k$, which is an operator learning framework designed for approximating PDE solution operators.
For completeness, we provide a brief overview of the specific FNO architecture adopted in this work; we refer the reader to~\cite{kovachki2023neural,li2021fourier} for a comprehensive discussion.

Recall that the optimal control problem \eqref{eq:opt-ctrl} is defined on a bounded domain $\Omega \subset \bR^d$.
The FNO parameterizes its integral operators by the Fourier coefficients of their kernels, and evaluates these operators efficiently by the fast Fourier transform (FFT).
Both rely on a periodic domain, i.e., the torus $\bT^d$, on which the Fourier basis is available.
For general nonperiodic $\Omega$, following the standard treatment \cite{kovachki2021universal,li2021fourier}, we assume without loss of generality that the closure $\overline{\Omega} \subset (0, 2\pi)^d$ and embed $\overline{\Omega}$ into the interior of the $d$-dimensional torus $\bT^d$.
We identify $\Omega$ as a subset of $\bT^d$ through this embedding.
The architecture for $\cS^k$ is defined as:
\begin{equation}\label{eq:fno}
    \cS^k (u) = \left.(\mathcal{Q} \circ \cK_{\text{FNO}} \circ \mathcal{P})(\mathcal{E}(u))\right|_\Omega, 
\end{equation}
where $\mathcal{E}: L^2(\Omega) \to L^2(\bT^d)$ is a linear and bounded extension operator satisfying $\mathcal{E}(u)|_\Omega = u$ for all $u \in L^2(\Omega)$; $\mathcal{P}: L^2(\bT^d) \to L^2(\bT^d; \bR^{m_p})$ is a pointwise lifting operator defined by $\mathcal{P}(v)(x) = P v(x)$ with a lifting dimension $m_p$ and a learnable linear transformation $P: \bR \to \bR^{m_p}$; and $\mathcal{Q}: L^2(\bT^d; \bR^{m_p}) \to L^2(\bT^d)$ is a pointwise projection operator satisfying $\mathcal{Q}(v)(x) = Q v(x)$ for some learnable linear transformation $Q: \bR^{m_p} \to \bR$.
We take $\mathcal{E}$ as the zero extension operator unless otherwise specified.
The operator $\cK_{\text{FNO}}: L^2(\bT^d; \bR^{m_p}) \to L^2(\bT^d; \bR^{m_p})$ represents a series of Fourier layers defined by
\begin{equation}\label{eq:fno-layers}
\left\{
\begin{aligned}
    &\cK_{\text{FNO}}(v) = v^{(L)}, \\
    &v^{(\ell + 1)} = \sigma(\mathcal{F}^{-1} \left( \mathcal{R}_\ell \cdot\mathcal{F}_{k_\text{max}} (v^{(\ell)}) \right) + W_\ell v^{(\ell)} {} + b_\ell), \quad \forall \, \ell = 0, \ldots, L-1, \\
    &v^{(0)} = v,
\end{aligned}
\right.
\end{equation}
where $L$ is the number of the Fourier layers.
Here, $\mathcal{F}_{k_\text{max}}$ denotes the Fourier transform truncated to the lowest $k_\text{max}$ Fourier modes in each coordinate; $\mathcal{F}^{-1}$ denotes the inverse Fourier transform; $\mathcal{R}_\ell: (\mathbb{Z} \cap [-k_\text{max}, k_\text{max}])^d \to \mathbb{C}^{m_p \times m_p}$ is a tensor of learnable complex-valued weights on the Fourier domain;  $W_\ell \in \mathbb{R}^{m_p \times m_p}$ is a learnable pointwise linear transformation acting on the original domain; $b_\ell \in L^2(\bT^d; \bR^{m_p})$ is a learnable bias function; and $\sigma$ is a nonlinear activation function acting pointwise, which is assumed to be non-polynomial, globally Lipschitz continuous, and infinitely differentiable as in \cite{kovachki2021universal}.
To ensure that $\mathcal{S}(u)$ is real-valued, we enforce the conjugate symmetry of $\mathcal{R}_\ell$ in the frequency domain.
We employ an identical neural network architecture for $\mathcal{A}^k$ with a distinct set of learnable parameters.

The composition $\cQ \circ \cK_{\text{FNO}} \circ \mathcal{P}: L^2(\bT^d) \to L^2(\bT^d)$ is referred to as an FNO in the literature.
The FNO architecture incorporates both global transformations in the frequency domain and local, pointwise transformations in the physical domain. This structure is particularly capable of capturing the nonlocal dependencies inherent in PDE solution operators $S$ and $S^*$.
Furthermore, it has been shown that FNOs are universal approximators for the class of continuous operators \cite{kovachki2021universal} and possess the discretization invariance property \cite{kovachki2023neural}, allowing the trained model to generalize across different mesh resolutions.
In implementation, if the input $u$ is discretized on a uniform grid $\mathcal{D}$ within the domain $\Omega$, both $\mathcal{F}_{k_\text{max}}$ and $\mathcal{F}^{-1}$ can be computed efficiently using the FFT.

\subsection{Neural Network Architecture for \texorpdfstring{$\cQ_S^k$}{QSk}}\label{se:QSk}

Since the convergence of \eqref{eq:inexact-uzawa} merely requires $Q_S$ to be self-adjoint and the condition $Q_S \succeq SN^{-1}S^* + I$, the choice of $Q_S$ is generally not unique.
Leveraging such flexibility, we design $\cQ_S^k$, the approximator of $Q_S^{-1}$, with a single-layer, purely linear structure to reduce the parameter complexity of the \texttt{iUzawa-Net}.  Despite its simplicity, this design
guarantees that $\cQ_S^k$ is self-adjoint and positive definite, which are essential properties for the asymptotic $\varepsilon$-optimality of the \texttt{iUzawa-Net} discussed in \Cref{sec:stability-analysis}.

We assume that $\overline{\Omega} \subset (0, 2\pi)^d$ and identify the bounded domain $\Omega$ as a subset of $\bT^d$ as in \Cref{se:Ak}.
Our design relies on the following result.

\begin{proposition}\label{cor:spd-oper-fourier}
    Let $\kappa \in L^2(\bT^d; \bR^{m \times m})$ and $W \in \bR^{m \times m}$ for some positive integer $m$.
    For any $v \in L^2(\bT^d; \bR^m)$, define the linear operator $\cK: L^2(\bT^d; \bR^m) \to L^2(\bT^d; \bR^m)$ by
    \begin{equation}\label{eq:spd-oper-fourier}
        \cK(v)(x) = W v(x) + \int_{\bT^d} \kappa(x-y) v(y) dy, \quad \forall v \in L^2(\bT^d; \bR^m), \quad \text{~for~a.e.~} x \in \bT^d.
    \end{equation}
    If $W = V^\top V$ for some $V \in \bR^{m \times m}$ and $\kappa$ takes the form
    \begin{equation*}
        \kappa(y) = \int_{\bT^d} \phi(z - y)^\top \phi(z) \, dz
    \end{equation*}
    for some $\phi \in L^2(\bT^d; \bR^{m \times m})$, then $\cK$ is bounded, self-adjoint, and positive semidefinite.
\end{proposition}

For notational convenience, we define the following convolution operation for matrix-valued functions on $\bT^d$:
\begin{equation}\label{eq:conv-vecfunc-def}
    (v_1 \ast v_2)(x) := \int_{\bT^d} v_1(x-y) v_2(y) \, dy, \quad \forall v_1 \in L^2(\bT^d; \bR^{m_1 \times m_2}), v_2 \in L^2(\bT^d; \bR^{m_2 \times m_3}),
\end{equation}
where the product in the integration is the standard matrix product.
It then follows that
\begin{equation}\label{eq:kappa-conv-repre}
    \kappa (x-y) = \int_{\bT^d} \phi(z-x)^\top \phi(z-y) dz = (\phi^\top_{-} \ast \phi) (x-y),
\end{equation}
where $\phi^\top_{-}(x) := \phi(-x)^\top$.
Note that the operation $\ast$ defined in \eqref{eq:conv-vecfunc-def} is associative.
Motivated by \Cref{cor:spd-oper-fourier} and \eqref{eq:kappa-conv-repre}, we apply $\cK$ in \eqref{eq:spd-oper-fourier} to define the following architecture for $\cQ_S^k$:
\begin{equation}\label{eq:qsk-def-pre}
\begin{aligned}
    &\cQ_S^k(u) = \left. \left( \mathcal{P}^\top \circ \cK \circ \mathcal{P} \right) \left( \mathcal{E}(u) \right) \right|_{\Omega} + \gamma u, \quad \forall u \in L^2(\Omega), \\
    &\cK(v)(x) = 
    \left( V^\top V v(x) + (\phi^\top_{-} \ast \phi \ast v)(x)\right), \quad \forall v \in L^2(\bT^d; \bR^{m_p}), \quad \text{for~a.e.~} x \in \bT^d,
\end{aligned}
\end{equation}
where $\mathcal{E}: L^2(\Omega) \to L^2(\bT^d)$ is the zero extension operator, which is linear and bounded;
$\mathcal{P}: L^2(\bT^d) \to L^2(\bT^d; \bR^{m_p})$ is a pointwise lifting operator satisfying $\mathcal{P}(v)(x) = P v(x)$ for some learnable linear transformation $P: \bR \to \bR^{m_p}$; $\mathcal{P}^\top: L^2(\bT^d; \bR^{m_p}) \to L^2(\bT^d)$ is defined by $\mathcal{P}^\top(v)(x) = P^\top v(x)$; $\phi \in L^2(\bT^d; \bR^{m_p \times m_p})$ and $V \in \bR^{m_p \times m_p}$ are learnable parameters.\footnote{Note that \eqref{eq:spd-oper-fourier} is similar to a single Fourier layer in FNO with a different parameterization strategy that ensures positive semidefiniteness.}
Using \Cref{cor:spd-oper-fourier}, it can be verified that $\cQ_S^k$ defined in \eqref{eq:qsk-def-pre} is linear, bounded, self-adjoint, and positive definite for all $\gamma > 0$.
We remark that the term $\gamma u$ with hyperparameter $\gamma > 0$ not only ensures that $\cQ_S^k$ is strictly positive definite, but also serves a role similar to a residual connection \cite{he2016deep} in the neural network.

Note that $\mathcal{F}(\phi)(k)$ is nonzero only if $k \in \bZ^d$.
Also, the convolution theorem applies to the definition \eqref{eq:conv-vecfunc-def}, i.e.,
\begin{equation*}
    \cF(v_1 \ast v_2)(k) = \cF(v_1)(k) \, \cF(v_2)(k), \quad \forall v_1 \in L^2(\bT^d; \bR^{m_1 \times m_2}), \, v_2 \in L^2(\bT^d; \bR^{m_2 \times m_3}), \, k \in \mathbb{Z}^d,
\end{equation*}
where the product on the right-hand side is complex-valued matrix multiplication.
We are thus motivated to parameterize the convolution kernel
$\phi$ directly in the Fourier domain via $\Phi := \mathcal{F}(\phi)$.
In particular, we have $\cF(\phi^\top_{-} \ast \phi \ast v) = \cF(\phi^\top_{-}) \cdot \cF(\phi) \cdot \cF(v)$, where the dot symbol denotes pointwise product of functions; and
\begin{equation*}
    \begin{aligned}
        \cF\left({\phi^\top_{-}}\right)(k) = \int_{\bT^d} \phi(-x)^\top e^{- i \langle k, x \rangle} \, dx = \int_{\bT^d} \phi(-x)^\top \overline{e^{- i \langle k, -x \rangle}} \, dx 
        = \int_{\bT^d} \overline{\phi(x)^\top e^{- i \langle k, x \rangle}} \, dx
        = \Phi(k)^*, \quad \forall k \in \bZ^d,
    \end{aligned}
\end{equation*}
where the star in the last equality denotes the Hermitian transpose.
Substituting this into \eqref{eq:qsk-def-pre} and denoting $\tilde{u} = \mathcal{E}(u)$, we obtain the following explicit expression for $\cQ_S^k$:
\begin{equation}\label{eq:qsk-def}
    \begin{aligned}
        \cQ_S^k(u)(x) = \left. P^\top V^\top V P u(x) + P^\top \cF^{-1} \bigg( \Phi^* \cdot \Phi \cdot \cF \left(P \tilde{u}(\cdot) \right)\bigg) \right|_{\Omega} (x) + \gamma u(x), \quad \forall u \in L^2(\Omega), \quad \text{for~a.e.~} x \in \Omega,
    \end{aligned}
\end{equation}
where $P$, $V$, and $\Phi$ are trainable parameters.
We illustrate the architecture of $\cQ_S^k$ in \Cref{fig:qsk}.

In implementation, we first discretize the torus $\bT^d$ as a finite set $\mathcal{D}$ with cardinality $|\mathcal{D}|$.
The input $u \in L^2(\Omega)$ is extended by zero to a function in $L^2(\bT^d)$, which is then represented as a vector in $\bR^{|\mathcal{D}|}$ by its values on $\mathcal{D}$.
We adopt a finite-dimensional parameterization of $\Phi$ by truncating it to a maximal number of modes $k_\text{max}$ in each coordinate, i.e., $\Phi(k) = 0$ for all $k \in \bZ^d$ with $\|k\|_\infty > k_\text{max}$.
To ensure that $\cQ_S^k u$ is real-valued, we impose the constraint $\Phi(k) = \overline{\Phi(-k)}$ for all $k \in \bZ^d$.
Hence, we directly parameterize $\Phi$ as a complex-valued tensor in $\mathbb{C}^{( (2k_{\text{max}} + 1)^d + 1 ) / 2 \times m_p \times m_p}$.
The Fourier transforms in \eqref{eq:qsk-def} are correspondingly implemented as discrete Fourier transforms.
In particular, the FFT is applicable when $\mathcal{D}$ is a uniform grid, which reduces computational complexity; see e.g., \cite{li2021fourier}.

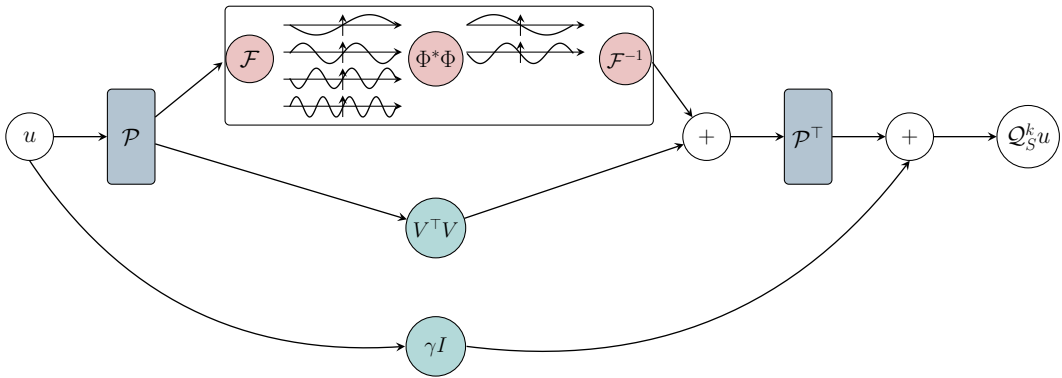
\begin{figure}[!ht]
    \centering
\begin{tikzpicture}[global scale=0.45]
\centering
\node at(-14,0.7)   (2) [circle,draw=black,minimum width =40pt, minimum height =40pt,font=\fontsize{20}{20}\selectfont]{$u$};

\node at (-11, 0.7)  (1) [rounded corners=0.6mm, minimum width =40pt, minimum height =80pt, inner sep=3pt,draw=black,font=\fontsize{20}{20}\selectfont,fill=sintefblue!30]  {$\mathcal{P}$};

\node at(-2,-2)   (3) [circle,draw=black,minimum width =30pt, minimum height =30pt,font=\fontsize{18}{18}\selectfont,fill=teal!30]{$V^\top V$};

\node at(6,0.7)   (4) [circle,draw=black,minimum width =40pt, minimum height =40pt,font=\fontsize{20}{20}\selectfont]{$+$};

\node at(15.5,0.7)   (5) [circle,draw=black,minimum width =40pt, minimum height =40pt,font=\fontsize{20}{20}\selectfont]{$\cQ_S^k u$};

\node at(-1.9,2.8)      (6)[rounded corners=0.6mm ,minimum width =360pt, minimum height =100pt, inner sep=5pt,draw=black,font=\fontsize{20}{20}\selectfont]  {};
\node at(-7.5,3)   (15) [circle,draw=black,minimum width =40pt, minimum height =40pt,font=\fontsize{20}{20}\selectfont,fill=sintefred!30]{$\mathcal{F}$};
\node at(-2,3)   (16) [circle,draw=black,minimum width =40pt, minimum height =40pt,font=\fontsize{20}{20}\selectfont,fill=sintefred!30]{$\Phi^* \Phi$};
\node at(3.6,3)   (17) [circle,draw=black,minimum width =40pt, minimum height =40pt,font=\fontsize{18}{18}\selectfont,fill=sintefred!30]{$\mathcal{F}^{-1}$};
\node at (9, 0.7)  (18)  [rounded corners=0.6mm,draw=black,minimum width =40pt, minimum height =80pt,font=\fontsize{20}{20}\selectfont,fill=sintefblue!30]{$\mathcal{P}^\top$};
\node at(12,0.7)   (19) [circle,draw=black,minimum width =40pt, minimum height =40pt,font=\fontsize{20}{20}\selectfont]{$+$};
\node at(-2,-5.5)   (20) [circle,draw=black,minimum width =50pt, minimum height =30pt,font=\fontsize{18}{18}\selectfont,fill=teal!30]{$\gamma I$};

\coordinate (21) at (-13.5,0.7);
\coordinate [left=0.01cm of 20.west] (22);
\coordinate [right=0.01cm of 20.east] (23);

\draw[eaxis,xscale=0.5] (-13,1.6) -- (-6,1.6);
\draw[eaxis,xscale=0.5] (-13,2.4) -- (-6,2.4);
\draw[eaxis,xscale=0.5] (-13,3.2) -- (-6,3.2);
\draw[eaxis,xscale=0.5] (-13,4) -- (-6,4);

\draw[eaxis,xscale=0.5] (-2.2,4) -- (4.8,4);
\draw[eaxis,xscale=0.5] (-2.2,3.2) -- (4.8,3.2);

\draw[eaxis,yscale=0.4] (-4.75,3.2) -- (-4.75,4.8);
\draw[eaxis,yscale=0.4] (-4.75,5.2) -- (-4.75,6.8);
\draw[eaxis,yscale=0.4] (-4.75,7.2) -- (-4.75,8.8);
\draw[eaxis,yscale=0.4] (-4.75,9.2) -- (-4.75,10.8);

\draw[eaxis,yscale=0.4] (0.5,7.2) -- (0.5,8.8);
\draw[eaxis,yscale=0.4] (0.5,9.2) -- (0.5,10.8);

\draw[elegant,black,line width=0.5,domain=-\num:\num,xscale=0.5,yscale=0.3,xshift=-270pt,yshift=380pt] plot(\x,{sin(\x r)});
\draw[elegant,black,line width=0.5,domain=-\num:\num,xscale=0.5,yscale=0.3,xshift=-270pt,yshift=300pt] plot(\x,{sin(2*\x r)});
\draw[elegant,black,line width=0.5,domain=-\num:\num,xscale=0.5,yscale=0.3,xshift=-270pt,yshift=230pt] plot(\x,{sin(3*\x r)});
\draw[elegant,black,line width=0.5,domain=-\num:\num,xscale=0.5,yscale=0.3,xshift=-270pt,yshift=150pt] plot(\x,{sin(4*\x r)});

\draw[elegant,black,line width=0.5,domain=-\num:\num,xscale=0.5,yscale=0.3,xshift=28pt,yshift=380pt] plot(\x,{-sin(\x r)});
\draw[elegant,black,line width=0.5,domain=-\num:\num,xscale=0.5,yscale=0.3,xshift=28pt,yshift=300pt] plot(\x,{-sin(2*\x r)});

\draw[-stealth][line width=0.5pt] (2) --(1);
\draw[-stealth][line width=0.5pt] (1) --(-8.3,2.9);
\draw[-stealth][line width=0.5pt] (4.4,2.9) --(4);
\draw[-stealth][line width=0.5pt] (4) --(18);
\draw[-stealth][line width=0.5pt] (18) --(19);
\draw[-stealth][line width=0.5pt] (1) --(3);
\draw[-stealth][line width=0.5pt] (3) --(4);
\draw[-stealth][line width=0.5pt] (19)--(5);

\draw[-stealth, bend right, line width=0.5pt] (2.south) to (22);
\draw[-stealth, bend right, line width=0.5pt] (23) to (19.south);

\end{tikzpicture}
\caption{The neural network architecture of $\cQ_S^k$.}
\label{fig:qsk}
\end{figure}

\subsection{Neural Network Architecture for \texorpdfstring{$\cQ_A^k$}{QAk}}\label{sec:qa-special}

We focus on the specific case where $N$ and $\theta$ satisfy the following pointwise assumption.

\begin{assumption}\label{assump:ptwise}
    \noindent We make the following assumptions:
    \begin{enumerate}[label=(\roman*)]
    \item The operator $N: L^2(\Omega) \to L^2(\Omega)$ acts as a multiplication operator:
    \begin{equation*}
        N(u)(x) = \lambda(x) u(x) \quad \text{~a.e.~in~} \Omega,
    \end{equation*}
    where $\lambda \in L^\infty(\Omega)$ satisfies $\lambda(x) \geq c_0 > 0$ a.e. in $\Omega$. \label{assump:n-ptwise}
    \item There exist an integer $m \geq 1$, a closed set $E \subset \bR^m$ and a function $\psi_\xi: \bR \to \bR \cup \{+\infty\}$ parameterized by $\xi \in E$, such that $\psi_\xi$ is convex, lower semicontinuous, and proper for all $\xi \in E$.
    The functional $\theta: L^2(\Omega) \to \bR \cup \{ +\infty \}$ is given by
    \begin{equation*}
        \theta(u) = \int_\Omega \psi_{\mu(x)}(u(x)) \, dx, \quad \forall u \in L^2(\Omega),
    \end{equation*}
    for some $\mu \in L^\infty(\Omega; \bR^m)$ satisfying $\mu(x) \in E$ a.e. in $\Omega$.
    Here, the integral is interpreted as $+\infty$ if the integrand equals $+\infty$ over a subset of $\Omega$ with positive measure. \label{assump:theta-ptwise}
    \end{enumerate}
\end{assumption}

Though formulated abstractly, \Cref{assump:ptwise} covers many practical scenarios.
Notable examples include:
\begin{enumerate}[label=(\roman*)]
    \item \emph{Control-constrained optimal control problems}: $N = \lambda I$ for some $\lambda > 0$ and $\theta(u) = I_{U_{ad}}(u)$, where
    \begin{equation}\label{eq:uad-def}
        U_{ad} = \{ u \in L^2(\Omega) \mid u_a(x) \leq u(x) \leq u_b(x) ~\text{a.e.~in} ~\Omega \}
    \end{equation}
    for some $u_a, u_b \in L^\infty(\Omega)$, and $I_{U_{ad}}(\cdot)$ is the indicator functional of the set $U_{ad}$.
    In this case, \Cref{assump:ptwise} \ref{assump:n-ptwise} holds with $\lambda(x) = \lambda$, and \Cref{assump:ptwise} \ref{assump:theta-ptwise} holds with 
    \begin{equation*}
        \begin{aligned}
            m = 2, \quad E = \{\xi := (\xi_1, \xi_2) \in \bR^2 \, : \, \xi_1 \leq \xi_2\}, \quad
            \psi_\xi(r) = I_{[\xi_1, \xi_2]}(r), \quad \mu(x) = (u_a(x), u_b(x))^\top.
        \end{aligned}
    \end{equation*}
    \item \emph{Sparse optimal control problems}: $N = \lambda I$ for some $\lambda > 0$ and $\theta(u) = \beta \| u \|_{L^1(\Omega)} + I_{U_{ad}}(u)$ for some $\beta \geq 0$ and $U_{ad}$ defined in \eqref{eq:uad-def}.
    In this case, \Cref{assump:ptwise} \ref{assump:n-ptwise} holds with $\lambda(x) = \lambda$, and \Cref{assump:ptwise} \ref{assump:theta-ptwise} holds with
    \begin{equation*}
        \begin{aligned}
            &m = 3, \quad E = \{\xi := (\xi_1, \xi_2, \xi_3) \in \bR^3 \, : \, \xi_1 \leq \xi_2, \, \xi_3 \geq 0\}, \\
            &\psi_\xi(r) = \xi_3 |r| + I_{[\xi_1, \xi_2]}(r), \quad \mu(x) = (u_a(x), u_b(x), \beta)^\top.
        \end{aligned}
    \end{equation*}
\end{enumerate}
In these cases, it suffices to apply a pointwise operator to approximate $(N + \tau I + \partial \theta)^{-1}$; this observation will be generalized in \Cref{lem:nresol-ptwise}.
To motivate a reasonable network architecture, denoting $\sigma$ as the ReLU activation function, we observe that
\begin{equation}\label{eq:prox-oper-expand-ind}
    \begin{aligned}
        \left((\lambda + \tau) I + \partial I_{U_{ad}}\right)^{-1}(u)(x) &= \mathrm{proj}_{[u_a(x), u_b(x)]} \left( \frac{1}{\lambda + \tau} u(x) \right) \\
        &= u_b(x) - \sigma \left( u_b(x) - u_a(x) - \sigma \left(\frac{1}{\lambda + \tau} u(x) - u_a(x) \right) \right),
    \end{aligned}
\end{equation}
\begin{equation}\label{eq:prox-oper-expand-ind-plus-l1}
    \begin{aligned}
        &\left((\lambda + \tau) I + \partial \left( \beta \|\cdot\|_{L^1(\Omega)} + I_{U_{ad}} \right) \right)^{-1}(u)(x) = \mathrm{proj}_{[u_a(x), u_b(x)]} \left( \left[ \frac{|u(x)| - \beta}{\lambda + \tau} \right]_+ \mathrm{sgn} \left( \frac{1}{\lambda + \tau} u(x) \right) \right)  \\
        &\quad =  u_b(x) - \sigma \left( u_b(x) - u_a(x) - \sigma \left( \sigma\left(\frac{u(x) - \beta}{\lambda + \tau}\right) - \sigma \left(\frac{-u(x) - \beta}{\lambda + \tau} \right) - u_a(x) \right) \right) \\
        &\quad = u_b(x) - \sigma \left( u_b(x) - u_a(x) - \sigma \left( \begin{pmatrix}
            1 & -1
        \end{pmatrix} \sigma \left( \frac{1}{\lambda + \tau} \begin{pmatrix}
            1 \\ -1
        \end{pmatrix} u(x) - \frac{\beta}{\lambda + \tau} \begin{pmatrix}
            1 \\ 1
        \end{pmatrix} \right) - u_a(x) \right) \right),
    \end{aligned}
\end{equation}
both of which can be parameterized as several layers of neural networks with input $(u(x), \mu(x))^\top$, ReLU activation function, and concatenated skip connections \cite{huang2017densely}.
Generalizing this observation to the cases where $\lambda(x)$ is non-constant, we propose the following neural network architecture for $\cQ_A^k$:
\begin{equation}\label{eq:qak-ptwise}
    \left(\cQ_A^k (u)\right)(x) = \mathcal{N}(u(x), \mu(x), \lambda(x)) \quad \text{~a.e.~in}~\Omega,
\end{equation}
where $\mathcal{N}: \bR^{m+2} \to \bR$ is a neural network defined by
\begin{equation}\label{eq:nn-rec-structure}
    \left\{
    \begin{aligned}
        &\mathcal{N}(r, \xi, \eta) = W_L \cdot (v^{(L-1)}, r, \xi, \eta)^\top + b_L, \\
        &v^{(l)} = \sigma \left( W_l \cdot (v^{(l-1)}, r, \xi, \eta)^\top + b_l \right), \quad l = 1, 2, \cdots, L-1, \\
        &v^{(0)} = W_0 \cdot (r, r, \xi, \eta)^\top.
    \end{aligned}
    \right.
\end{equation}
Here, $r \in \bR$, $\xi \in \bR^m$, $\eta \in \bR$, $L$ is the number of layers, $W_l$ and $b_l$ are the weight matrices and bias vectors of the $l$-th layer, respectively, and $\sigma$ is the ReLU activation function.
We allow the vector $(r, \xi, \eta)^\top$ to be fed into each hidden layer, thereby facilitating concatenated skip connections from the input to every hidden layer.
The architecture of $\cQ_A^k$ is illustrated in \Cref{fig:qak-ptwise}.

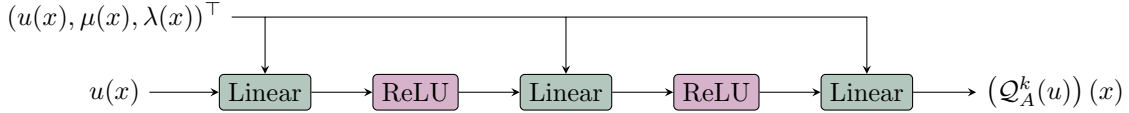
\begin{figure}[!ht]
    \centering
    \begin{tikzpicture}
        \node   (u)         {$u(x)$};
        \node   (linear1)   at ([xshift=2cm]u)     [draw, rounded corners=0.6mm, fill=sintefdarkgreen!30]  {Linear};
        \node   (relu1)    at ([xshift=2cm]linear1)     [draw, rounded corners=0.6mm, fill=sinteflilla!30]     {ReLU};
        \node   (linear2) at ([xshift=2cm]relu1)   [draw, rounded corners=0.6mm, fill=sintefdarkgreen!30]  {Linear};
        \node   (relu2)    at ([xshift=2cm]linear2)     [draw, rounded corners=0.6mm,,fill=sinteflilla!30]     {ReLU};
        \node   (linear3) at ([xshift=2cm]relu2)   [draw, rounded corners=0.6mm,fill=sintefdarkgreen!30]  {Linear};
        \node   (v)         at ([xshift=2.5cm]linear3)      {$\left(\cQ_A^k (u)\right)(x)$};
        \node   [align=center]   (uab)       at ([yshift=1cm]u)     {$(u(x), \mu(x), \lambda(x))^\top$};

        \coordinate (coord1)  at ([yshift=1cm]linear1);
        \coordinate (coord2)  at ([yshift=1cm]linear2);

        \draw[-stealth]   (u) -- (linear1);
        \draw[-stealth]   (linear1) -- (relu1);
        \draw[-stealth]   (relu1) -- (linear2);
        \draw[-stealth]   (linear2) -- (relu2);
        \draw[-stealth]   (relu2) -- (linear3);
        \draw[-stealth]   (linear3) -- (v);
        \draw[-stealth]   (uab) -| (linear1);
        \draw[-stealth]   (coord1) -| (linear2);
        \draw[-stealth]   (coord2) -| (linear3);
    \end{tikzpicture}
    \caption{The neural network architecture of $\cQ_A^k$.}
    \label{fig:qak-ptwise}
\end{figure}

In implementation, if $\lambda(x)$ or $\mu(x)$ is constant over $\Omega$, we simplify $\cQ_A^k$ by omitting the constant parameters from the input. 
For instance, if $\lambda$ is constant, the reduced architecture reads as
\begin{equation}\label{eq:qak-ptwise-reduced}
    \left(\cQ_A^k (u)\right)(x) = \mathcal{N}(u(x), \mu(x)), \quad \text{~a.e.~} x \in \Omega,
\end{equation}
where $\mathcal{N}: \bR^{m+1} \to \bR$ is a neural network defined by
\begin{equation}\label{eq:nn-rec-structure-reduced}
    \left\{
    \begin{aligned}
        &\mathcal{N}(r, \xi) = W_L \cdot (v^{(L-1)}, r, \xi)^\top + b_L, \\
        &v^{(l)} = \sigma \left( W_l \cdot (v^{(l-1)}, r, \xi)^\top + b_l \right), \quad l = 1, 2, \cdots, L-1, \\
        &v^{(0)} = W_0 \cdot (r, r, \xi)^\top,
    \end{aligned}
    \right.
\end{equation}
for some weights $\{W_\ell\}_{\ell=0}^{L}$ and biases $\{b_\ell\}_{\ell=1}^L$.

\subsection{Full Architecture and Training Framework of the iUzawa-Net}

With the constructions of $\cS^k$, $\cA^k$, $\cQ_S^k$, and $\cQ_A^k$ detailed in Sections \ref{se:Ak}-\ref{sec:qa-special}, the \texttt{iUzawa-Net} defined in \eqref{eq:inexact-uzawa-duf} becomes fully implementable. These neural modules collectively specify all learnable parameters in the \texttt{iUzawa-Net}, replacing the preconditioners and PDE solution operators required by the inexact Uzawa iteration \eqref{eq:inexact-uzawa-qa-spec}. Our design ensures that each layer of the \texttt{iUzawa-Net} is optimization-informed and structurally mirrors one iteration of  \eqref{eq:inexact-uzawa-qa-spec}.
A schematic illustration of the \texttt{iUzawa-Net}, highlighting the interaction among $\cS^k$, $\cA^k$, $\cQ_S^k$, and $\cQ_A^k$ within each layer, is provided in Figure~\ref{fig:iuzawa-net}.
The terms $u^0$ and $p^0$ correspond to the initialization of \eqref{eq:inexact-uzawa-qa-spec}.
For simplicity, we fix $u^0 = 0$ and $p^0 = 0$ throughout this work.
 
\begin{figure}[!ht]
	\centering
	\vspace{0.5em}
	\begin{subfigure}{\textwidth}
		\centering
		\begin{tikzpicture}
			\node  (u0)                        {$u_0$};
			\node  (p0)       [below=0.16cm of u0]   {$p_0$};
			\coordinate(u0_p0_mid) at ($(u0)!0.5!(p0)$);
			\node  (layer1)   [draw, rounded corners=0.6mm, minimum width=1cm, minimum height=1.0cm, right=0.7cm of u0_p0_mid,fill=teal!30]   {Layer $1$};
            \coordinate[below=0.6cm of layer1.south] (layer1_below);
            \node  (f) [left=2.8cm of layer1_below] {$f$};
            \node  (yd) [above=0.2cm of f] {$y_d$};
            \coordinate[right=0.5cm of f] (f_right);
			\path let \p1=($(layer1.east) + (1cm, 0)$), \p2=(u0) in node (u1) at (\x1, \y2) {$u^1$};
			\path let \p1=($(layer1.east) + (1cm, 0)$), \p2=(p0) in node (p1) at (\x1, \y2) {$p^1$};
			\coordinate(u1_p1_mid) at ($(u1)!0.5!(p1)$);
			\node  (layer2)   [draw, rounded corners=0.6mm, minimum width=1cm, minimum height=1.0cm, right=of u1_p1_mid,fill=teal!30]   {Layer $2$};
			\path let \p1=($(layer2.east) + (1cm, 0)$), \p2=(u1) in node (u2) at (\x1, \y2) {$u^2$};
			\path let \p1=($(layer2.east) + (1cm, 0)$), \p2=(p1) in node (p2) at (\x1, \y2) {$p^2$};
			\node  (dots1)    [right=0.5cm of u2]   {$\cdots$};
			\node  (dots2)    [right=0.5cm of p2]   {$\cdots$};
			\node  (uk-1)       [right=0.5cm of dots1]   {$u^{L-1}$};
			\node  (pk-1)       [right=0.5cm of dots2]   {$p^{L-1}$};
			\coordinate(uk-1_pk-1_mid) at ($(uk-1)!0.5!(pk-1)$);
			\node  (layerk)   [draw, rounded corners=0.6mm, minimum width=1cm, minimum height=1.0cm, right=of uk-1_pk-1_mid,fill=teal!30]   {Layer $L$};
			\path let \p1=($(layerk.east) + (1.5cm, 0)$), \p2=(uk-1) in node (uk) at (\x1, \y2) {$u^L$};
			\path let \p1=($(layerk.east) + (1.5cm, 0)$), \p2=(pk-1) in node (pk) at (\x1, \y2) {$p^L$};
			\node   (lammu)     at ([yshift=1.2cm,xshift=-1.0cm]u0) {$\lambda, \, \mu$};
			\node   (dots3)  at ([yshift=1.2cm]dots1) {$\cdots$};
            \coordinate[below=0.6cm of layer2.south] (layer2_below);
            \coordinate[below=0.6cm of layerk.south] (layerk_below);
            \node   (dots4) at ([yshift=-1.428cm]dots1) {$\cdots$};
			
			\draw[-stealth] (u0) -- (u0 -| layer1.west);
			\draw[-stealth] (p0) -- (p0 -| layer1.west);
			\draw[-stealth] (u1 -| layer1.east) -- (u1);
			\draw[-stealth] (p1 -| layer1.east) -- (p1);
			\draw[-stealth] (u1) -- (u1 -| layer2.west);
			\draw[-stealth] (p1) -- (p1 -| layer2.west);
			\draw[-stealth] (u2 -| layer2.east) -- (u2);
			\draw[-stealth] (p2 -| layer2.east) -- (p2);
			\draw[-stealth] (u2) -- (dots1);
			\draw[-stealth] (p2) -- (dots2);
			\draw[-stealth] (dots1) -- (uk-1);
			\draw[-stealth] (dots2) -- (pk-1);
			\draw[-stealth] (uk-1) -- (uk-1 -| layerk.west);
			\draw[-stealth] (pk-1) -- (pk-1 -| layerk.west);
			\draw[-stealth] (uk -| layerk.east) -- (uk);
			\draw[-stealth] (pk -| layerk.east) -- (pk);
			\draw[-stealth, dashed] (lammu) -| (layer1.north);
			\draw[-stealth, dashed] (lammu) -| (layer2.north);
			\draw[-stealth, dashed] (lammu) -- (dots3);
			\draw[-stealth, dashed] (dots3) -| (layerk.north);
            \draw[-stealth] (f.east) -| (layer1);
            \draw[-] (yd) -| (f_right);
            \draw[-stealth] (layer1_below) -| (layer2);
            \draw[-stealth] (layer2_below) -- (dots4);
            \draw[-stealth] (dots4) -| (layerk);

            \draw[line width=0.4mm, dotted, rounded corners=10pt] ($(layer1.north west) + (-1.2cm, 0.5cm)$) rectangle ($(layerk.south east) + (0.5cm, -1.0cm)$);
		\end{tikzpicture}
		\subcaption{An overview of the \texttt{iUzawa-Net}.}
	\end{subfigure}

\par\vspace{0.5em}

\begin{subfigure}{\textwidth}
    \centering
    \begin{tikzpicture}[
        block/.style={draw, rounded corners=0.6mm, minimum width=1cm},
        primal/.style={fill=sintefblue!30},
        dual/.style={fill=sintefred!30},
        mainarrow/.style={-stealth, line width=0.4mm}
    ]
        \node (uk) {$u^k$};
        \node (pk) [below=0.5cm of uk] {$p^k$};
        \node (cA) [block, primal, right=of pk] {$-\cA^k$};
        \node (delta) [block, primal] at (uk -| cA) {$\tau I$};
        \node (sum_u1) [draw, circle, primal, right=of cA] {$+$};
        \node (Qa) [block, primal, right=of sum_u1] {$\cQ_A^k$};
        \node (lammu) [above=1.8cm of Qa] {$\lambda,\,\mu$};

        \node (cS) [block, dual, below=1.6cm of cA] {$\cS^k$};
        \node (sum_S) [draw, circle, dual, above=0.35cm of cS] {$+$};
        \node (sum_p1) [draw, circle, dual] at (cS -| sum_u1) {$+$};
        \node (Qs) [block, dual] at (sum_p1 -| Qa) {$\cQ_S^k$};
        \node (sum_p2) [draw, circle, dual, right=of Qs] {$+$};
        \node (pk1) [right=of sum_p2] {$p^{k+1}$};
        \node (uk1) at (Qa -| pk1) {$u^{k+1}$};

        \node (sum_p3) [draw, circle, dual, below=0.5cm of sum_p1] {$+$};
        \node (yd) [left=4cm of sum_p3] {$y_d$};
        \node (f) at (yd |- sum_S) {$f$};
        \node (label) [above=0.5cm of uk1] {Layer $k$};

        \coordinate (p_bus) at ($(sum_p3.south)+(0,-0.6cm)$);

        \draw[mainarrow] (pk) -- (cA);
        \draw[mainarrow] (cA) -- (sum_u1);
        \draw[-stealth] (uk) -- (delta);
        \draw[-stealth] (delta.east) -| (sum_u1.north);
        \draw[mainarrow] (sum_u1) -- (Qa);
        \draw[mainarrow] (Qa) -- (uk1);

        \draw[mainarrow] (uk1.south)
            -- (uk1.south |- sum_S.east) -- (sum_S.east);
        \draw[-stealth] (f.east) -- (sum_S.west);
        \draw[mainarrow] (sum_S.south) -- (cS.north);
        \draw[mainarrow] (cS.east) -- (sum_p1.west);

        \draw[-stealth] (pk.south)
            -- (pk.south |- p_bus) -| (sum_p3.south);
        \draw[-stealth] (yd) -- (sum_p3);
        \draw[-stealth] (sum_p3) -- node[right] {$-$} (sum_p1);
        \draw[mainarrow] (sum_p1) -- (Qs);
        \draw[mainarrow] (Qs) -- (sum_p2);
        \draw[mainarrow] (sum_p2) -- (pk1);
        \draw[-stealth] (pk.south)
            -- (pk.south |- p_bus) -| (sum_p2.south);

        \draw[-stealth] ($(uk.west)+(-0.7cm,0)$) -- (uk);
        \draw[-stealth] ($(pk.west)+(-0.7cm,0)$) -- (pk);
        \draw[-stealth] (uk1) -- ($(uk1.east)+(1.0cm,0)$);
        \draw[-stealth] (pk1) -- ($(pk1.east)+(1.0cm,0)$);
        \draw[-stealth, dashed] (lammu) -- (Qa);

        \draw[line width=0.4mm, dotted, rounded corners=10pt]
            ($(cA.north west)+(-0.5cm,1.5cm)$)
            rectangle ($(pk1.south east |- p_bus)+(0.3cm,-0.3cm)$);
    \end{tikzpicture}
    \subcaption{A single layer of the \texttt{iUzawa-Net}.}
\end{subfigure}

\caption{An overview of the \texttt{iUzawa-Net} \eqref{eq:inexact-uzawa-duf}.}
\label{fig:iuzawa-net}
\end{figure}

To train the \texttt{iUzawa-Net}, we generate a dataset $\{((y_d)_i, f_i)^\top, u^*_i\}_{i=1}^N$, where $(y_d)_i, f_i$ are the parameters of \eqref{eq:opt-ctrl} and $u^*_i$ is the corresponding optimal control. 
Each $u_i^*$ can be obtained from real-world observations or computed using traditional numerical solvers or deep learning-based methods.
Motivated by \cite{kovachki2023neural,li2023fourier,wen2022u}, the \texttt{iUzawa-Net} is trained by minimizing the following relative $L^2$-error loss:
\begin{equation}\label{eq:relative-l2-err-def}
    \mathcal{L}(\theta_{\cT}) = \frac{1}{N} \sum_{i=1}^N \frac{ \left\| \mathcal{T}\left((y_d)_i, f_i; \theta_{\mathcal{T}} \right) - u^*_i \right\|_{L^2(\Omega)}} { \max \left\{ \left\| u^*_i \right\|_{L^2(\Omega)}, \, \varepsilon_L \right\}},
\end{equation}
where $\varepsilon_L > 0$ is a small constant to ensure stability.
In implementation, the $L^2$-norms in \eqref{eq:relative-l2-err-def} are approximated using a discretized representation of the functions over the sampled points.
This end-to-end training approach facilitates a data-driven preconditioner training mechanism and enables real-time inference of the optimal control for new problem parameters $y_d$ and $f$, effectively compensating for the initial computational cost in data generation and training.

\section{Universal Approximation of the iUzawa-Net}\label{sec:universal-approx}

In this section, we prove that the \texttt{iUzawa-Net} \eqref{eq:inexact-uzawa-duf} admits the universal approximation property for the class of solution operators of \eqref{eq:opt-ctrl} under \Cref{assump:ptwise}. 
Specifically, we show that 
the solution operator $T$ defined in \eqref{eq:param-to-sol-oper} can be approximated to arbitrary accuracy on any compact set by an \texttt{iUzawa-Net} with two layers.
For simplicity, we assume that $\Omega$ is a spatial domain, and remark that the subsequent analysis extends directly to spatio-temporal domains.

\subsection{Preliminaries}
In this subsection, we present some preliminary results related to \eqref{eq:opt-ctrl} and its optimality system \eqref{eq:opt-ctrl-opt-cond}, which will be utilized in the subsequent analysis.
Recall that $U = Y = L^2(\Omega)$.
Throughout \Cref{sec:universal-approx,sec:stability-analysis}, we retain $U$ and $Y$ to distinguish the control and state spaces, and use $L^2(\Omega)$ when its specific function-space properties are relevant.

\begin{lemma}\label{lem:nresolv-lip}
    Let $N: U \to U$ and $\theta: U \to \bR \cup \{ +\infty \}$ be defined as in \eqref{eq:opt-ctrl}.
    Then the operator $(N + \partial \theta)^{-1}: U \rightrightarrows U$ is single-valued.
    Furthermore, for any $v_1, v_2 \in U$, let $u_1 = (N + \partial \theta)^{-1}(v_1)$ and $u_2 = (N + \partial \theta)^{-1}(v_2)$.
    Then
    \begin{equation}\label{eq:nfirm-nonexp}
        \left\langle v_1 - v_2, u_1 - u_2 \right\rangle_U \geq c_0 \left\| u_1 - u_2 \right\|_U^2,
    \end{equation}
    where $c_0 > 0$ is the coercivity constant of $N$.
    Consequently, $(N + \partial \theta)^{-1}$ is Lipschitz continuous with Lipschitz constant $c_0^{-1}$.
\end{lemma}

\begin{proof}
    Since $\theta$ is convex, lower semicontinuous, and proper, and $N$ is coercive, it follows from \cite[Lemma 1]{eckstein1993nonlinear} and \cite[Theorem 2.3]{bauschke2008general} that $(N + \partial \theta)^{-1}$ is single-valued.
    Moreover, the operator $\partial \theta$ is maximally monotone \cite[Theorem 20.25]{bauschke2017convex}.
    Hence, for any $v_1, v_2 \in U$, if $u_1 = (N + \partial \theta)^{-1}(v_1)$, and $u_2 = (N + \partial \theta)^{-1}(v_2)$, the monotonicity of $\partial \theta$ implies
    \begin{equation}\label{eq:nresolv-nfirm-nonexp}
        \left\langle v_1- v_2, u_1 - u_2 \right\rangle _U \geq \left\langle N(u_1 - u_2), u_1 - u_2 \right\rangle _U.
    \end{equation}
    Applying the coercivity of $N$ to the right hand side and the Cauchy-Schwarz inequality to the left hand side, we obtain
    \begin{equation*}
        \left\|v_1 - v_2\right\|_U \left\|u_1 - u_2 \right\|_U \geq \left\langle v_1 - v_2, u_1 - u_2 \right\rangle _U \geq \left\langle N(u_1 - u_2), u_1 - u_2 \right\rangle _U \geq c_0 \| u_1 - u_2 \|_U^2.
    \end{equation*}
    This establishes \eqref{eq:nfirm-nonexp}.
    The Lipschitz continuity follows immediately by dividing both sides by $c_0 \left\|u_1 - u_2 \right\|_U$ (assuming $u_1 \neq u_2$, as the case $u_1 = u_2$ is trivial).
\end{proof}

\begin{lemma}\label{lem:param-to-ctrl-lip}
    The solution operator $T$ of \eqref{eq:opt-ctrl} is a mapping of $Sf - y_d$, i.e.,
    \begin{equation*}
        T(y_d, f) = T_0(Sf - y_d)
    \end{equation*}
    for some operator $T_0: Y \to U$.
    Moreover, $T_0$ is Lipschitz continuous with Lipschitz constant $c_0^{-1} \| S \|$.
\end{lemma}

\begin{proof}
    The existence of $T_0$ is a direct consequence of the optimality condition \eqref{eq:opt-ctrl-opt-cond}, whose solution $(u^*, p^*)^\top$ depends on $(y_d, f)$ only through $Sf - y_d$.
    Let $z_1, z_2 \in Y$ and denote $u_i^* = T_0(z_i)$ for $i=1,2$.
    By \eqref{eq:opt-ctrl-opt-cond}, there exists $p_i^* \in Y$ such that
    \begin{equation}\label{eq:lptcl-pf-aux-1}
        \left\{
        \begin{aligned}
            & 0 \in Nu_i^* + \partial \theta(u_i^*) + S^* p_i^*, \\
            & 0 = Su_i^* - p_i^* + z_i. 
        \end{aligned}
        \right.
    \end{equation}
    By \Cref{lem:nresolv-lip}, $(N+\partial \theta)^{-1}$ is single-valued.
    Hence \eqref{eq:lptcl-pf-aux-1} is equivalent to
    \begin{equation}\label{eq:lptcl-pf-aux-2}
        (S \circ (N+\partial \theta)^{-1} \circ S^* + I)(-p_i^*) + z_i = 0,
    \end{equation}
    where $I: Y \to Y$ denotes the identity operator.
    It follows from \eqref{eq:lptcl-pf-aux-2} and \eqref{eq:nfirm-nonexp} that
    \begin{equation*}
        \begin{aligned}
            &\left\langle z_1 - z_2, p_1^* - p_2^* \right\rangle _Y \\
            = &\left\| p_1^* - p_2^* \right\|_Y^2 - \left\langle (S \circ (N+\partial \theta)^{-1} \circ S^*)(-p_1^*) - (S \circ (N+\partial \theta)^{-1} \circ S^*)(-p_2^*), \ p_1^* - p_2^* \right\rangle _Y \\
            = &\left\| p_1^* - p_2^* \right\|_Y^2 + \left\langle (N+\partial \theta)^{-1} (-S^* p_1^*) - (N+\partial \theta)^{-1}(-S^* p_2^*), \ -S^*(p_1^* - p_2^*) \right\rangle_U \\
            \geq &\left\| p_1^* - p_2^* \right\|_Y^2 + c_0 \left\| (N+\partial \theta)^{-1} (-S^* p_1^*) - (N+\partial \theta)^{-1}(-S^* p_2^*) \right\|_U^2 \\
            \geq &\left\| p_1^* - p_2^* \right\|_Y^2.
        \end{aligned}
    \end{equation*}
    Hence $\left\| p_1^* - p_2^* \right\|_Y \leq \left\| z_1 - z_2 \right\|_Y$.
    Applying \Cref{lem:nresolv-lip} again yields
    \begin{equation*}
        \begin{aligned}
            \left\| u_1^* - u_2^* \right\|_U &= \| (N+\partial \theta)^{-1}(-S^* p_1^*) - (N+\partial \theta)^{-1}(-S^* p_2^*) \|_U \\
            &\leq c_0^{-1} \| S^* \| \| p_1^* - p_2^* \|_Y \\
            &\leq c_0^{-1} \| S^* \| \| z_1 - z_2 \|_Y.
        \end{aligned}
    \end{equation*}
\end{proof}

\begin{remark}
    Though $T$ is a mapping of $Sf - y_d$, we insist on designing the \texttt{iUzawa-Net} as an operator with input $(y_d, f)^\top$ rather than $Sf - y_d$, since evaluating $Sf$ requires an additional PDE solve.
\end{remark}

We now state, without proof, the following universal approximation theorems for fully-connected neural networks (FCNNs) and FNOs.
Note that \Cref{prop:fc-univ-approx} follows directly from \cite[Theorem 3.1]{pinkus1999approximation}, the Tietze extension theorem, and the fact that uniform convergence of an $\bR^n$-valued function is equivalent to uniform convergence in each of its coordinates.
The compactness of the image in \Cref{prop:fno-univ-approx-general} is guaranteed by the fact that the mapping $v \mapsto (\mathcal{G} \circ \mathcal{E}(v))|_{\Omega}$ is continuous from $H^{s_1}(\Omega)$ to $H^{s_2}(\Omega)$, which holds by noticing the following: the extension operator $\mathcal{E}$ in \Cref{prop:fno-univ-approx-general} and the restriction operator $H^{s_2}(\bT^d) \to H^{s_2}(\Omega)$ are continuous; and the FNO $\mathcal{G}$ given by \cite[Theorem 9]{kovachki2021universal} is constructed via \cite[Theorem 5]{kovachki2021universal}, and hence is continuous as an operator from $H^{s_1}(\bT^d)$ to $H^{s_2}(\bT^d)$.

\begin{proposition}[cf. {\cite[Theorem 3.1]{pinkus1999approximation}}]\label{prop:fc-univ-approx}
    Let $F \subset \bR^m$ be a closed set and $g : F \to \bR^n$ be a continuous function.
    If $\sigma: \bR \to \bR$ is continuous and not a polynomial, then for any compact set $K \subset F$ and any $\varepsilon > 0$, there exists a fully connected neural network $\mathcal{N}_{\text{FC}}: \bR^m \to \bR^n$ with activation function $\sigma$, such that
    \begin{equation*}
        \sup_{x \in K} \left\| g(x) - \mathcal{N}_{\text{FC}}(x) \right\| < \varepsilon.
    \end{equation*}
\end{proposition}

\begin{proposition}[cf. {\cite[Theorem 9]{kovachki2021universal}}]\label{prop:fno-univ-approx-general}
    Let $s_1, s_2 \geq 0$ and $\Omega \subset \bR^d$ be a bounded domain with Lipschitz boundary satisfying $\overline{\Omega} \subset (0, 2\pi)^d$.
    Let $G: H^{s_1}(\Omega) \to H^{s_2}(\Omega)$ be a continuous operator and $K \subset H^{s_1}(\Omega)$ be a compact set.
    Then for any $\varepsilon > 0$, there exist a continuous embedding $\mathcal{E}: H^{s_1}(\Omega) \to H^{s_1}(\bT^d)$, possibly not the zero extension, and an FNO $\mathcal{G}: H^{s_1}(\bT^d) \to H^{s_2}(\bT^d)$, such that
    \begin{equation*}
        \sup_{v \in K} \left\| G(v) - \left(\mathcal{G} \circ \mathcal{E} (v)\right)|_{\Omega} \right\|_{H^{s_2}(\Omega)} < \varepsilon.
    \end{equation*}
    Furthermore, if $s_2 \leq s_1$, the image $(\mathcal{G} \circ \cE)(K)|_{\Omega}$ is compact in $H^{s_2}(\Omega)$.
    If $s_1 = s_2 = 0$, the embedding $\mathcal{E}$ can be chosen as the zero extension.
\end{proposition}

As a direct consequence of \Cref{prop:fc-univ-approx}, the universal approximation property also holds for the neural network architecture defined in \eqref{eq:nn-rec-structure}, since the architecture in \eqref{eq:nn-rec-structure} encompasses standard FCNNs (e.g., by setting the weights of the concatenated skip connections to zero).

\begin{lemma}\label{lem:nn-rec-univ-approx}
    Let $m$ be any positive integer, $F \subset \bR^{m+2}$ be a closed set, $g: F \to \bR$ be a continuous function, and $\sigma(x) = \max\{0, x\}$ be the ReLU activation function.
    Then for any compact set $K \subset F$ and any $\varepsilon > 0$, there exists a neural network $\mathcal{N}: \bR^{m+2} \to \bR$ with the architecture defined in \eqref{eq:nn-rec-structure}, such that
    \begin{equation*}
        \sup_{(r, \xi, \eta)^\top \in K} \left\| g(r, \xi, \eta) - \mathcal{N}(r, \xi, \eta) \right\| < \varepsilon.
    \end{equation*}
\end{lemma}

For the neural network $\cQ_S^k$ constructed in \eqref{eq:qsk-def}, we establish the following properties regarding its approximation capabilities and regularity.

\begin{proposition}\label{prop:qsk-univ-approx}
    The following assertions hold:
    \begin{enumerate}[label=(\roman*)]
        \item For any $\gamma > 0$ and $\tilde{\gamma} \geq \gamma$, there exists a neural network $\cQ_S$ of the form \eqref{eq:qsk-def} such that $\cQ_S = \tilde{\gamma} I$.
        \item For any linear, bounded, self-adjoint, and positive definite operator $Q_S$, there exists $\cQ_S$ of the architecture \eqref{eq:qsk-def} with nonzero $P$, $V$, $\Phi$, and $\gamma > 0$ such that $\cQ_S$ is self-adjoint and positive definite, and the difference $\cQ_S^{-1} - Q_S$ is linear, bounded, self-adjoint, and positive semidefinite. Moreover, $\cQ_S$ maps $H^s(\Omega)$ continuously into $H^s(\Omega)$ for any $s \geq 0$.
    \end{enumerate}
\end{proposition}

\begin{proof}
    To prove (i), let $P = e_1 \in \bR^{m_p}$ be the first standard basis vector, $V = \sqrt{\tilde{\gamma}-\gamma} I$, and $\Phi = 0$.
    Then for any $u \in L^2(\Omega)$,
    \begin{equation*}
        \cQ_S^k(u)(x) = P^\top V^\top V P ((\mathcal{E} u)(x)) + \gamma u(x) = e_1^\top (\tilde{\gamma}-\gamma) I e_1 u(x) + \gamma u(x) = (\tilde{\gamma}-\gamma)u(x) + \gamma u(x) = \tilde{\gamma} u(x), \quad \text{~a.e.~in~} \Omega.
    \end{equation*}

    For (ii), since $\|Q_S\|^{-1} > 0$, we can take $\gamma > 0$ and nonzero $P \in \bR^{m_p}$, $V \in \bR^{m_p \times m_p}$, and $\Phi$ supported on finitely many Fourier modes with $\Phi(k) = \overline{\Phi(-k)}$ for all $k \in \bZ^d$, such that
    \begin{equation*}
        \gamma + \| V P \|^2 + \max_{k \in \bZ^d} \| \Phi(k) P \|^2 \leq \| Q_S \|^{-1}.
    \end{equation*}
    Let $\cQ_S$ be defined by \eqref{eq:qsk-def} with these parameters and $\mathcal{E}$ being the zero extension operator, whose adjoint $\mathcal{E}^*: L^2(\bT^d) \to L^2(\Omega)$ is the restriction operator $v \mapsto v|_\Omega$.
    Then it follows from \eqref{eq:qsk-def-pre} that
    \begin{equation*}
        \cQ_S = \mathcal{E}^* \circ \mathcal{P}^\top \circ \cK \circ \mathcal{P} \circ \mathcal{E} + \gamma I,
    \end{equation*}
    where $\cK$ is bounded, self-adjoint, and positive semidefinite by \Cref{cor:spd-oper-fourier}.
    Since $\gamma > 0$, $\mathcal{P}$ is bounded, and $\mathcal{P}^\top$ is the adjoint of $\mathcal{P}$, the operator $\cQ_S$ is bounded, self-adjoint, and positive definite.
    Moreover, for any $u \in L^2(\Omega)$, denoting $\tilde{u} = \mathcal{E}(u)$ and $\phi = \cF^{-1}(\Phi)$, it follows from the convolution theorem and Parseval's identity that
    \begin{equation*}
        \begin{aligned}
            \gamma \|u\|_{L^2(\Omega)}^2 \leq \left\langle \cQ_S u, u \right\rangle_{L^2(\Omega)} &= \gamma \left\| u \right\|_{L^2(\Omega)}^2 + \left\| V P \tilde{u} \right\|_{L^2(\bT^d; \bR^{m_p})}^2 + \left\| \phi \ast (P \tilde{u}) \right\|_{L^2(\bT^d; \bR^{m_p})}^2 \\
            &\leq \Big( \gamma + \| V P \|^2 + \max_{k \in \bZ^d} \| \Phi(k) P \|^2 \Big) \left\| u \right\|_{L^2(\Omega)}^2 \leq \| Q_S \|^{-1} \left\| u \right\|_{L^2(\Omega)}^2.
        \end{aligned}
    \end{equation*}
    Hence, $\cQ_S^{-1}$ is bounded, self adjoint, and satisfies $\langle \cQ_S^{-1} p, p \rangle_Y \geq \|Q_S\| \|p\|_Y^2 \geq \langle Q_S p, p \rangle_Y$ for any $p \in Y$.
    This guarantees that $\cQ_S^{-1} - Q_S$ is bounded, self-adjoint, and positive semidefinite.
    Finally, note that $\cQ_S(u) = (\gamma + \| V P \|^2) u + \left. \mathcal{P}^\top \left( \phi^\top_{-} \ast \phi \ast \mathcal{P}(\tilde{u}) \right) \right|_{\Omega}$ for any $u \in L^2(\Omega)$.
    Since $\Phi$ is supported on finitely many Fourier modes, the second term defines a continuous operator from $L^2(\Omega)$ into a finite-dimensional space of trigonometric polynomials restricted to $\Omega$, which is contained in $H^s(\Omega)$ for any $s \geq 0$.
    Therefore, $\cQ_S$ maps $H^s(\Omega)$ continuously into $H^s(\Omega)$.

\end{proof}

\subsection{Universal Approximation of the iUzawa-Net}

We now establish a universal approximation property of the \texttt{iUzawa-Net}, highlighting its capacity to approximate the solution operator of \eqref{eq:opt-ctrl} with theoretical guarantees.

\begin{theorem}\label{thm:iuzawa-univ-approx}
    Let $T: Y \times U \to U$ be the solution operator of \eqref{eq:opt-ctrl}.
    For any compact set $K \subset Y \times U$ and $\varepsilon > 0$, there exists a two-layer \texttt{iUzawa-Net} $\mathcal{T}(y_d, f; \theta_{\cT})$ defined in \eqref{eq:inexact-uzawa-duf}, such that
    \begin{equation*}
        \sup_{(y_d, \, f)^\top \in K} \left\| T(y_d, f) - \mathcal{T}(y_d, f; \theta_{\mathcal{T}}) \right\|_U \leq \varepsilon.
    \end{equation*}
\end{theorem}

\begin{proof}
    For any two-layer \texttt{iUzawa-Net} $\mathcal{T}$, we have $\mathcal{T}(y_d, f; \theta_{\mathcal{T}}) = u^2$, where $u^2$ is the second-layer output defined in \eqref{eq:inexact-uzawa-duf}.
    It follows that
    \begin{equation*}
        u^2 = \cQ_A^1 \left( -\cA^1 \left( \cQ_S^0 \left( \cS^0 \left( \cQ_A^0 \left( - \cA^0 \left(p^0\right) \right) + f \right) - y_d\right)\right) + \tau \cQ_A^0 \left( - \cA^0 \left(p^0\right) \right) \right).
    \end{equation*}
    We take an FNO $\cA^0: Y \to U$ with all-zero parameter (hence $\cA^0 = 0$).
    Similarly, we take $\cQ_A^0 = 0$ defined in \eqref{eq:qak-ptwise}, where the neural network $\mathcal{N}: \bR^{m+2} \to \bR$ of the form \eqref{eq:nn-rec-structure} has all-zero parameters.
    It follows that, for any $\varepsilon_1 > 0$,
    \begin{equation}\label{eq:no-qa0-approx-target}
        \left\| \cQ_A^0 \left( - \cA^0 \left(p^0\right) \right) \right\|_U < \varepsilon_1.
    \end{equation}
    The compactness of $K$ implies that its projections onto $U$ and $Y$, 
denoted by $K_U$ and $K_Y$, are compact.
    Define $K_0 = \left\{\cQ_A^0 \left( - \cA^0 \left(p^0\right)\right)\right\}$ and $K_1 = K_U + K_0$. Then both $K_0$ and $K_1$ are compact.
    By \Cref{prop:fno-univ-approx-general}, there exists an FNO $\cS^0$ such that
    \begin{equation*}
        \sup_{v \in K_1} \left\| Sv - \cS^0(v) \right\|_Y < \varepsilon_2,
    \end{equation*}
    and hence for any $f \in K_U$,
    \begin{equation}\label{eq:tiua-aux-1}
        \begin{aligned}
            \left\| Sf - \cS^0\left(\cQ_A^0 \left( - \cA^0 \left(p^0\right) \right) + f \right) \right\|_Y &\leq \left\| Sf - S\left(\cQ_A^0 \left( - \cA^0 \left(p^0\right) \right) + f \right) \right\|_Y \\
            &\qquad + \left\| S\left(\cQ_A^0 \left( - \cA^0 \left(p^0\right) \right) + f \right) - \cS^0\left(\cQ_A^0 \left( - \cA^0 \left(p^0\right) \right) + f \right) \right\|_Y \\
            & < \|S\|\varepsilon_1 + \varepsilon_2.
        \end{aligned}
    \end{equation}
    By \Cref{lem:param-to-ctrl-lip}, $T(y_d, f) = T_0(Sf - y_d)$ for some continuous $T_0: Y \to U$.
    Thus \Cref{prop:fno-univ-approx-general} guarantees an FNO $\cA^1$ such that
    \begin{equation}\label{eq:no-a-approx-target}
        \sup_{z \in K_2} \left\| -T_0(z) - \cA^1(z) \right\|_U < \varepsilon_3,
    \end{equation}
    where $K_2 := \cS^0 (K_1) - K_Y$ is compact by \Cref{prop:fno-univ-approx-general}.
    From \Cref{prop:qsk-univ-approx}, we take any $\gamma < 1$ in \eqref{eq:qsk-def} and choose $\cQ_S^0 = I$.
    It follows from \Cref{lem:param-to-ctrl-lip}, \eqref{eq:tiua-aux-1} and \eqref{eq:no-a-approx-target} that
    \begin{equation*}
        \begin{aligned}
            & \quad \left\| T_0(Sf - y_d) - \left((-\cA^1) \circ \cQ_S^0\right) \left( \cS^0 \left( \cQ_A^0 \left( - \cA^0 \left(p^0\right) \right) + f \right) - y_d\right) \right\|_U \\
            & \leq \left\| T_0(Sf - y_d) - T_0\left( \cS^0 \left( \cQ_A^0 \left( - \cA^0 \left(p^0\right) \right) + f \right) - y_d\right) \right\|_U \\
            & \qquad + \left\| T_0\left( \cS^0 \left( \cQ_A^0 \left( - \cA^0 \left(p^0\right) \right) + f \right) - y_d\right) + \cA^1 \left( \cS^0 \left( \cQ_A^0 \left( - \cA^0 \left(p^0\right) \right) + f \right) - y_d \right) \right\|_U \\
            & < c_0^{-1} \|S\| (\|S\|\varepsilon_1 + \varepsilon_2) + \varepsilon_3. 
        \end{aligned}
    \end{equation*}
    Finally, we take $\cQ_A^1$ to approximate the identity operator over the compact set
    \begin{equation*}
        K_3 := -\cA^1(K_2) + \tau K_0,
    \end{equation*}
    with accuracy
    \begin{equation*}
        \sup_{v \in K_3} \| v - \cQ_A^1 v \|_U < \varepsilon_4.
    \end{equation*}
    This is possible since the identity operator is itself of the form \eqref{eq:qak-ptwise}, where the neural network in \eqref{eq:nn-rec-structure} is given by $\mathcal{N}(r, \xi, \eta) = r$, i.e., $W_L$ selects the input $r$ through the concatenated skip connection and all the other parameters are zero.
    Let $\bar{v}=\left((-\cA^1) \circ \cQ_S^0\right) \left( \cS^0 \left( \cQ_A^0 \left( - \cA^0 \left(p^0\right) \right) + f \right) - y_d\right)$, then $\bar{v}\in -\cA^1(K_2)$ and thus $\bar{v}+\tau \cQ_A^0 \left( - \cA^0 \left(p^0\right)\right)\in K_3$.
    As a result, we have
    \begin{equation*}
    \begin{aligned}
    &\|u^2 - T(y_d, f) \|_U=\|\cQ_A^1\left(\bar{v}+\tau \cQ_A^0 \left( - \cA^0 \left(p^0\right)\right)\right)-T_0(Sf-y_d)\|\\
    \leq& \left\|\cQ_A^1\left(\bar{v}+\tau \cQ_A^0 \left( - \cA^0 \left(p^0\right)\right)\right)-\left(\bar{v}+\tau \cQ_A^0 \left( - \cA^0 \left(p^0\right)\right)\right)\right\|+\|\bar{v}-T_0(Sf-y_d)\|+\|\tau \cQ_A^0 \left( - \cA^0 \left(p^0\right)\right)\|\\
    <&\left(c_0^{-1} \|S\|^2 + \tau \right) \varepsilon_1 + c_0^{-1} \|S\| \varepsilon_2 + \varepsilon_3 + \varepsilon_4 =: \varepsilon,
    \end{aligned}
    \end{equation*}
    for any $(y_d, f) \in K$.
    The result follows immediately by choosing $\varepsilon_1, \varepsilon_2, \varepsilon_3, \varepsilon_4$ sufficiently small.
\end{proof}

Note that \Cref{thm:iuzawa-univ-approx} does not imply universal approximation for arbitrary mappings between Hilbert spaces, since its proof exploits the factorization $T(y_d,f)=T_0(Sf-y_d)$ of the target solution
operator.
While the design of the \texttt{iUzawa-Net} restricts the class of approximable operators, it embeds the intrinsic mathematical structure of the optimal control problem \eqref{eq:opt-ctrl} into the network. Consequently, the \texttt{iUzawa-Net} fundamentally distinguishes itself from standard black-box operator learning approaches, leveraging its optimization-informed nature to enhance both interpretability and robustness.

The proof of \Cref{thm:iuzawa-univ-approx} constructs a network where modules like $\cA^1$ and $\cQ_A^0$ do not correspond to their counterparts $S^*$ or $(N+\tau I + \partial \theta)^{-1}$ in the blueprint optimization algorithm \eqref{eq:inexact-uzawa}, but rather serve as generic approximators, and its two layers consist of modules with distinct parameters.
Such a configuration does not adhere to the optimization-informed design principle of the \texttt{iUzawa-Net}.
To address this, we provide stronger approximation results in \Cref{sec:stability-analysis}, which guarantee that the learnable modules approximate their respective target operators and, more importantly, ensure that the layer outputs of a weight tying \texttt{iUzawa-Net} faithfully track the iterates of the inexact Uzawa method.

\section{Asymptotic \texorpdfstring{$\varepsilon$}{epsilon}-Optimality of Layer Outputs under Weight Tying}\label{sec:stability-analysis}

In this section, we prove the existence of an \texttt{iUzawa-Net} $\cT(y_d, f; \theta_\cT)$ whose layer outputs closely approximate inexact iterates of \eqref{eq:inexact-uzawa} and whose layers share identical parameters.
We formally define these two desired properties as \emph{algorithm tracking} and \emph{weight tying} in \Cref{def:alg-track,def:weight-tying}, respectively.
As a consequence, we establish in \Cref{thm:iuzawa-net-univ-approx-rec} the asymptotic $\varepsilon$-optimality of the layer outputs of $\mathcal{T}$.
This yields a universal approximation theorem for the \texttt{iUzawa-Net} under the algorithm tracking and weight tying requirements.
If the input of the \texttt{iUzawa-Net} admits additional regularity, the approximation result holds on bounded sets in the corresponding Sobolev spaces, which does not require the conventional compactness assumptions in universal approximation theorems.

We remark that both algorithm tracking and weight tying are naturally expected for an optimization-informed neural network.
The former is motivated by the optimization-informed structural design, and the latter corresponds to the identical operations in each iteration of the underlying optimization algorithm \eqref{eq:inexact-uzawa}.

\begin{definition}\label{def:alg-track}
    Let $Q_A = N + \tau I$ for some $\tau \geq 0$ and $Q_S$ be a fixed linear, bounded, self-adjoint, and positive definite operator.
    For any $K \subset Y \times U$ and $\delta > 0$, we say that an \texttt{iUzawa-Net} is \emph{algorithm tracking with respect to $K$ and $\delta$}, or simply \emph{$(K, \delta)$-tracking}, if for any input $(y_d, f)^\top \in K$ and initialization $(p^0, u^0)^\top \in K \cup \{0\}$, its $k$-th layer output $u^k$ and $p^k$ satisfy
    \begin{subequations}\label{eq:inexact-inexact-uzawa} 
    \begin{empheq}[left=\empheqlbrace]{align}
        & \bar{u}^{k+1} = \left(Q_A + \partial \theta\right)^{-1} \left(\tau u^k - S^* p^k\right), \label{eq:inexact-inexact-uzawa-bar-u}\\
        & \bar{p}^{k+1} = p^k + Q_S^{-1} (S \bar{u}^{k+1} - p^k + Sf - y_d), \label{eq:inexact-inexact-uzawa-bar-p}\\
        & u^{k+1} \in \left\{ u \in U \, : \, \left\|u - \bar{u}^{k+1} \right\|_U \leq \delta \right\}, \label{eq:inexact-inexact-uzawa-u}\\
        & p^{k+1} \in \left\{ p \in Y \, : \, \left\|p - \bar{p}^{k+1} \right\|_Y \leq \delta \right\}, \label{eq:inexact-inexact-uzawa-p}
    \end{empheq}
    \end{subequations}
    for any $k = 0, \ldots, L-1$. \footnote{We require $(p^0, u^0)^\top \in K \cup \{0\}$ since we are mainly interested in those \texttt{iUzawa-Net}s with $u^0 = p^0 = 0$.}
\end{definition}

\begin{definition}\label{def:weight-tying}
    An \texttt{iUzawa-Net} defined in \eqref{eq:inexact-uzawa-duf} is called \emph{weight tying} if each layer consists of identical neural network parameters.
    Specifically, we have $\cS^k = \cS$, $\cA^k = \cA$, $\cQ_A^k = \cQ_A$ and $\cQ_S^k = \cQ_S$ for some fixed neural networks $\cS$, $\cA$, $\cQ_A$ and $\cQ_S$ for any $k = 0, \ldots, L-1$.
\end{definition}

Our analysis proceeds as follows.
First, we analyze the inexact optimization algorithm \eqref{eq:inexact-inexact-uzawa} itself, proving that its sequence eventually enters and remains within a neighborhood of the optimal control $u^*$.
Next, we prove that the FNOs and $\cQ_A^k$ defined in \eqref{eq:inexact-uzawa-duf} can approximate their target operators with arbitrary accuracy while, for each compact set, keeping the images inside a common compact set independent of the chosen accuracy.
This strengthens the classical universal approximation property of these operators.
Finally, we show that for any compact $K \subset Y \times U$ and $\delta > 0$, there exists an \texttt{iUzawa-Net} that is $(K, \delta)$-tracking and weight tying by combining the aforementioned results.

\subsection{Inexact Iterations of \texorpdfstring{\eqref{eq:inexact-uzawa}}{the inexact Uzawa method} and Asymptotic \texorpdfstring{$\varepsilon$-Optimality}{eps-Optimality}}

In the remaining part of this subsection, we fix $Q_A = N + \tau I$ and $Q_S$ such that the operators
\begin{equation}\label{eq:inuzawa-conv-cond}
    Q_A - N, \quad Q_S - (I + SN^{-1}S^*) \quad \text{are positive semidefinite,}
\end{equation}
and consider \emph{any} inexact iterates $\{u^k\}_{k=0}^{\infty}, \{p^k\}_{k=0}^\infty$ that satisfy \eqref{eq:inexact-inexact-uzawa} for all $k \geq 0$ (note that $\{u^k\}_{k=0}^{\infty}$ and $\{p^k\}_{k=0}^\infty$ are not necessarily layer outputs of an \texttt{iUzawa-Net}).
For convenience, we denote $w, \bar{w} \in U \times Y$ and two bounded linear self-adjoint operators $Q: U \times Y \to U \times Y$ and $R: U \times Y \to U \times Y$ as
\begin{equation*}    
    w = \begin{pmatrix}
        u \\
        p
    \end{pmatrix}, \quad \bar{w} = \begin{pmatrix}
        \bar{u} \\
        \bar{p}
    \end{pmatrix}, \quad Q(u, p) = \begin{pmatrix}
        (Q_A - N) \, u \\
        Q_S \, p
    \end{pmatrix}, \quad R(u, p) = \begin{pmatrix}
        (Q_A - N) \, u \\
        (Q_S - \frac{1}{2}(I + SN^{-1}S^*)) \, p
    \end{pmatrix}.
\end{equation*}
We take the inner product $\langle \cdot, \, \cdot \rangle_{U \times Y}$ on $U \times Y$ as the canonical one induced by the Cartesian product.
It is clear that $Q$ and $R$ are positive semidefinite, which introduces the following seminorms on $U \times Y$:
\begin{equation*}
    \left\| w \right\|_Q^2 := \left\langle Q w, w \right\rangle _{U \times Y}, \quad \left\| w \right\|_R^2 := \left\langle R w, w \right\rangle _{U \times Y}.
\end{equation*}
Similar notations to the above will be used to denote the seminorms induced by other positive semidefinite operators.
We use $\sigma_0(\cdot)$ to denote the infimum of the spectrum modulus of a given operator.

We begin by showing that $\bar{w}^{k+1}$ is a contractive iterate with respect to $w^k$ for all $k \geq 0$.
The following \Cref{lem:next-exact-iter-contract} is a direct extension of \cite[Theorem 4.4]{he2014convergence} and \cite[Theorem 3.3]{song2019inexact} to the infinite-dimensional setting.
We provide a detailed proof for completeness.

\begin{lemma}\label{lem:next-exact-iter-contract}
    Let $\{w^k\}_{k=0}^\infty$ and $\{\bar{w}^k\}_{k=1}^\infty$ satisfy \eqref{eq:inexact-inexact-uzawa}.
    Let $w^* = (u^*, p^*)^\top$ be the unique solution to \eqref{eq:opt-ctrl}.
    For any $k \geq 0$, it holds that
    \begin{equation}
        \| \bar{w}^{k+1} - w^* \|_Q^2 \leq \| w^k - w^* \|_Q^2 - \| \bar{w}^{k+1} - w^k \|_R^2.
    \end{equation}
\end{lemma}

\begin{proof}
    By \eqref{eq:inexact-inexact-uzawa}, it is easy to verify that
    \begin{equation}
        \left\{
        \begin{aligned}
            & Q_A(u^k - \bar{u}^{k+1}) - Nu^k - S^* p^k \in \partial \theta(\bar{u}^{k+1}), \\
            & Q_S(p^k - \bar{p}^{k+1}) + S \bar{u}^{k+1} - p^k + (Sf - y_d) = 0.
        \end{aligned}
        \right.
    \end{equation}
    Since $w^*$ solves \eqref{eq:opt-ctrl}, it follows from the optimality condition \eqref{eq:opt-ctrl-opt-cond} that
    \begin{equation}
        \left\{
        \begin{aligned}
            & - Nu^* - S^* p^* \in \partial \theta(u^*), \\
            & S u^* - p^* + (Sf - y_d) = 0.
        \end{aligned}
        \right.
    \end{equation}
    From the monotonicity of $\partial \theta$, we have
    \begin{equation*}
        \begin{aligned}
            & \langle (Q_A - N)(u^k - \bar{u}^{k+1}) - N(\bar{u}^{k+1} - u^*) - S^* (p^k - p^*), \ \bar{u}^{k+1} - u^* \rangle_U \geq 0, \\
            & Q_S (p^k - \bar{p}^{k+1}) + S(\bar{u}^{k+1} - u^*) - (p^k - p^*) = 0.
        \end{aligned}
    \end{equation*}
    It follows that
    \begin{equation}\label{eq:inner-prod-lower-bound}
        \begin{aligned}
            &\left\langle Q(w^k - \bar{w}^{k+1}), \bar{w}^{k+1} - w^* \right\rangle _{U \times Y} \\
            = &\left\langle(Q_A - N) (u^k - \bar{u}^{k+1}), \bar{u}^{k+1} - u^* \right\rangle _U + \left\langle Q_S(p^k - \bar{p}^{k+1}), \bar{p}^{k+1} - p^k + p^k - p^* \right\rangle _Y \\
            \geq &\left\langle(Q_A - N) (u^k - \bar{u}^{k+1}), \bar{u}^{k+1} - u^* \right\rangle_U + \left\langle -S(\bar{u}^{k+1} - u^*) + (p^k - p^*), \bar{p}^{k+1} - p^k + p^k - p^* \right\rangle_Y \\
            \geq &\left\langle N(\bar{u}^{k+1} - u^*) + S^* (p^k - p^*), \ \bar{u}^{k+1} - u^* \right\rangle _U - \left\| p^k - \bar{p}^{k+1} \right\|_{Q_S}^2 \\
            & \qquad - \left\langle S(\bar{u}^{k+1} - u^*), p^k - p^* \right\rangle _Y + \left\| p^k - p^* \right\|_Y^2 \\
            = &\left\langle N(\bar{u}^{k+1} - u^*), \ \bar{u}^{k+1} - u^* \right\rangle _U + \left\langle S(\bar{u}^{k+1} - u^*), \ p^k - \bar{p}^{k+1} \right\rangle _Y \\
            & \qquad - \left\langle p^k - p^*, \ p^k - \bar{p}^{k+1} \right\rangle _Y + \left\| p^k - p^* \right\|_Y^2.
        \end{aligned}
    \end{equation} 
    Since $N$ is bounded, self-adjoint, and positive definite over the Hilbert space $U$, it admits a bounded, self-adjoint, and positive definite square root.
    Utilizing this square root and applying the Cauchy-Schwarz inequality, we have for all $\omega > 0$ that
    \begin{equation}\label{eq:cauchy-schwarz-inequal}
        \begin{aligned}
            &\left\langle S(\bar{u}^{k+1} - u^*), \ p^k - \bar{p}^{k+1} \right\rangle _Y \geq - \omega \left\| \bar{u}^{k+1} - u^* \right\|_N^2 - \frac{1}{4\omega} \left\| p^k - \bar{p}^{k+1} \right\|_{SN^{-1}S^*}^2, \\
            &\left\langle p^k - p^*, \ p^k - \bar{p}^{k+1} \right\rangle _Y \leq \omega \left\| p^k - p^* \right\|_Y^2 + \frac{1}{4\omega} \left\| p^k - \bar{p}^{k+1} \right\|_Y^2.
        \end{aligned}
    \end{equation}
    Take $\omega = 1$ in \eqref{eq:cauchy-schwarz-inequal}. Together with \eqref{eq:inner-prod-lower-bound}, we obtain
    \begin{equation*}
        \begin{aligned}
            \left\langle Q(w^k - \bar{w}^{k+1}), \bar{w}^{k+1} - w^* \right\rangle _{U \times Y}
            \geq - \frac{1}{4} \left\| p^k - \bar{p}^{k+1} \right\|_{SN^{-1}S^*}^2 - \frac{1}{4} \left\| p^k - \bar{p}^{k+1} \right\|_Y^2 = -\frac{1}{4} \left\| p^k - \bar{p}^{k+1} \right\|_{I + SN^{-1}S^*}^2.
        \end{aligned} 
    \end{equation*}
    Therefore, we have
    \begin{equation*}
        \begin{aligned}
            \left\| w^k - w^* \right\|_Q^2 &= \left\| w^k - \bar{w}^{k+1} \right\|_Q^2 + \left\| \bar{w}^{k+1} - w^* \right\|_Q^2 + 2 \left\langle Q(w^k - \bar{w}^{k+1}), \bar{w}^{k+1} - w^* \right\rangle _{U \times Y} \\
            &\geq  \left\| w^k - \bar{w}^{k+1} \right\|_Q^2 + \left\| \bar{w}^{k+1} - w^* \right\|_Q^2 - \frac{1}{2} \left\| p^k - \bar{p}^{k+1} \right\|_{SN^{-1}S^* + I}^2 \\
            &= \left\| \bar{w}^{k+1} - w^* \right\|_Q^2 + \left\|w^k - \bar{w}^{k+1} \right\|_R^2.
        \end{aligned}
    \end{equation*}
\end{proof}

Next, we bound $\|\bar{w}^{k} - w^*\|_{U \times Y}$ using the properties of the following set-valued mapping:
\begin{equation}\label{eq:phi-oper-def}
    \varphi: U \times Y \rightrightarrows U \times Y, \quad \varphi(w) = \begin{pmatrix}
        \varphi_1(w) \\
        \varphi_2(w)
    \end{pmatrix} := \begin{pmatrix}
        Nu + \partial \theta(u) + S^* p \\
        -Su + p - (Sf - y_d)
    \end{pmatrix}.
\end{equation}
Note that $0 \in \varphi(w^*)$ by the optimality conditions \eqref{eq:opt-ctrl-opt-cond}.

\begin{lemma}\label{lem:phi-metric-subreg}
    There exists some $\kappa_0 > 0$ such that for any $w_1, w_2 \in U \times Y$, it holds that
    \begin{equation}\label{eq:phi-metric-subreg}
        \kappa_0 \cdot \mathrm{dist}\left(\varphi(w_1), \varphi(w_2)\right) \geq \left\| w_1 - w_2 \right\|_{U \times Y}.
    \end{equation}
\end{lemma}

\begin{proof}
    The inequality trivially holds if $\varphi(w_1)$ or $\varphi(w_2)$ is empty, where the distance in \eqref{eq:phi-metric-subreg} is infinity.
    Otherwise, take any $\zeta_1 \in \varphi(w_1)$ and $\zeta_2 \in \varphi(w_2)$.
    By \eqref{eq:phi-oper-def}, there exist $g_i \in \partial \theta(u_i)$, $i = 1, 2$, such that
    \begin{equation*}
        \zeta_i = \begin{pmatrix}
            N u_i + g_i + S^* p_i \\
            - S u_i + p_i - (Sf - y_d)
        \end{pmatrix}, \quad i = 1, 2.
    \end{equation*}
    Note that the terms involving $S$ and $S^*$ cancel, i.e., $\left\langle S^*(p_1 - p_2), u_1 - u_2 \right\rangle_U - \left\langle S(u_1 - u_2), p_1 - p_2 \right\rangle_Y = 0$.
    Hence, it follows from the monotonicity of $\partial \theta$ and the coercivity of $N$ that
    \begin{equation*}
        \begin{aligned}
            \left\langle \zeta_1 - \zeta_2, w_1 - w_2 \right\rangle_{U \times Y} &= \left\langle N(u_1 - u_2) + g_1 - g_2, u_1 - u_2 \right\rangle_U + \left\| p_1 - p_2 \right\|_Y^2 \\
            &\geq c_0 \left\| u_1 - u_2 \right\|_U^2 + \left\| p_1 - p_2 \right\|_Y^2 \geq \min\{c_0, 1\} \left\| w_1 - w_2 \right\|_{U \times Y}^2.
        \end{aligned}
    \end{equation*}
    Applying the Cauchy-Schwarz inequality to the left-hand side, we obtain
    \begin{equation*}
        \left\| w_1 - w_2 \right\|_{U \times Y} \leq \kappa_0 \left\| \zeta_1 - \zeta_2 \right\|_{U \times Y}, \quad \text{with~} \kappa_0 := \max\{c_0^{-1}, 1\}.
    \end{equation*}
    Since $\zeta_1 \in \varphi(w_1)$ and $\zeta_2 \in \varphi(w_2)$ are taken arbitrarily, we conclude that \eqref{eq:phi-metric-subreg} holds.
\end{proof}

\begin{lemma}\label{lem:ppa-iter-err-bd}
    Let $\{w^k\}_{k=0}^\infty$ and $\{\bar{w}^{k+1}\}_{k=0}^\infty$ satisfy \eqref{eq:inexact-inexact-uzawa}.
    Let $w^* = (u^*, p^*)^\top$ be the unique solution to \eqref{eq:opt-ctrl}.
    Then for any $k \geq 0$, it holds that
    \begin{equation}
        \left\| \bar{w}^{k+1} - w^* \right\|_Q \leq \kappa \| \bar{w}^{k+1} - w^k \|_R
    \end{equation}
    for some constant $\kappa > 0$.
\end{lemma}

\begin{proof}
    Note that \eqref{eq:inexact-inexact-uzawa} implies
    \begin{equation}\label{eq:next-exact-iter-oper-form}
        Q\left(w^k - \bar{w}^{k+1}\right) - \begin{pmatrix}
            S^*(p^k - \bar{p}^{k+1}) \\
            p^k - \bar{p}^{k+1}
        \end{pmatrix} \in \varphi(\bar{w}^{k+1}).
    \end{equation}
    By \eqref{eq:next-exact-iter-oper-form}, \Cref{lem:phi-metric-subreg}, and the fact that $0 \in \varphi(w^*)$, we have
    \begin{equation}\label{eq:pieb-bound-step-1}
        \begin{aligned}
            \left\|\bar{w}^{k+1} - w^*\right\|_Q^2 &\leq \left\| Q \right\| \left\|\bar{w}^{k+1} - w^* \right\|_{U \times Y}^2 \\
            &\leq \kappa_0^2 \left\| Q \right\| \cdot \mathrm{dist}^2\left( \varphi(w^*), \varphi(\bar{w}^{k+1}) \right) \\
            &\leq \kappa_0^2 \left\| Q \right\| \cdot \mathrm{dist}^2\left( 0, \varphi(\bar{w}^{k+1}) \right) \\
            &\leq \kappa_0^2 \left\| Q \right\| \left\| Q(w^k - \bar{w}^{k+1}) - \begin{pmatrix}
                S^*(p^k - \bar{p}^{k+1}) \\
                p^k - \bar{p}^{k+1}
            \end{pmatrix} \right\|_{U \times Y}^2 \\
            &\leq \kappa_0^2 \left\| Q \right\| \left( \left\| (Q_A - N)(u^k - \bar{u}^{k+1}) - S^*(p^k - \bar{p}^{k+1}) \right\|_{U}^2 + \left\| (Q_S - I) (p^k - \bar{p}^{k+1}) \right\|_{Y}^2 \right) \\
            &\leq \kappa_0^2 \left\|Q\right\| \biggl( 2 \left\|Q_A - N\right\| \left\| u^k - \bar{u}^{k+1} \right\|_{Q_A - N}^2 + 2 \left\|p^k - \bar{p}^{k+1}\right\|_{SS^*}^2 \\
            & \qquad + \left\|Q_S - I\right\| \left\|p^k - \bar{p}^{k+1} \right\|_{Q_S - I}^2 \biggr).
        \end{aligned}
    \end{equation}
    Since $N$ is bounded and continuously invertible, by \eqref{eq:inuzawa-conv-cond} we have
    \begin{equation}\label{eq:pieb-bd-step-2}
        \begin{aligned}
            2\left( Q_S - \frac{1}{2}I - \frac{1}{2} SN^{-1}S^* \right) &\succeq \left( Q_S - \frac{1}{2}I - \frac{1}{2} SN^{-1}S^* \right) + \left( \frac{1}{2}I + \frac{1}{2} SN^{-1}S^* \right) = Q_S  \\
            &\succeq Q_S - I, \\
            2\left( Q_S - \frac{1}{2}I - \frac{1}{2} SN^{-1}S^* \right) &\succeq 2 \left(\frac{1}{2}I + \frac{1}{2}SN^{-1}S^*\right) \succeq \sigma_0(N^{-1}) SS^* = \|N\|^{-1} SS^*,
        \end{aligned}
    \end{equation}
    where the notation $\succeq$ means that the operator on its left hand side minus the operator on its right hand side is positive semidefinite.
    Combining \eqref{eq:pieb-bound-step-1} and \eqref{eq:pieb-bd-step-2}, we conclude that
    \begin{equation*}
        \begin{aligned}
            \left\|\bar{w}^{k+1} - w^*\right\|_Q^2 &\leq 2 \kappa_0^2 \left\|Q\right\| \max\bigg\{ \left\|Q_A - N\right\|, \, 2 \|N\|+ \left\|Q_S - I\right\| \bigg\} \left\| \bar{w}^{k+1} - w^k \right\|_R^2,
        \end{aligned}
    \end{equation*}
    and thus complete the proof by taking the square root of both sides with 
    \begin{equation*}
    \begin{aligned}
        \kappa &:= \Big( 2\kappa_0^2 \left\|Q\right\| \max\big\{ \left\|Q_A - N\right\|, \, 2 \|N\|+ \left\|Q_S - I\right\| \big\} \Big)^{1/2} \\
        &= \Big( 2 \kappa_0^2 \cdot \max\big\{\left\|Q_A - N\right\|, \, \left\| Q_S \right\|\big\} \cdot \max\big\{ \left\|Q_A - N\right\|, \, 2 \|N\|+ \left\|Q_S - I\right\| \big\} \Big)^{1/2}.
    \end{aligned}
    \end{equation*}
\end{proof}

With the above lemmas, we are now in a position to establish the asymptotic $\varepsilon$-optimality of \eqref{eq:inexact-inexact-uzawa}. 

\begin{proposition}\label{prop:inexact-inexact-uzawa-stability}
    Let $\{w^k\}_{k=0}^\infty$ and $\{\bar{w}^k\}_{k=1}^\infty$ satisfy \eqref{eq:inexact-inexact-uzawa}.
    Let $w^* = (u^*, p^*)^\top$ be the unique solution to \eqref{eq:opt-ctrl}.
    There exist constants $0 < \rho < 1$ and $c > 0$ independent of $\delta$, such that
    \begin{equation}\label{eq:perturbed-lin-conv-rate}
        \left\|w^{k+1} - w^* \right\|_Q \leq \rho \left\|w^k - w^* \right\|_Q + c \delta.
    \end{equation}
    In particular, for any $\varepsilon > 0$, there exist $\delta_0 = C \varepsilon$ with $C > 0$ independent of $\varepsilon$ and $L = O(\log(1/\varepsilon))$, such that for any $\delta < \delta_0$ and $k \geq L$, the iterates $\{w^k\}_{k=0}^\infty$ of $\eqref{eq:inexact-inexact-uzawa}$ satisfy
    \begin{equation}\label{eq:inexact-uzawa-conv-to-eps}
        \left\|w^{k} - w^* \right\|_{U \times Y} < \varepsilon.
    \end{equation}
\end{proposition}

\begin{proof}
    Combining \Cref{lem:next-exact-iter-contract} and \Cref{lem:ppa-iter-err-bd} yields
    \begin{equation}\label{eq:iius-pf-aux-1}
        \left\| \bar{w}^{k+1} - w^* \right\|_Q^2 \leq \frac{\kappa^2}{1 + \kappa^2} \left\| w^k - w^* \right\|_Q^2.
    \end{equation}
    By \eqref{eq:iius-pf-aux-1}, \eqref{eq:inexact-inexact-uzawa-u} and \eqref{eq:inexact-inexact-uzawa-p}, we have
    \begin{equation}\label{eq:inexact-uzawa-stability-bd}
    \begin{aligned}
        \left\| w^{k+1} - w^* \right\|_Q &\leq \left\| \bar{w}^{k+1} - w^* \right\|_Q + \left\| w^{k+1} - \bar{w}^{k+1} \right\|_Q \leq \rho \left\| w^k - w^* \right\|_Q + (2 \|Q\|)^{1/2} \delta \\
        &= \rho \left\| w^k - w^* \right\|_Q + c \delta,
    \end{aligned}
    \end{equation}
    where
    \begin{equation*}
        \rho = \left(\kappa^2 / (1 + \kappa^2)\right)^{1/2} \in (0, 1), \quad c = \left(2 \|Q \| \right)^{1/2} = \left(2 \max \left\{ \left\|Q_A - N \right\|, \left\| Q_S \right\| \right\} \right)^{1/2} > 0.
    \end{equation*}
    Note that both $\rho$ and $c$ are independent of $\delta$.
    
    Now we fix some $\delta > 0$. Consider any 
    \begin{equation}\label{eq:tilde-eps1-def}
        \varepsilon_1 > \left(\frac{2}{1-\rho} + 1 \right) c \delta,
    \end{equation}
    and define
    \begin{equation*}
        \tilde{\rho} = \rho + \frac{c \delta}{\varepsilon_1 - c \delta}.
    \end{equation*}
    Note that $\rho < \tilde{\rho} < \frac{1 + \rho}{2} \in (0, 1)$.
    Indeed, $\rho < \tilde{\rho}$ follows from $\varepsilon_1 > c \delta$. Moreover, \eqref{eq:tilde-eps1-def} is equivalent to $\varepsilon_1 - c\delta > \frac{2 c \delta}{1 - \rho}$, which implies
    \begin{equation*}
        \tilde{\rho} = \rho + \frac{c \delta}{\varepsilon_1 - c \delta} < \rho + \frac{1 - \rho}{2} = \frac{1 + \rho}{2} < 1,
    \end{equation*}
    where the last inequality is due to $\rho < 1$.
    For each iterate $w^k$, by \eqref{eq:inexact-uzawa-stability-bd} and the definitions of $\varepsilon_1$ and $\tilde{\rho}$, it is easy to verify the following:
    \begin{itemize}
        \item If $\| w^k - w^* \|_Q \geq \varepsilon_1 - c \delta$, then $\| w^{k+1} - w^* \|_Q \leq \tilde{\rho} \|w^k - w^* \|_Q$;
        \item If $\| w^k - w^* \|_Q < \varepsilon_1 - c \delta$, then $\| w^{k+1} - w^* \|_Q < \varepsilon_1$.
    \end{itemize}
    Hence, all the iterates $\{w^k: k \geq L - 1\}$ with
    \begin{equation}\label{eq:K-def}
        L = \left\lceil \log_{\tilde{\rho}} \frac{\varepsilon_1 - c \delta}{\left\| w^0 - w^* \right\|_Q} \right\rceil + 2 \leq \left\lceil \left( \log \frac{1 + \rho}{2} \right)^{-1} \log \frac{\varepsilon_1}{(2 - \rho) \left\| w^0 - w^* \right\|_Q} \right\rceil + 2 = O\left(\log\left(\frac{1}{\varepsilon_1}\right)\right)
    \end{equation}
    satisfy $\left\| w^k - w^* \right\|_Q < \varepsilon_1$.
    Since $\rho$ does not depend on $\delta$, the constant in $O\left(\log\left(1 / \varepsilon_1\right)\right)$ does not depend on $\delta$.
    
    Next, we bound $\left\| w^k - w^* \right\|_{U \times Y}$ from $\left\| w^{k-1} - w^* \right\|_Q$ for any $\varepsilon_1$ satisfying \eqref{eq:tilde-eps1-def}.
    Note that $Q_S \succeq I + SN^{-1}S^*$ and hence $\sigma_0(Q_S) \geq 1$, where $\sigma_0(Q_S)$ is the infimum of the spectrum of $Q_S$.
    Clearly, $Q_S$ is linear, bounded, self-adjoint, and coercive, hence it admits a bounded inverse $Q_S^{-1}$ with $\|Q_S^{-1}\| = \sigma_0(Q_S)^{-1}$.
    By the above discussions, there exists $L = O(\log(1/\varepsilon_1))$ defined in \eqref{eq:K-def} such that for all $k \geq L - 1$,
    \begin{equation}\label{eq:pf-epsopt-aux-1}
        \begin{aligned}
            &\left\| u^k - u^* \right\|_{Q_A - N} \leq \left\| w^k - w^* \right\|_{Q} < \varepsilon_1, \\
            &\left\|p^k - p^* \right\|_Y \leq \sigma_0(Q_S)^{-1/2} \left\| p^k - p^* \right\|_{Q_S} \leq \sigma_0(Q_S)^{-1/2} \left\|w^k - w^*\right\|_Q < \sigma_0(Q_S)^{-1/2} \varepsilon_1.
        \end{aligned}
    \end{equation}
    Since $Q_A - N$ is positive semidefinite, $Q_A$ is positive definite.
    Applying \Cref{lem:nresolv-lip} to $Q_A$ shows that $(Q_A + \partial \theta)^{-1}$ is Lipschitz with constant $\sigma_0(Q_A)^{-1} \leq c_0^{-1}$.
    Thus, from \eqref{eq:inexact-inexact-uzawa-bar-u} we have
    \begin{equation}\label{eq:pf-epsopt-aux-2}
        \begin{aligned}
            \left\|\bar{u}^{k+1} - u^*\right\|_U &\leq \sigma_0(Q_A)^{-1} \left( \left\|(Q_A - N) (u^k - u^*)\right\|_U + \left\|S^*(p^k - p^*)\right\|_U \right) \\
            &\leq \sigma_0(Q_A)^{-1} \left( \left\|Q_A - N\right\|^{1/2} \left\|u^k - u^*\right\|_{Q_A - N} + \left\|S\right\| \left\|p^k - p^*\right\|_Y \right) \\
            &< \sigma_0(Q_A)^{-1} \left( \left\|Q_A - N\right\|^{1/2} + \sigma_0(Q_S)^{-1/2} \left\|S\right\| \right) \varepsilon_1.
        \end{aligned}
    \end{equation}
    Therefore, for any $k \geq L$, it follows from \eqref{eq:pf-epsopt-aux-1}, \eqref{eq:pf-epsopt-aux-2} and \eqref{eq:inexact-inexact-uzawa-u}--\eqref{eq:inexact-inexact-uzawa-p} that
    \begin{equation}\label{eq:final-iter-err-bd}
        \begin{aligned}
            \left\| w^k - w^* \right\|_{U \times Y} &\leq \left\| \bar{w}^k - w^* \right\|_{U \times Y} + \left\| \bar{w}^k - w^k \right\|_{U \times Y} \\
            &= \left( \left\| \bar{u}^k - u^* \right\|_{U}^2 + \left\| \bar{p}^k - p^* \right\|_{Y}^2 \right)^{1/2} + \left( \left\| \bar{u}^k - u^k \right\|_{U}^2 + \left\| \bar{p}^k - p^k \right\|_{Y}^2 \right)^{1/2} \\
            &< \left( \sigma_0(Q_A)^{-2} \left( \left\|Q_A - N\right\|^{1/2} + \sigma_0(Q_S)^{-1/2} \left\|S\right\|  \right)^2 +  \sigma_0(Q_S)^{-1}  \right)^{1/2} \varepsilon_1 + \sqrt{2} \delta \\
            &< \left( \left( \sigma_0(Q_A)^{-2} \left( \left\|Q_A - N\right\|^{1/2} + \sigma_0(Q_S)^{-1/2} \left\|S\right\|  \right)^2 +  \sigma_0(Q_S)^{-1}  \right)^{1/2} + \frac{\sqrt{2}(1-\rho)}{(3-\rho) c} \right) \varepsilon_1 \\
            &\leq \left( \left( c_0^{-2} \left( \left\|Q_A - N\right\|^{1/2} + \left\|S\right\|  \right)^2 + 1 \right)^{1/2} + \frac{\sqrt{2}(1-\rho)}{(3-\rho) c} \right) \varepsilon_1 \\
            &:= C_1 \varepsilon_1,
        \end{aligned}
    \end{equation}
    where $C_1 > 0$ is a constant independent of $\varepsilon_1$ and $\delta$.

    For any $\varepsilon > 0$, take
    \begin{equation*}
        \varepsilon_1 = \varepsilon / C_1, \quad \delta_0 = \frac{1-\rho}{(3-\rho) c C_1} \varepsilon =: C \varepsilon.
    \end{equation*}
    Clearly $C > 0$ is independent of $\varepsilon$.
    It is easy to verify that any $\delta < \delta_0$ satisfies \eqref{eq:tilde-eps1-def}, hence \eqref{eq:final-iter-err-bd} implies for any $k \geq L$ that $\left\| w^k - w^* \right\|_{U \times Y} < \varepsilon$, where $L$ is defined as in \eqref{eq:K-def}.
    Finally, recall that the constant in the big-O notation in \eqref{eq:K-def} is independent of $\delta$, which concludes the proof.
\end{proof}

\begin{remark}\label{rem:bigo-const-dep}
    Note that \eqref{eq:perturbed-lin-conv-rate} remains valid even when $\delta = 0$, which verifies the linear convergence rate of \eqref{eq:inexact-uzawa} for the specific model \eqref{eq:opt-ctrl}.
    Note from \eqref{eq:K-def} and \eqref{eq:final-iter-err-bd} that the constant in $L = O(\log(1/\varepsilon))$ and $C$ depend only on $\|w^0 - w^* \|_Q$ and the choice of the preconditioners $Q_A$ and $Q_S$ in \eqref{eq:inexact-inexact-uzawa}. Moreover, $L$ and $C$ are uniform for all $w^0$ with $\|w^0 - w^* \|_Q$ bounded by a constant. This observation will be crucial for establishing the uniformity of the constants $C$ and $L$ in \Cref{thm:iuzawa-net-univ-approx-rec}, the asymptotic $\varepsilon$-optimality result of the \texttt{iUzawa-Net}.
\end{remark}

\subsection{Approximation Sequence of FNOs with Compact Containment}

The key step for allowing weight tying is to ensure that the images of each layer stay in a common compact set for any target approximation accuracy, so that the approximation domain of the next layer can be chosen independently.
By \Cref{prop:qsk-univ-approx}, the operator $\mathcal Q_S^k$ can be realized as a positive definite operator defined on $L^2(\Omega)$ whose inverse is sufficiently positive definite if the hyperparameter $\gamma$ in \eqref{eq:qsk-def} is sufficiently small, which is already adequate for our purposes.
In this subsection and \Cref{se:QAk-univ-approx}, we address the mentioned requirement by constructing sequences of neural networks of the forms \eqref{eq:fno} and \eqref{eq:qak-ptwise} that converge to $S$, $S^*$, and $(N + \tau I + \partial \theta)^{-1}$ uniformly on compact sets in $L^2(\Omega)$.
For any compact set, we further show that its images under the constructed sequence are contained in a common compact set.
The following observation from \cite{kovachki2023neural} is crucial for our analysis.

\begin{lemma}[cf. {\cite[Lemma 21]{kovachki2023neural}}]\label{prop:unif-approx-compact}
    Let $K \subset L^2(\Omega)$ be compact.
    Suppose that $G: L^2(\Omega) \to L^2(\Omega)$ and $G_j: L^2(\Omega) \to L^2(\Omega)$, $j \geq 1$ are continuous mappings such that
    \begin{equation}\label{eq:unif-approx-compact-conv}
        \lim_{j \to \infty} \sup_{u \in K} \left\| G(u) - G_j(u) \right\|_{L^2(\Omega)} = 0.
    \end{equation}
    Then the set $K^+ := G(K) \cup \bigcup_{j \geq 1} G_j(K)$ is compact in $L^2(\Omega)$.
\end{lemma}

The following proposition strengthens \Cref{prop:fno-univ-approx-general} in the $L^2(\Omega)$ setting.
It shows that for any continuous operator, one can construct sequences of FNOs converging uniformly on compact sets, and for any compact set, its images under the entire sequence are contained in a common compact set.
In particular, it applies to $S$ and $S^*$.

\begin{proposition}\label{lem:fno-seq-univ-approx}
    Let $G: L^2(\Omega) \to L^2(\Omega)$ be a continuous operator.
    Then there exists a sequence of FNOs $\{\cG_j\}_{j \geq 1}$ of the form \eqref{eq:fno}--\eqref{eq:fno-layers} with the zero extension operator $\cE$, such that for any compact subset $K \subset L^2(\Omega)$,
    \begin{equation}\label{eq:fno-seq-univ-approx}
        \lim_{j \to \infty} \sup_{u \in K} \left\| G(u) - \cG_j(u) \right\|_{L^2(\Omega)} = 0.
    \end{equation}
    Moreover, the set $K^+ := G(K) \cup \bigcup_{j \geq 1} \cG_j(K)$ is compact in $L^2(\Omega)$.
\end{proposition}

\begin{proof}
    For any integer $j \geq 1$, let $\Pi_j := \cF^{-1} \circ \cF_j$, i.e., the truncation of functions in $L^2(\bT^d)$ to their lowest $j$ Fourier modes.
    Define
    \begin{equation*}
        D_j := \left\{ v \in \Pi_j \big( L^2(\bT^d) \big) \, : \, \| v \|_{L^2(\bT^d)} \leq j \right\}.
    \end{equation*}
    Since the range of $\Pi_j$ is finite-dimensional, $D_j$ is compact in $L^2(\bT^d)$.
    Note also that the operator $v \mapsto \cE(G(v|_\Omega))$ is continuous from $L^2(\bT^d)$ to $L^2(\bT^d)$.
    Hence, by \cite[Theorem 5]{kovachki2021universal}, whose proof yields an FNO with finitely many Fourier modes (see \cite[proof of Theorem 15]{kovachki2021universal}), for each $j \geq 1$, there exists an FNO $\tilde{\cG}_j = \mathcal{Q} \circ \cK_{\text{FNO}} \circ \mathcal{P}: L^2(\bT^d) \to L^2(\bT^d)$ with Fourier layers of the form \eqref{eq:fno-layers}, such that
    \begin{equation}\label{eq:fno-seq-pf-torus}
        \sup_{v \in D_j} \left\| \cE(G(v|_\Omega)) - \tilde{\cG}_j(v) \right\|_{L^2(\bT^d)} < \frac{1}{j}.
    \end{equation}
    We define $\cG_j(u) := \left. \tilde{\cG}_j \left( \Pi_j \cE(u) \right) \right|_\Omega$ for all $u \in L^2(\Omega)$.

    We first show that each $\cG_j$ is an FNO of the form \eqref{eq:fno}--\eqref{eq:fno-layers}.
    Since $\mathcal{P}$ acts pointwise and linearly, we have $\mathcal{P} \circ \Pi_j = \Pi_j \circ \mathcal{P}$.
    Hence, for any $v \in L^2(\bT^d)$, the input $v^{(0)} = \mathcal{P}(\Pi_j v)$ of the first Fourier layer of $\tilde{\cG}_j$ satisfies
    \begin{equation*}
        \cF^{-1} \left( \cR_0 \cdot \cF_{k_\text{max}} \left(v^{(0)}\right) \right) + W_0 v^{(0)} = \cF^{-1} \left( \tilde{\cR}_0 \cdot \cF_{\tilde{k}_\text{max}} \left(\mathcal{P}(v)\right) \right),
    \end{equation*}
    where $\tilde{k}_\text{max} := \max\{k_\text{max}, j\}$, and
    \begin{equation*}
        \tilde{\cR}_0(k) := \begin{cases}
            \cR_0(k) + W_0, & \text{if } |k|_\infty \leq \min\{k_\text{max}, j\}, \\
            W_0, & \text{if } k_\text{max} < |k|_\infty \leq j, \\
            0, & \text{otherwise}.
        \end{cases}
    \end{equation*}
    Therefore, $\tilde{\cG}_j \circ \Pi_j$ is the FNO obtained from $\tilde{\cG}_j$ by replacing $k_\text{max}$, $\cR_0$, and $W_0$ with $\tilde{k}_\text{max}$, $\tilde{\cR}_0$, and $0$, respectively, and extending $\cR_\ell$ by zero for $\ell \geq 1$.
    In particular, $\cG_j = \left. (\tilde{\cG}_j \circ \Pi_j)(\cE(\cdot)) \right|_\Omega$ is of the form \eqref{eq:fno}--\eqref{eq:fno-layers}.

    We now prove \eqref{eq:fno-seq-univ-approx}.
    Fix any compact $K \subset L^2(\Omega)$.
    By Parseval's identity, for any $j \geq 1$, the mapping $v \mapsto \| v - \Pi_j v \|_{L^2(\bT^d)}$ is continuous on $L^2(\bT^d)$.
    Moreover, for any $v \in L^2(\bT^d)$, $\| v - \Pi_j v \|_{L^2(\bT^d)}$ is nonincreasing in $j$ and tends to zero as $j \to \infty$.
    Since $\cE(K)$ is compact, it follows from Dini's theorem that
    \begin{equation}\label{eq:fno-seq-pf-unif-cont-pre}
        0 \leq \lim_{j \to \infty} \sup_{u \in K} \left\| \left. (\Pi_j \cE(u)) \right|_\Omega - u \right\|_{L^2(\Omega)} \leq \lim_{j \to \infty} \sup_{u \in K} \left\| \Pi_j \cE(u) - \cE(u) \right\|_{L^2(\bT^d)} = 0.
    \end{equation}
    Since the mappings $u \mapsto (\Pi_j \cE(u))|_\Omega$ are continuous on $L^2(\Omega)$, using \Cref{prop:unif-approx-compact} with $G$ being the identity and \eqref{eq:fno-seq-pf-unif-cont-pre}, the set $K^+_0 := K \cup \bigcup_{j \geq 1} \left\{ \left. (\Pi_j \cE(u)) \right|_\Omega : u \in K \right\}$ is compact.
    In particular, $G$ is uniformly continuous on $K^+_0$.
    Consequently,
    \begin{equation}\label{eq:fno-seq-pf-unif-cont}
        \lim_{j \to \infty} \sup_{u \in K} \left\| G\left( \left. (\Pi_j \cE(u)) \right|_\Omega \right) - G(u) \right\|_{L^2(\Omega)} = 0.
    \end{equation}
    Note that for any $j \geq \| \cE \| \sup_{u \in K} \|u\|_{L^2(\Omega)}$ and $u \in K$, we have $v := \Pi_j \cE(u) \in D_j$.
    Fix this $v$ for the subsequent analysis.
    Since $\cE(w)|_\Omega = w$ for all $w \in L^2(\Omega)$, it follows from \eqref{eq:fno-seq-pf-torus} that
    \begin{equation*}
        \left\| G(v|_\Omega) - \cG_j(u) \right\|_{L^2(\Omega)} = \left\| \left. \left( \cE(G(v|_\Omega)) - \tilde{\cG}_j(v) \right) \right|_\Omega \right\|_{L^2(\Omega)} < \frac{1}{j}.
    \end{equation*}
    Hence, for all sufficiently large $j$, we have
    \begin{equation*}
        \sup_{u \in K} \left\| G(u) - \cG_j(u) \right\|_{L^2(\Omega)} \leq \frac{1}{j} + \sup_{u \in K} \left\| G\left( \left. (\Pi_j \cE(u)) \right|_\Omega \right) - G(u) \right\|_{L^2(\Omega)},
    \end{equation*}
    which, together with \eqref{eq:fno-seq-pf-unif-cont}, yields \eqref{eq:fno-seq-univ-approx}.

    Finally, each $\cG_j$ is continuous on $L^2(\Omega)$, since $\cE$, $\Pi_j$, and the restriction to $\Omega$ are bounded linear operators, and the FNO $\tilde{\cG}_j$ is continuous on $L^2(\bT^d)$ for its Lipschitz continuous activation functions.
    Therefore, the compactness of $K^+$ follows from \eqref{eq:fno-seq-univ-approx} and \Cref{prop:unif-approx-compact}.
\end{proof}

\subsection{Approximation Sequence of \texorpdfstring{$\cQ_A^k$}{QAk} with Compact Containment}\label{se:QAk-univ-approx}

Similar to the previous subsection, we prove that under \Cref{assump:ptwise} and the mild condition \Cref{assump:subgrad-gcont}, there exists a sequence of neural networks of the form $\cQ_A^k$ in \eqref{eq:qak-ptwise}-\eqref{eq:nn-rec-structure} that converges to $(N + \tau I + \partial \theta)^{-1}$ uniformly on compact sets. Then, using \Cref{prop:unif-approx-compact}, we conclude that the images of any compact set under this sequence do not escape a common compact set.
We begin with the following continuity assumption on the subgradient of $\psi_\xi$ defined in \Cref{assump:ptwise}, which is required to apply \Cref{prop:fc-univ-approx} in the subsequent analysis.

\begin{assumption}\label{assump:subgrad-gcont}
    Let $E$ and $\psi$ be defined as in \Cref{assump:ptwise}.
    Denote $\mathrm{gph} \, \partial \psi_{\xi}$ as the graph of $\partial \psi_{\xi}$ \cite{rockafellar1998variational}.
    We assume that the set-valued mapping $\xi \mapsto \mathrm{gph} \, \partial \psi_{\xi}$ is continuous; or equivalently,
    for any $\xi \in E$ and any sequence $\{\xi_n\} \subset E$ with $\xi_n \to \xi$, the set-valued mappings $\partial \psi_{\xi_n}$ converge to $\partial \psi_\xi$ in the sense of graph limit, i.e.,
    \begin{equation}\label{eq:subgrad-glim}
        \glimsupn \partial \psi_{\xi_n} \subset \partial \psi_\xi \subset \gliminfn \partial \psi_{\xi_n},
    \end{equation}
    where the graphical outer limit $\glimsup$ and graphical inner limit $\gliminf$ are defined as in \cite[Definition 5.32]{rockafellar1998variational}\footnote{For any two set-valued mappings $\varphi_1: X \rightrightarrows Y$ and $\varphi_2: X \rightrightarrows Y$, we follow the convention in \cite{rockafellar1998variational} and denote $\varphi_1 \subset \varphi_2$ if and only if $\varphi_1(x) \subset \varphi_2(x)$ for any $x \in X$.}.
\end{assumption}

\begin{remark}\label{rem:subgrad-gcont}
    The notations of graphical convergence and graphical outer/inner limit are standard and widely adopted in variational analysis.
    The condition \eqref{eq:subgrad-glim} is a mild assumption that requires $\psi_\xi$ to be parameterized by $\xi$ in a graphically continuous manner with respect to its subdifferential.
    Since $\psi_\xi$ is an extended real-valued function on $\bR$, \eqref{eq:subgrad-glim} is verifiable for many practical cases, including the examples given in \Cref{sec:qa-special}.
\end{remark}

We claim that the mapping $(N + \tau I + \partial \theta)^{-1}$ acts pointwise on its input.
Furthermore, under \Cref{assump:subgrad-gcont}, we show that this pointwise action is continuous by some technical variational analysis arguments.

\begin{lemma}\label{lem:nresol-ptwise}
    Assume that $N$ and $\theta$ satisfy \Cref{assump:ptwise} and let
    $\lambda$, $c_0$, $E$ and $\mu$ be defined as therein.
    Then there exists a function $\phi: \bR \times E \times [c_0, \infty) \to \bR$, such that for any $v \in L^2(\Omega)$, we have
    \begin{equation}\label{eq:ptwise-resol}
        (N + \tau I + \partial \theta)^{-1}(v)(x) = \phi(v(x), \mu(x), \lambda(x)), \quad ~\text{a.e.~in}~\Omega,
    \end{equation}
    and $\phi\left(v(\cdot), \mu(\cdot), \lambda(\cdot)\right) \in L^2(\Omega)$.
    Moreover, $\phi(s, \xi, \eta)$ is Lipschitz continuous in the variable $s$ with a uniform Lipschitz constant for all $(\xi, \eta)^\top \in E \times [c_0, \infty)$.
    If \Cref{assump:subgrad-gcont} holds in addition, then $\phi$ is continuous on $\bR \times E \times [c_0, \infty)$.
\end{lemma}

\begin{proof}
    Fix any $\bar{v} \in L^2(\Omega)$ and denote $\bar{u} = (N + \tau I + \partial \theta)^{-1}(\bar{v})$.
    We have
    \begin{equation*}
        \bar{u} = \mathrm{arg\,min}_{u \in L^2(\Omega)} \left\{  \int_\Omega \frac{1}{2} \left(\lambda(x) + \tau\right) |u(x)|^2 - \bar{v}(x) u(x) + \psi_{\mu(x)}(u(x)) \, dx \right\}.
    \end{equation*}
    Clearly, the above optimization problem can be equivalently decomposed into pointwise minimization problems, i.e., $\bar{u}(x)$ satisfies
    \begin{equation*}
        \bar{u}(x) = \mathrm{arg\,min}_{r \in \bR} \left\{ \frac{1}{2} \left(\lambda(x)+ \tau \right) r^2 - \bar{v}(x) r + \psi_{\mu(x)}(r) \right\}, \quad\text{a.e.~in}~\Omega.
    \end{equation*}
    Thus, $(N + \tau I + \partial \theta)^{-1}(\bar{v})(x)$ satisfies \eqref{eq:ptwise-resol} with
    \begin{equation*}
        \phi(s, \xi, \eta) = \mathrm{arg\,min}_{r \in \bR} \left\{ \frac{1}{2} (\eta + \tau) r^2 - s r + \psi_\xi (r) \right\}, \quad \text{for }s \in \bR, \xi \in E, \eta \geq c_0.
    \end{equation*}
    We have $\phi\left(v(\cdot), \mu(\cdot), \lambda(\cdot)\right) \in L^2(\Omega)$ since $(N + \tau I + \partial \theta)^{-1}$ maps $L^2(\Omega)$ into $L^2(\Omega)$.
    Moreover, note that
    \begin{equation}\label{eq:phi-prox-form}
        \phi(s, \xi, \eta) = \left((\eta + \tau)I + \partial \psi_\xi \right)^{-1}(s).
    \end{equation}
    Hence, applying \Cref{lem:nresolv-lip} to $(\eta + \tau)I$ and $\psi_\xi$, we see that $\phi(s, \xi, \eta)$ is Lipschitz continuous in $s$ for all $(\xi, \eta)^\top \in E \times [c_0, \infty)$ with a uniform Lipschitz constant $(c_0 + \tau)^{-1}$.

    Denote the set-valued mappings $\Xi_\xi = \partial \psi_{\xi}$ and $H_\eta = (\eta + \tau) I$ for any $\xi \in E$ and $\eta \geq c_0$.
    Thus \eqref{eq:subgrad-glim} becomes
    \begin{equation}\label{eq:subgrad-glim-concise}
        \glimsupn \Xi_{\xi_n} \subset \Xi_{\xi} \subset \gliminfn \Xi_{\xi_n},
    \end{equation}
    and \eqref{eq:phi-prox-form} becomes
    \begin{equation}\label{eq:phi-prox-form-concise}
        \phi(s, \xi, \eta) = (\Xi_{\xi} + H_{\eta})^{-1}(s).
    \end{equation}
    We now show that under \eqref{eq:subgrad-glim-concise}, for any $(\bar{\xi}, \bar{\eta})^\top \in E \times [c_0, \infty)$ and any sequence $\{(\xi_n, \eta_n)^\top\} \subset E \times [c_0, \infty)$ with $(\xi_n, \eta_n)^\top \to (\bar{\xi}, \bar{\eta})^\top$, we have
    \begin{equation}\label{eq:sum-glim}
        \glimsupn \left(\Xi_{\xi_n} + H_{\eta_n} \right) \subset \Xi_{\bar{\xi}} + H_{\bar{\eta}} \subset \gliminfn \left(\Xi_{\xi_n} + H_{\eta_n} \right).
    \end{equation}
    By \cite[Proposition 5.33]{rockafellar1998variational}, for any $\bar{r} \in \bR$, we have
    \begin{equation}\label{eq:glim-sup-inf-ptwise-eval}
    \begin{aligned}
        &\left( \glimsupn \left(\Xi_{\xi_n} + H_{\eta_n} \right) \right) (\bar{r}) = \bigcup_{\{r_n \to \bar{r}\}} \limsup_{n \to \infty} \left( \Xi_{\xi_n} (r_n) + H_{\eta_n}(r_n) \right), \\
        &\left( \gliminfn \left(\Xi_{\xi_n} + H_{\eta_n} \right) \right) (\bar{r}) = \bigcup_{\{r_n \to \bar{r}\}} \liminf_{n \to \infty} \left( \Xi_{\xi_n} (r_n) + H_{\eta_n}(r_n) \right),
    \end{aligned}
    \end{equation}
    where the unions are taken over all sequences $\{r_n\} \subset \bR$ such that $r_n \to \bar{r}$, and the outer limit $\limsup$ and the inner limit $\liminf$ are defined in \cite[Definition 4.1]{rockafellar1998variational}.
    Since $(\eta_n + \tau) r_n \to (\bar{\eta} + \tau) \bar{r}$ for any $r_n \to \bar{r}$, by the definition of $H_{\eta}$, we have 
    \begin{equation}\label{eq:setvalm-h-lim}
        \limsup_{n \to \infty} H_{\eta_n}(r_n) = \liminf_{n \to \infty} H_{\eta_n}(r_n) = H_{\bar{\eta}}(\bar{r}).
    \end{equation}
    Note that $H_{\eta_n}(r_n) = \{ (\eta_n + \tau) r_n \}$ is a singleton for each $n$.
    Hence, for any $r_n \to \bar{r}$ and any subsequence $\{n_j\}$, any sequence $\{ z_{n_j} + (\eta_{n_j} + \tau) r_{n_j} \}$ with $z_{n_j} \in \Xi_{\xi_{n_j}}(r_{n_j})$ converges if and only if $\{ z_{n_j} \}$ converges, and the two limits differ by $(\bar{\eta} + \tau) \bar{r}$.
    By the definitions of the outer and inner limits \cite[Definition 4.1]{rockafellar1998variational} and \eqref{eq:setvalm-h-lim}, we obtain that for any $r_n \to \bar{r}$,
    \begin{equation}\label{eq:limsup-inclusion}
        \limsup_{n \to \infty} \left( \Xi_{\xi_n} (r_n) + H_{\eta_n} (r_n) \right) \mathrel{=} H_{\bar{\eta}}(\bar{r}) + \limsup_{n \to \infty} \Xi_{\xi_n}(r_n),
    \end{equation}
    and
    \begin{equation}\label{eq:liminf-inclusion}
        \liminf_{n \to \infty} \left( \Xi_{\xi_n} (r_n) + H_{\eta_n} (r_n) \right) \mathrel{=} H_{\bar{\eta}}(\bar{r}) + \liminf_{n \to \infty} \Xi_{\xi_n}(r_n).
    \end{equation}
    Take the union over all sequences $r_n \to \bar{r}$ in \eqref{eq:limsup-inclusion} and \eqref{eq:liminf-inclusion}.
    It follows from \eqref{eq:subgrad-glim-concise}, \eqref{eq:glim-sup-inf-ptwise-eval}, and \cite[Proposition 5.33]{rockafellar1998variational} that
    \begin{equation*}
        \begin{aligned}
            \left( \glimsupn \left(\Xi_{\xi_n} + H_{\eta_n} \right) \right) (\bar{r}) &\subset \bigcup_{\{r_n \to \bar{r}\}} \left( H_{\bar{\eta}}(\bar{r}) + \limsup_{n \to \infty} \Xi_{\xi_n}(r_n) \right) \\
            &= H_{\bar{\eta}}(\bar{r}) + \bigcup_{\{r_n \to \bar{r}\}} \limsup_{n \to \infty} \Xi_{\xi_n}(r_n) \\
            &= H_{\bar{\eta}}(\bar{r}) + \left( \glimsupn \Xi_{\xi_n} \right) (\bar{r}) \\
            &\subset  \left( \Xi_{\bar{\xi}} + H_{\bar{\eta}} \right) (\bar{r}),
        \end{aligned}
    \end{equation*}
    and similarly,
    \begin{equation*}
        \left( \gliminfn \left(\Xi_{\xi_n} + H_{\eta_n} \right) \right) (\bar{r}) \supset \left( \Xi_{\bar{\xi}} + H_{\bar{\eta}} \right) (\bar{r}),
    \end{equation*}
    which proves \eqref{eq:sum-glim}.

    We are now ready to establish the continuity of $\phi$ at any $(\bar{s}, \bar{\xi}, \bar{\eta})^\top \in \bR \times E \times [c_0, \infty)$ under \eqref{eq:subgrad-glim} (or equivalently, under \eqref{eq:subgrad-glim-concise}).
    Recall that $\phi$ takes the form \eqref{eq:phi-prox-form-concise}.
    Since $\psi_\xi$ is convex, lower semicontinuous, and proper for any $\xi \in E$, $\Xi_{\bar{\xi}} = \partial \psi_{\bar{\xi}}$ is closed valued \cite[Section 23]{rockafellar1970convex}.
    By \eqref{eq:sum-glim} and \cite[Theorem 5.37]{rockafellar1998variational}, we have for any $(\xi_n, \eta_n)^\top \to (\bar{\xi}, \bar{\eta})^\top$ that
    \begin{equation*}
        \limsup_{n \to \infty} \left(\Xi_{\xi_n} + H_{\eta_n} \right)^{-1}(\bar{s}) \subset \left(\Xi_{\bar{\xi}} + H_{\bar{\eta}} \right)^{-1}(\bar{s}).
    \end{equation*}
    Hence, by \eqref{eq:phi-prox-form-concise} and the definition of $\limsup$, any limit point of the sequence $\{\phi(\bar{s}, \xi_n, \eta_n)\}_{n=1}^\infty$ is $\phi(\bar{s}, \bar{\xi}, \bar{\eta})$.
    We claim that $\{\phi(\bar{s}, \xi_n, \eta_n)\}_{n=1}^\infty$ is bounded.
    If not, there exists an unbounded subsequence of $\{\phi(\bar{s}, \xi_n, \eta_n)\}_{n=1}^\infty$, and assume without loss of generality that $|\phi(\bar{s}, \xi_n, \eta_n)| \to \infty$.
    Again, by \cite[Theorem 5.37]{rockafellar1998variational},
    \begin{equation*}
        \left( \Xi_{\bar{\xi}} + H_{\bar{\eta}} \right)^{-1}(\bar{s}) \subset \bigcap_{\omega > 0} \liminf_{n \to \infty} \left(\Xi_{\xi_n} + H_{\eta_n} \right)^{-1}(\overline{B}(\bar{s}, \omega)),
    \end{equation*}
    where $\overline{B}(\bar{s}, \omega)$ denotes the closed ball centered at $\bar{s}$ with radius $\omega$.
    Hence by \eqref{eq:phi-prox-form-concise}, for any $\omega > 0$, there exists a sequence $\{s_n\}_{n=1}^\infty \subset \overline{B}(\bar{s}, \omega)$ such that $\phi(s_n, \xi_n, \eta_n) \to \phi(\bar{s}, \bar{\xi}, \bar{\eta})$ as $n \to \infty$.
    In particular, there exists $s_n \to \bar{s}$ such that $\phi(s_n, \xi_n, \eta_n) \to \phi(\bar{s}, \bar{\xi}, \bar{\eta})$.
    Together with the Lipschitz continuity of $\phi(\cdot, \xi_n, \eta_n)$ and the uniformity of the Lipschitz constant, this contradicts $|\phi(\bar{s}, \xi_n, \eta_n)| \to \infty$.
    Therefore, $\{\phi(\bar{s}, \xi_n, \eta_n)\}_{n=1}^\infty$ is bounded with every convergent subsequence converging to $\phi(\bar{s}, \bar{\xi}, \bar{\eta})$, which implies that
    \begin{equation}\label{eq:phi-continuity}
        \lim_{n \to \infty} \phi(\bar{s}, \xi_n, \eta_n) = \phi(\bar{s}, \bar{\xi}, \bar{\eta}).
    \end{equation}
    Since the above argument holds for any sequence $\{(\xi_n, \eta_n)^\top\} \subset E \times [c_0, \infty)$ with $(\xi_n, \eta_n)^\top \to (\bar{\xi}, \bar{\eta})^\top$, we conclude that $\phi(\bar{s}, \xi, \eta)$ is continuous in $(\xi, \eta)^\top$ on $E \times [c_0, \infty)$ for any $\bar{s} \in \bR$.
    Finally, the joint continuity of $\phi$ follows from the Lipschitz continuity in $s$ and the continuity in $(\xi, \eta)^\top$.
\end{proof}

\begin{lemma}\label{lem:ptwise-operator-univ-approx-reg}
    Suppose that \Cref{assump:ptwise} holds, and let $\lambda$, $\mu$, $E$ and $c_0$ be defined as therein.
    Let $\phi: \bR \times E \times [c_0, \infty) \to \bR$ be a continuous function such that $\phi(s, \xi, \eta)$ is Lipschitz continuous in $s$ with a uniform Lipschitz constant $L_\phi > 0$ for all $\xi \in E$ and $\eta \in [c_0, \infty)$.
    Assume that $G: L^2(\Omega) \to L^2(\Omega)$ satisfies
    \begin{equation}\label{eq:ptwise-operator-reg}
        G(u)(x) = \phi(u(x), \mu(x), \lambda(x)), \quad ~\text{a.e.~in}~\Omega, \quad \forall u \in L^2(\Omega).
    \end{equation}
    Then there exists a sequence of neural networks $\{\mathcal{N}_j\}_{j \geq 1}$ with $\mathcal{N}_j: \bR^{m+2} \to \bR$ defined in \eqref{eq:nn-rec-structure}, such that the neural networks $\cG_j: L^2(\Omega) \to L^2(\Omega)$ defined by
    \begin{equation}\label{eq:cg-def-reg}
        \cG_j ( u )(x) = \mathcal{N}_j\left(u(x), \mu(x), \lambda(x)\right), \quad ~\text{a.e.~in}~\Omega, \quad \forall u \in L^2(\Omega),
    \end{equation}
    satisfy, for any compact subset $K \subset L^2(\Omega)$,
    \begin{equation}\label{eq:ptwise-operator-univ-approx-reg}
        \lim_{j \to \infty} \sup_{u \in K} \left\| G(u) - \cG_j(u) \right\|_{L^2(\Omega)} = 0.
    \end{equation}
    Moreover, the set
    \begin{equation}\label{eq:ptwise-operator-cpt-img-reg}
        K^+ := G(K) \cup \bigcup_{j \geq 1} \cG_j(K)
    \end{equation}
    is compact in $L^2(\Omega)$.
\end{lemma}

\begin{proof}
    Denote by $\nu(x) = (\mu(x), \lambda(x))^\top \in L^\infty(\Omega; \bR^{m+1})$.
    For any $u \in L^2(\Omega)$ and $r > 0$, define $u_r: \Omega \to \bR$ (up to a measure zero set) by
    \begin{equation}\label{eq:u-trunc}
        u_r(x) = \mathrm{proj}_{[-r, r]}\left(u(x)\right) = \begin{cases}
            r, & u(x) > r, \\
            u(x), & |u(x)| \leq r, \\
            -r, & u(x) < -r,
        \end{cases} \quad \text{a.e.~in}~\Omega.
    \end{equation}
    Define the constants
    \begin{equation*}
        R_1 := \left\| \nu \right\|_{L^\infty(\Omega; \bR^{m+1})}, \quad R_0^{(j)} := R_1 + j, \quad \forall j \geq 1,
    \end{equation*}
    and the compact sets
    \begin{equation*}
        \cK_1^{(j)} := [-R_0^{(j)}, R_0^{(j)}] \subset \bR, \quad \cK_2 := ([-R_1, R_1]^m \cap E) \times [c_0, R_1] \subset \bR^{m+1}.
    \end{equation*}
    Note that $R_1 \leq R_0^{(j)}$ and $\nu(x) \in \cK_2$ a.e. in $\Omega$ by \Cref{assump:ptwise}.

    We first construct the neural networks $\cN_j$.
    Fix any $j \geq 1$.
    Since $\phi$ is continuous on the closed set $\bR \times E \times [c_0, \infty)$, and $\cK_1^{(j)} \times \cK_2$ is compact, by \Cref{lem:nn-rec-univ-approx}, there exists a neural network $\mathcal{N}_{0, j}: \bR^{m+2} \to \bR$ of the form \eqref{eq:nn-rec-structure}, such that
    \begin{equation}\label{eq:fcnn-linf-err-bd}
        \sup_{s \in \cK_1^{(j)}, \, \zeta \in \cK_2} \left| \phi(s, \zeta) - \mathcal{N}_{0, j}(s, \zeta) \right| < \frac{1}{2 j |\Omega|^{1/2}}.
    \end{equation}
    Define the neural network $\cN_j: \bR^{m+2} \to \bR$ with ReLU activation function $\sigma$ by
    \begin{equation*}
        \mathcal{N}_j(s, \zeta) = \mathcal{N}_{0, j} \left(R_0^{(j)}\mathbf{1} - \sigma \left( 2R_0^{(j)} \mathbf{1} - \sigma \left( (s, \zeta)^\top + R_0^{(j)} \mathbf{1} \right) \right)\right),
    \end{equation*}
    where $\mathbf{1} \in \bR^{m+2}$ is the all-one vector.
    Then $\cN_j$ also has the structure defined in \eqref{eq:nn-rec-structure}, and it follows from $R_1 \leq R_0^{(j)}$ that
    \begin{equation}\label{eq:fcnn-proj-property-reg}
        \mathcal{N}_j(s, \zeta) = \mathcal{N}_{0, j} \left( \mathrm{proj}_{\cK_1^{(j)}}(s), \zeta \right), \quad \forall s \in \bR, \ \zeta \in \cK_2.
    \end{equation}

    We now prove \eqref{eq:ptwise-operator-univ-approx-reg}.
    Let $\cG_j$ be defined by \eqref{eq:cg-def-reg}.
    For any $u \in L^2(\Omega)$, we have $u_{R_0^{(j)}}(x) = \mathrm{proj}_{\cK_1^{(j)}}(u(x)) \in \cK_1^{(j)}$ and $\nu(x) \in \cK_2$ a.e. in $\Omega$.
    Hence, it follows from \eqref{eq:fcnn-linf-err-bd}, \eqref{eq:fcnn-proj-property-reg}, and the Lipschitz continuity of $\phi$ in $s$ that
    \begin{equation*}
        \begin{aligned}
            \left| \phi(u(x), \nu(x)) - \cN_j(u(x), \nu(x)) \right| &\leq \left| \phi(u(x), \nu(x)) - \phi(u_{R_0^{(j)}}(x), \nu(x)) \right| \\
            & \qquad + \left| \phi(u_{R_0^{(j)}}(x), \nu(x)) - \cN_{0, j}(u_{R_0^{(j)}}(x), \nu(x)) \right| \\
            &< L_\phi \left| u(x) - u_{R_0^{(j)}}(x) \right| + \frac{1}{2 j |\Omega|^{1/2}}, \quad \text{a.e.~in}~\Omega.
        \end{aligned}
    \end{equation*}
    Squaring both sides, applying the inequality $(a + b)^2 \leq 2a^2 + 2b^2$, and integrating over $\Omega$, we obtain
    \begin{equation}\label{eq:ptwise-operator-glob-err-reg}
        \left\| G(u) - \cG_j(u) \right\|_{L^2(\Omega)}^2 \leq 2 L_\phi^2 \left\| u - u_{R_0^{(j)}} \right\|_{L^2(\Omega)}^2 + \frac{1}{2 j^2}, \quad \forall u \in L^2(\Omega), \ j \geq 1.
    \end{equation}
    Fix any compact $K \subset L^2(\Omega)$.
    Consider the mapping $\pi_r: L^2(\Omega) \to \bR$ defined by $\pi_r(u) = \left\|u - u_r \right\|_{L^2(\Omega)}$.
    By the dominated convergence theorem, we have $\lim_{r \to \infty} \pi_r(u) = 0$.
    Since the mapping $s \mapsto s - \mathrm{proj}_{[-r, r]}(s)$ is nonexpansive on $\bR$, we have $| \pi_r(u) - \pi_r(v) | \leq \| u - v \|_{L^2(\Omega)}$ for any $u, v \in L^2(\Omega)$.
    Hence $\pi_r$ is continuous.
    Moreover, $r \mapsto \pi_r(u)$ is clearly monotone for each $u \in L^2(\Omega)$.
    Since $K$ is compact, it follows from Dini's theorem that
    \begin{equation}\label{eq:pfaux-lpoua-dini}
        \lim_{r \to \infty} \sup_{u \in K} \left\| u - u_r \right\|_{L^2(\Omega)} = 0.
    \end{equation}
    Since $R_0^{(j)} \to \infty$ as $j \to \infty$, taking the supremum over $u \in K$ in \eqref{eq:ptwise-operator-glob-err-reg} and letting $j \to \infty$ yield \eqref{eq:ptwise-operator-univ-approx-reg}.

    For \eqref{eq:ptwise-operator-cpt-img-reg}, by \eqref{eq:ptwise-operator-reg}, the Lipschitz continuity of $\phi$ in $s$, and the uniformity of Lipschitz constant $L_\phi$, the mapping $G$ is Lipschitz continuous on $L^2(\Omega)$ with the Lipschitz constant $L_\phi$.
    Since each $\cN_j$ is a composition of affine mappings and the ReLU activation function, it is Lipschitz continuous on $\bR^{m+2}$ with some Lipschitz constant $L_{\cN_j} > 0$.
    It follows that
    \begin{equation*}
        \left\| \cG_j(u) - \cG_j(v) \right\|_{L^2(\Omega)}^2 = \int_\Omega \left| \cN_j(u(x), \nu(x)) - \cN_j(v(x), \nu(x)) \right|^2 \, dx \leq L_{\cN_j}^2 \left\| u - v \right\|_{L^2(\Omega)}^2, \quad \forall u, v \in L^2(\Omega),
    \end{equation*}
    which implies that each $\cG_j$ is continuous on $L^2(\Omega)$.
    Therefore, the compactness of $K^+$ follows from \eqref{eq:ptwise-operator-univ-approx-reg} and \Cref{prop:unif-approx-compact}.
\end{proof}

\begin{proposition}\label{prop:qak-ptwise-univ-approx-bddset}
    Let $N: L^2(\Omega) \to L^2(\Omega)$ and $\theta: L^2(\Omega) \to \bR \cup \{ +\infty \}$ satisfy \Cref{assump:ptwise,assump:subgrad-gcont}.
    Then there exists a sequence of neural networks $\{\cQ_A^{(j)}\}_{j \geq 1}$ of the form \eqref{eq:qak-ptwise} such that, for any compact subset $K \subset L^2(\Omega)$, we have
    \begin{equation}\label{eq:qak-seq-univ-approx-reg}
        \lim_{j \to \infty} \sup_{v \in K} \left\| (N + \tau I + \partial \theta)^{-1}(v) - \cQ_A^{(j)}(v) \right\|_{L^2(\Omega)} = 0,
    \end{equation}
    and the set
    \begin{equation}\label{eq:qak-seq-cpt-img-reg}
        K^+ := (N + \tau I + \partial \theta)^{-1}(K) \cup \bigcup_{j \geq 1} \cQ_A^{(j)}(K)
    \end{equation}
    is compact in $L^2(\Omega)$.
\end{proposition}

\begin{proof}
    By \Cref{assump:ptwise,assump:subgrad-gcont} and \Cref{lem:nresol-ptwise}, the operator $(N + \tau I + \partial \theta)^{-1}$ satisfies \eqref{eq:ptwise-resol} with a function $\phi$ that is continuous on $\bR \times E \times [c_0, \infty)$ and Lipschitz continuous in its first argument with a uniform Lipschitz constant.
    The desired result thus follows directly from \Cref{lem:ptwise-operator-univ-approx-reg} with $\nu(x) = (\mu(x), \lambda(x))^\top$ and $G = (N + \tau I + \partial \theta)^{-1}$.
\end{proof}

We emphasize that the compact set $K^+$ in \eqref{eq:qak-seq-cpt-img-reg} contains the image $\cQ_A^{(j)}(K)$ for every $j \geq 1$, and similarly, the compact set $K^+$ in \Cref{lem:fno-seq-univ-approx} contains the image $\cG_j(K)$ for every $j \geq 1$.
This property is crucial for constructing a weight tying \texttt{iUzawa-Net} in the subsequent subsections, since it allows us to enlarge the approximation domains after the approximating sequences of neural networks have been chosen.

\subsection{The iUzawa-Net as Inexact Iterations of \texorpdfstring{\eqref{eq:inexact-uzawa}}{the inexact Uzawa method}}

In this subsection, we show that the \texttt{iUzawa-Net} \eqref{eq:inexact-uzawa-duf} can be embedded into the abstract framework \eqref{eq:inexact-inexact-uzawa}.
Specifically, we first prove that a single layer of the \texttt{iUzawa-Net} with neural networks from the sequences in \Cref{lem:fno-seq-univ-approx} and \Cref{prop:qak-ptwise-univ-approx-bddset} maps compact sets in $L^2(\Omega)$ into compact sets in $L^2(\Omega)$ that are independent of the approximation accuracy.
As a consequence, the conditions \eqref{eq:inexact-inexact-uzawa-u}--\eqref{eq:inexact-inexact-uzawa-p} can be satisfied with arbitrary $\delta > 0$ by choosing sufficiently accurate neural network surrogates $\cQ_A$, $\cS$, and $\cA$ from these sequences for the corresponding operators.

\begin{proposition}\label{prop:iuzawa-net-err-reg}
    Let \Cref{assump:ptwise,assump:subgrad-gcont} hold.
    Moreover, let $\cQ_S$ be given by \Cref{prop:qsk-univ-approx}, $\{\cS^{(j)}\}_{j \geq 1}$ and $\{\cA^{(j)}\}_{j \geq 1}$ be the sequences of FNOs given in \Cref{lem:fno-seq-univ-approx} with $G = S$ and $G = S^*$, respectively, and $\{\cQ_A^{(j)}\}_{j \geq 1}$ be the sequence of neural networks given in \Cref{prop:qak-ptwise-univ-approx-bddset}.
    Then the following assertions hold:
    \begin{enumerate}[label=(\roman*)]
        \item[(i)] For any compact set $K \subset L^2(\Omega)$, there exists a compact \emph{critical set} $K_0 \subset L^2(\Omega)$ with $K \subset K_0$, such that the following assertion holds.
        Let $\cS = \cS^{(j)}$, $\cA = \cA^{(j)}$, and $\cQ_A = \cQ_A^{(j)}$ for some $j \geq 1$.
        For any $(y_d, f)^\top \in K \times K$ and $(u^k, p^k)^\top \in K \times K$, if $(u^{k+1}, p^{k+1})^\top$ is the $(k+1)$-th layer output of the iUzawa-Net \eqref{eq:inexact-uzawa-duf} with inputs $y_d$, $f$, $u^k$, $p^k$ and neural networks $\cS^k = \cS$, $\cA^k = \cA$, $\cQ_A^k = \cQ_A$ and $\cQ_S^k = \cQ_S$, then
    \begin{equation}\label{eq:iuzawa-net-output-bdd-reg}
        p^k \in K_0, \quad \tau u^k - \cA(p^k) \in K_0, \quad u^{k+1} + f \in K_0, \quad \text{and} \quad (u^{k+1}, p^{k+1})^\top \in K_0 \times K_0.
    \end{equation}
    Moreover, $K_0$ depends only on $K$, $\cQ_S$, and the sequences $\{\cS^{(j)}\}_{j \geq 1}$, $\{\cA^{(j)}\}_{j \geq 1}$, and $\{\cQ_A^{(j)}\}_{j \geq 1}$.
    In particular, it is independent of $j$.
        \item[(ii)] For any compact set $K' \subset L^2(\Omega)$ and $\delta > 0$, there exists an integer $j_0 \geq 1$, such that for any $j \geq j_0$ and $y_d, f, u^k, p^k \in K'$, if $\bar{w}^{k+1} = (\bar{u}^{k+1}, \bar{p}^{k+1})^\top$ is defined by \eqref{eq:inexact-inexact-uzawa-bar-u}--\eqref{eq:inexact-inexact-uzawa-bar-p} with $Q_A = N + \tau I$ and $Q_S = \cQ_S^{-1}$, and $w^{k+1} = (u^{k+1}, p^{k+1})^\top$ is the output of the $(k+1)$-th layer of the iUzawa-Net \eqref{eq:inexact-uzawa-duf} with inputs $y_d$, $f$, $u^k$, $p^k$ and neural networks $\cS^k = \cS^{(j)}$, $\cA^k = \cA^{(j)}$, $\cQ_A^k = \cQ_A^{(j)}$, $\cQ_S^k = \cQ_S$, then
        \begin{equation}\label{eq:iuzawa-net-err}
            \max \left\{ \|u^{k+1} - \bar{u}^{k+1}\|_{L^2(\Omega)}, \ \|p^{k+1} - \bar{p}^{k+1}\|_{L^2(\Omega)} \right\} < \delta.
        \end{equation}
    \end{enumerate}
\end{proposition}

\begin{proof}
    Fix any compact set $K \subset L^2(\Omega)$, and define $K_0 := K \cup K_{\cQ_A} \cup K_\cS \cup K_U^+ \cup K_Y^+$, where
    \begin{equation}\label{eq:iner-pf-critical-sets}
    \begin{aligned}
        &K_\cA^+ := S^*(K) \cup \bigcup_{j \geq 1} \cA^{(j)}(K), \qquad K_{\cQ_A} := \tau K - K_\cA^+, \\
        &K_U^+ := (N + \tau I + \partial \theta)^{-1}(K_{\cQ_A}) \cup \bigcup_{j \geq 1} \cQ_A^{(j)}(K_{\cQ_A}), \qquad K_\cS := K_U^+ + K, \\
        &K_\cS^+ := S(K_\cS) \cup \bigcup_{j \geq 1} \cS^{(j)}(K_\cS), \qquad K_Y^+ := \left\{ p + \cQ_S \left( q - p - y_d \right) \, : \, p \in K, \ y_d \in K, \ q \in K_\cS^+ \right\}.
    \end{aligned}
    \end{equation}
    By \Cref{lem:fno-seq-univ-approx}, $K_\cA^+$ is compact, and hence $K_{\cQ_A}$ is compact as the sum of compact sets. Then $K_U^+$ is compact by \Cref{prop:qak-ptwise-univ-approx-bddset}.
    Similarly, the set $K_\cS$ is compact, and hence $K_\cS^+$ is compact by \Cref{lem:fno-seq-univ-approx}.
    Since $\cQ_S$ is continuous on $L^2(\Omega)$, the set $K_Y^+$ is also compact.
    Therefore, $K_0$ is a compact subset of $L^2(\Omega)$ with $K \subset K_0$.
    Clearly, $K_0$ is determined by $K$, $\cQ_S$, and the sequences $\{\cS^{(j)}\}_{j \geq 1}$, $\{\cA^{(j)}\}_{j \geq 1}$, and $\{\cQ_A^{(j)}\}_{j \geq 1}$ only.

    Let $\cS = \cS^{(j)}$, $\cA = \cA^{(j)}$, and $\cQ_A = \cQ_A^{(j)}$ for some $j \geq 1$, and consider any $(y_d, f)^\top \in K \times K$ and $(u^k, p^k)^\top \in K \times K$.
    Since $p^k \in K$, we have $\cA(p^k) \in \cA^{(j)}(K) \subset K_\cA^+$, and hence
    \begin{equation*}
        \tau u^k - \cA(p^k) \in \tau K - K_\cA^+ = K_{\cQ_A} \subset K_0.
    \end{equation*}
    By the definition of $K_U^+$, we have
    \begin{equation*}
        u^{k+1} = \cQ_A \left( \tau u^k - \cA(p^k) \right) \in \cQ_A^{(j)}(K_{\cQ_A}) \subset K_U^+ \subset K_0,
    \end{equation*}
    which holds regardless of the choice of $j$ (i.e., the approximation accuracy of $\cQ_A$).
    Consequently, we have $u^{k+1} + f \in K_\cS \subset K_0$, and hence $\cS(u^{k+1} + f) \in \cS^{(j)}(K_\cS) \subset K_\cS^+$.
    It then follows from \eqref{eq:inexact-uzawa-duf-iter} that
    \begin{equation*}
        p^{k+1} = p^k + \cQ_S \left( \cS(u^{k+1} + f) - p^k - y_d \right) \in K_Y^+ \subset K_0.
    \end{equation*}
    This proves \eqref{eq:iuzawa-net-output-bdd-reg} and thus (i).

    Now fix arbitrary $\delta > 0$ and a compact set $K' \subset L^2(\Omega)$. Let $K_0' \subset L^2(\Omega)$ be a critical set of $K'$ given by (i).
    For any $j \geq 1$, denote
    \begin{equation*}
        \begin{aligned}
            &e_\cA(j) := \sup_{p \in K_0'} \left\| S^* p - \cA^{(j)}(p) \right\|_{L^2(\Omega)}, \qquad e_\cS(j) := \sup_{v \in K_0'} \left\| S v - \cS^{(j)}(v) \right\|_{L^2(\Omega)}, \\
            &e_{\cQ_A}(j) := \sup_{v \in K_0'} \left\| (N + \tau I + \partial \theta)^{-1}(v) - \cQ_A^{(j)}(v) \right\|_{L^2(\Omega)}.
        \end{aligned}
    \end{equation*}
    By \Cref{lem:fno-seq-univ-approx} and \Cref{prop:qak-ptwise-univ-approx-bddset}, we have $e_\cA(j)$, $e_{\cQ_A}(j)$, and $e_\cS(j)$ converging to zero as $j \to \infty$.
    To verify \eqref{eq:iuzawa-net-err}, fix any $j \geq 1$ and $y_d, f, u^k, p^k \in K'$.
    By (i) with $K'$ in place of $K$, we have $p^k \in K_0'$, $\tau u^k - \cA^{(j)}(p^k) \in K_0'$, and $u^{k+1} + f \in K_0'$.
    Applying \Cref{lem:nresolv-lip} to $N + \tau I$ yields that $(N + \tau I + \partial \theta)^{-1}$ is Lipschitz continuous with constant $(c_0 + \tau)^{-1}$.
    Hence, it follows from \eqref{eq:inexact-uzawa-duf-iter}, \eqref{eq:inexact-inexact-uzawa-bar-u}--\eqref{eq:inexact-inexact-uzawa-bar-p}, and the triangle inequality that
    \begin{equation*}
        \begin{aligned}
            \left\| u^{k+1} - \bar{u}^{k+1} \right\|_{L^2(\Omega)} &= \left\| \cQ_A^{(j)} \left( \tau u^k - \cA^{(j)}(p^k) \right) - (N + \tau I + \partial \theta)^{-1} \left( \tau u^k - S^* p^k \right) \right\|_{L^2(\Omega)} \\
            &\leq e_{\cQ_A}(j) + (c_0 + \tau)^{-1} e_\cA(j) =: a_j, \\
            \left\| p^{k+1} - \bar{p}^{k+1} \right\|_{L^2(\Omega)} &= \left\| \cQ_S \left( \cS^{(j)}(u^{k+1} + f) - S(\bar{u}^{k+1} + f) \right) \right\|_{L^2(\Omega)} \leq \|\cQ_S\| \left( e_\cS(j) + \|S\| a_j \right) =: b_j.
        \end{aligned}
    \end{equation*}
    Note that $a_j$ and $b_j$ are independent of $y_d$, $f$, $u^k$, and $p^k$, and taking $j \to \infty$ yields $a_j \to 0$ and $b_j \to 0$.
    This proves \eqref{eq:iuzawa-net-err} and thus (ii).
\end{proof}

\subsection{Asymptotic \texorpdfstring{$\varepsilon$}{epsilon}-Optimality of the iUzawa-Net Layer Outputs under Weight Tying}

We are now ready to establish the following asymptotic $\varepsilon$-optimality of the \texttt{iUzawa-Net} layer outputs, which can also be viewed as a new universal approximation property of the \texttt{iUzawa-Net}.

\begin{theorem}\label{thm:iuzawa-net-univ-approx-rec}
    Suppose that \Cref{assump:ptwise,assump:subgrad-gcont} hold.
    For any compact set $K \subset L^2(\Omega) \times L^2(\Omega)$ and $\varepsilon > 0$, there exist $\delta_0 = C \varepsilon$ with $C>0$ independent of $\varepsilon$ and $L_0 = O(\log(1/\varepsilon))$, such that for any $\delta < \delta_0$ and $L \geq L_0$, there exists an \texttt{iUzawa-Net} $\mathcal{T}(y_d, f; \theta_{\mathcal{T}})$ with $L$ layers that is algorithm tracking with respect to $K$ and $\delta$ in the sense of \Cref{def:alg-track} and weight tying in the sense of \Cref{def:weight-tying}.
    Moreover, for any $(y_d, f)^\top \in K$, the layer outputs $\{(u^k, p^k)^\top\}_{k=1}^L$ of $\cT(y_d, f; \theta_\cT)$ satisfy
    \begin{equation}\label{eq:iuzawa-net-conv-to-eps}
        \left\| T(y_d, f) - u^k \right\|_{L^2(\Omega)} < \varepsilon, \quad \forall k \geq L_0.
    \end{equation}
\end{theorem}

\begin{proof}
    Denote by $K_Y$ and $K_U$ the projections of $K$ onto the first and the second components, respectively, which are compact.
    By \Cref{prop:qsk-univ-approx} (ii) with $Q_S = I + SN^{-1}S^*$, there exists $\cQ_S$ of the architecture \eqref{eq:qsk-def} such that $\cQ_S$ is bounded, linear, self-adjoint, positive definite, and $\cQ_S^{-1} - (I + SN^{-1}S^*)$ is positive semidefinite.
    In view of \Cref{prop:inexact-inexact-uzawa-stability}, it suffices to construct $\cS$, $\cA$, and $\cQ_A$, such that the layer outputs of the \texttt{iUzawa-Net} \eqref{eq:inexact-uzawa-duf} satisfy \eqref{eq:inexact-inexact-uzawa} for $Q_A = N + \tau I$, $Q_S = \cQ_S^{-1}$ and sufficiently small $\delta$.
    To this end, we recursively apply \Cref{prop:iuzawa-net-err-reg} for $k = 0, \ldots, L-1$ to construct a series of sets $\{K^{(k)}\}_{k=0}^{L}$ and neural networks.
    Specifically, fix arbitrary $\delta > 0$ and $L \geq 1$, and denote $K^{(0)} := K_Y \cup K_U \cup \{0\}$, which is compact in $L^2(\Omega)$.
    \begin{itemize}
        \item For $k = 0, 1, \ldots, L-1$, apply \Cref{prop:iuzawa-net-err-reg} (i) with $K^{(k)}$ in place of $K$ to obtain a critical set $K^{(k+1)} \subset L^2(\Omega)$ of $K^{(k)}$, which is compact and satisfies $K^{(k)} \subset K^{(k+1)}$.
    \end{itemize}
    Take $K' := K^{(L)}$, let $j_0$ be given by \Cref{prop:iuzawa-net-err-reg} (ii) with $K'$ and $\delta$, and set $\cS = \cS^{(j)}$, $\cA = \cA^{(j)}$, and $\cQ_A = \cQ_A^{(j)}$ for some $j \geq j_0$.
    Now consider the \texttt{iUzawa-Net} $\cT(y_d, f; \theta_\cT)$ with each layer sharing the neural networks $\cS$, $\cA$, $\cQ_A$, and $\cQ_S$.
    Clearly, $\cT$ is weight tying.
    Consider any $(y_d, f)^\top \in K$ and any initialization $(p^0, u^0)^\top \in K \cup \{0\}$.
    Let $\{(u^k, p^k)^\top\}_{k=0}^L$ be the corresponding layer outputs of $\cT(y_d, f; \theta_\cT)$.
    Then $y_d, f, u^0, p^0 \in K^{(0)}$.
    Recall that $K^{(0)} \subset K^{(k)}$ and $K^{(k+1)} \subset K^{(L)} = K'$ for all $k = 0, \ldots, L-1$, and that the critical sets are independent of $K'$ and of the choice of $\cS$, $\cA$, and $\cQ_A$.
    Hence, applying \Cref{prop:iuzawa-net-err-reg} (i) inductively with $K^{(k)}$ and $K^{(k+1)}$ in place of $K$ and $K_0$, we have for all $k = 0, \ldots, L-1$ that
    \begin{equation*}
        (u^{k+1}, p^{k+1})^\top \in K^{(k+1)} \times K^{(k+1)}, \quad \text{and} \quad y_d, \ f, \ u^k, \ p^k \in K^{(k)} \subset K'.
    \end{equation*}
    That is, the inputs $y_d$, $f$, $u^k$, and $p^k$ of the $(k+1)$-th layer of $\cT$ stay inside the approximation domain $K'$ for all $k = 0, \ldots, L-1$.
    Hence, by \Cref{prop:iuzawa-net-err-reg} (ii), we have
    \begin{equation}\label{eq:tired-of-naming}
        \max\left\{ \left\| u^{k+1} - \bar{u}^{k+1} \right\|_{L^2(\Omega)},  \, \left\| p^{k+1} - \bar{p}^{k+1} \right\|_{L^2(\Omega)} \right\} \leq \delta, \quad \forall k = 0, \ldots, L - 1,
    \end{equation}
    for the sequences $\{\bar{u}^{k+1}\}$ and $\{ \bar{p}^{k+1} \}$ satisfying \eqref{eq:inexact-inexact-uzawa} with $Q_A = N + \tau I$ and $Q_S = \cQ_S^{-1}$, i.e., $\cT$ is $(K, \delta)$-tracking.

    By \Cref{lem:param-to-ctrl-lip} and \eqref{eq:opt-ctrl-opt-cond}, $w^* = (u^*, p^*)^\top$ depends continuously on $(y_d, f)^\top$. Since $K$ is compact, $\| w^0 - w^* \|_Q$ is uniformly bounded for any $(y_d, f)^\top \in K$ and $w^0 \in K^{(0)} \times K^{(0)}$.
    Therefore, by \eqref{eq:tired-of-naming}, \Cref{prop:inexact-inexact-uzawa-stability}, and \Cref{rem:bigo-const-dep}, for any $\varepsilon > 0$, there exists a $\delta_0 = C \varepsilon$ where $C$ is independent of $\varepsilon$, and $L_0 = O(\log(1/\varepsilon))$, both independent of $(y_d, f)^\top$, such that for any $\delta < \delta_0$ and $L \geq L_0$, the above specified $\cT$ satisfies
    \begin{equation}\label{eq:i-am-done}
        \left\| T(y_d, f) - u^k \right\|_{L^2(\Omega)} \leq \left\| w^k - w^* \right\|_{L^2(\Omega) \times L^2(\Omega)} < \varepsilon, \quad \forall k \geq L_0.
    \end{equation}
\end{proof}

By setting $k = L_0$ and taking supremum over $K$ in \eqref{eq:iuzawa-net-conv-to-eps}, \Cref{thm:iuzawa-net-univ-approx-rec} yields a new universal approximation theorem for the \texttt{iUzawa-Net}.
This result is distinguished from \Cref{thm:iuzawa-univ-approx} in that it guarantees the constructed \texttt{iUzawa-Net} to be $(K, \delta)$-tracking and weight tying.
Moreover, \Cref{thm:iuzawa-net-univ-approx-rec} extends naturally to the dual variables $\{p^k\}_{k=0}^L$ in the layer outputs, establishing the asymptotic $\varepsilon$-optimality of $\{p^k\}_{k=0}^L$ with respect to the dual optimal solution $p^*$ of \eqref{eq:opt-ctrl-pd}.
It is worth remarking that the asymptotic $\varepsilon$-optimality result presented in \Cref{thm:iuzawa-net-univ-approx-rec} does not, in itself, imply convergence, which requires the error to vanish as $k \to \infty$.
In addition, the shared neural networks constructed in the proof of \Cref{thm:iuzawa-net-univ-approx-rec} depend on the number of layers $L$ through the approximation domain $K' = K^{(L)}$.

As a consequence of \Cref{thm:iuzawa-net-univ-approx-rec}, the compactness assumption on $K$ can be relaxed to boundedness in Sobolev spaces.
The observation that additional regularity of the inputs avoids the compactness assumption is grounded in the classical Rellich--Kondrachov theorem and its generalizations \cite[Theorems 6.13 and 8.5]{leoni2023first}.

\begin{corollary}\label{cor:iuzawa-net-univ-approx-rec-reg}
    Suppose that \Cref{assump:ptwise,assump:subgrad-gcont} hold and let $s > 0$.
    For any $\varepsilon > 0$ and any bounded set $B \subset H^s(\Omega)$, there exists $\delta_0 = C\varepsilon$ with $C>0$ independent of $\varepsilon$ and $L_0 = O(\log(1/\varepsilon))$, such that for any $\delta < \delta_0$ and $L \geq L_0$, there exists an \texttt{iUzawa-Net} $\mathcal{T}(y_d, f; \theta_{\mathcal{T}})$ with $L$ layers that is weight tying and algorithm tracking with respect to $B \times B$ and $\delta$. Moreover, for any $(y_d, f)^\top \in B \times B$, the layer outputs $\{(u^k, p^k)^\top\}_{k=1}^L$ of $\cT(y_d, f; \theta_\cT)$ satisfy
    \begin{equation}\label{eq:iuzawa-net-univ-approx-rec-reg}
        \left\| T(y_d, f) - u^k \right\|_{L^2(\Omega)} < \varepsilon, \quad \forall k \geq L_0.
    \end{equation}
\end{corollary}

\begin{proof}
    If $s > 1$, then $B$ is bounded in $H^{\lfloor s \rfloor}(\Omega)$ by \cite[Theorem 11.8]{leoni2023first}, and hence bounded in $H^1(\Omega)$.
    Thus, we assume without loss of generality that $0 < s \leq 1$.
    By the Rellich--Kondrachov theorem and its generalizations \cite[Theorems 6.13 and 8.5]{leoni2023first}, the embedding $H^s(\Omega) \hookrightarrow L^2(\Omega)$ is compact, and hence the closure $\overline{B}$ of $B$ in $L^2(\Omega)$ is compact.
    The result follows immediately by applying \Cref{thm:iuzawa-net-univ-approx-rec} with the compact set $K := \overline{B} \times \overline{B} \subset Y \times U$ and the given $\varepsilon$.
\end{proof}

\section{Numerical Experiments}\label{sec:num-exp}

In this section, we demonstrate the application of the proposed \texttt{iUzawa-Net} to several representative instances of the optimal control problem~\eqref{eq:opt-ctrl}. A series of numerical results is presented to illustrate the effectiveness and efficiency of the \texttt{iUzawa-Net}.
The codes were written in Python, and all the deep learning methods are implemented with PyTorch. The numerical experiments were conducted on a Linux server equipped with an Intel i7-14700K processor and one NVIDIA GeForce RTX 4090 GPU.
All the codes and generated data are available at \url{https://github.com/tianyouzeng/iUzawa-Net}.

\subsection{Elliptic Optimal Control: Isotropic Case}\label{sec:num-exp-ellip-iso}

We consider the following elliptic optimal control problem defined on the bounded domain $\Omega := (0, 1) \times (0, 1) \subset \bR^2$:
\begin{equation}\label{eq:elliptic-opt-ctrl}
    \begin{aligned}
        & \underset{y\in L^2(\Omega), \, u\in L^2(\Omega)}{\min} \ &&\dfrac{1}{2} \|y - y_d\|_{L^2(\Omega)}^2 + \dfrac{\alpha}{2} \|u\|_{L^2(\Omega)}^2 + I_{U_{ad}}(u), \\
        & \qquad \ \text{s.t.} \ && \left\{
        \begin{aligned}
        & - \Delta y = u + f \quad &&\text{in}~\Omega, \\
        & y= 0 \quad &&\text{on}~\partial\Omega,
        \end{aligned}
        \right.
    \end{aligned}
\end{equation}
where $\alpha > 0$ is a regularization parameter and $I_{U_{ad}}(\cdot)$ is the indicator functional of the admissible set
\begin{equation*}
    U_{ad} = \{ u \in L^2(\Omega) \mid u_a(x) \leq u(x) \leq u_b(x) ~\text{a.e.~in}~\Omega \}
\end{equation*}
for some given $u_a, u_b \in L^\infty(\Omega)$.
Note that $I_{U_{ad}}(\cdot)$ equivalently imposes a control constraint $u \in U_{ad}$.

In this example, we consider the even more challenging case where $u_a$ and $u_b$ are also treated as parameters of \eqref{eq:elliptic-opt-ctrl} and vary across instances. This is allowed in implementation, since the neural network $\cQ_A^k$ in the \texttt{iUzawa-Net} admits $u_a$ and $u_b$ as inputs.
We fix the regularization parameter $\alpha = 0.01$ and apply the \texttt{iUzawa-Net} to learn the solution operator that maps the problem data $(y_d, f, u_a, u_b)$ to the optimal control $u^*$.
To generate the training and testing datasets, we first discretize the domain $\overline{\Omega}$ as a uniform $m \times m$ grid denoted by $\overline{\mathcal{D}}$.
For training, we set $m=64$ and sample $N = 16384$ functions $\{(y_d)_i\}_{i=1}^N$ and $\{f_i\}_{i=1}^N$ independently from the following Gaussian random fields (GRFs) with Mat\'ern covariance kernels \cite{lindgren2011explicit}:
\begin{equation*}
    \begin{aligned}
        & (y_d)_i \sim \sigma \mathcal{GR}(0, (-\Delta_D + 9I)^{-3}), \\
        & f_i \sim \frac{\sigma}{\sqrt{2}} \left( \mathcal{GR}(0, (-\Delta_D + 9I)^{-3}) + \mathcal{GR}(0, (-\Delta_N + 9I)^{-3}) \right),
    \end{aligned}
\end{equation*}
where $\Delta_D$ and $\Delta_N$ denote the Laplacian operators on $\Omega$ with homogeneous Dirichlet and Neumann boundary conditions, respectively, and $\mathcal{GR}(0, \mathcal{C})$ denotes the GRF with mean zero and covariance operator $\mathcal{C}$. We set $\sigma = 256\sqrt{2}$.
The control constraints $\{(u_a)_i\}_{i=1}^N$ and $\{(u_b)_i\}_{i=1}^N$ are generated as $(u_a)_i = \min\{a_i + v_i, 0\}$ and $(u_b)_i = \max\{b_i + w_i, 0\}$ for $i = 1, \ldots, N$, with
\begin{equation*}
    \begin{aligned}
        & a_i \sim \mathrm{Unif}(-10, -1), \quad b_i \sim \mathrm{Unif}(1, 10), \\
        & v_i, \, w_i \sim \frac{\sigma}{\sqrt{2}} \left( \mathcal{GR}(0, (-\Delta_D + 9I)^{-4}) + \mathcal{GR}(0, (-\Delta_N + 9I)^{-4}) \right).
    \end{aligned}
\end{equation*}
For each generated data tuple $\left((y_d)_i, f_i, (u_a)_i, (u_b)_i\right)$, we compute the corresponding optimal control $u^*_i$ of \eqref{eq:opt-ctrl} using the SSN method detailed in \cite{ulbrich2002semismooth}.
The training set is thus formulated as $\{((y_d)_i, f_i, (u_a)_i, (u_b)_i; u^*_i)\}_{i=1}^N$.
A testing set of size $N=2048$ is generated following the same procedure.
To evaluate the generalization ability of the trained \texttt{iUzawa-Net} across different resolutions, we further generate two additional testing sets from the same distributions, each of size $N=2048$, on finer grids of resolutions $m=128$ and $256$, respectively.
For each generated dataset, the number of active optimal controls $u_i^*$ and the measure ratio $|A(u_i^*)| / |\Omega|$ for each active set $A(u_i^*)$ of each active\footnote{An optimal control $u^*$ is considered active if $u^*(x) = u_a(x)$ or $u^*(x) = u_b(x)$ for some $x \in \overline{\mathcal{D}}$. For each active $u^*$, its active set is defined by $A(u^*) := \{ x \in \Omega : u^*(x) = u_a(x) \text{~or~} u^*(x) = u_b(x) \}$.} $u_i^*$ are summarized in \Cref{tab:elliptic-active-set-stats}.
The presence of both active and inactive instances in the dataset ensures a comprehensive evaluation of \texttt{iUzawa-Net} across problems with active or redundant control constraint $u \in U_{ad}$.

\begin{table}[!ht]
    \centering
	\small
    \begin{tabular}{c c c c c}
		\toprule
		\multirow{2}[3]{*}{$m$} & \multicolumn{2}{c}{$\#$ of active $u_i^*$} & \multicolumn{2}{c}{$|A(u_i^*)| / |\Omega|$ for active $u_i^*$} \\
        \cmidrule(lr){2-3} \cmidrule(lr){4-5}
        ~ & Training & Testing & Training Mean & Testing Mean \\
		\midrule
		$64$ & $11450$ & $1422$ & $0.2116$ & $0.2128$ \\
		$128$ & --- & $1419$ & --- & $0.2130$ \\
		$256$ & --- & $1433$ & --- & $0.2256$ \\
		\bottomrule
	\end{tabular}
    \caption{Statistics of active optimal control and active sets in the training and testing datasets of different resolutions $m$. A dash (---) indicates that the corresponding data is not applicable.}
    \label{tab:elliptic-active-set-stats}
    \normalsize
\end{table}

We construct two \texttt{iUzawa-Net}s, called \texttt{iUzawa-Net-F} and \texttt{iUzawa-Net-S}, respectively.
The \texttt{iUzawa-Net-F} uses different neural networks $\cS^k$, $\cA^k$, $\cQ_A^k$, and $\cQ_S^k$ for each different layer $k$ (hence the suffix \texttt{-F} for free), while the \texttt{iUzawa-Net-S} employs weight tying across all layers (hence the suffix \texttt{-S} for shared).
The rationality of the weight tying construction is justified by \Cref{thm:iuzawa-net-univ-approx-rec} and \Cref{cor:iuzawa-net-univ-approx-rec-reg}, since \Cref{assump:ptwise,assump:subgrad-gcont} hold by the discussions in \Cref{sec:qa-special} and \Cref{rem:subgrad-gcont}, and the Sobolev regularity of $(y_d)_i$ and $f_i$ is induced by the Laplacian operators in the Mat\'ern covariance kernels\footnote{Although \Cref{thm:iuzawa-net-univ-approx-rec} and \Cref{cor:iuzawa-net-univ-approx-rec-reg} require $u_a$ and $u_b$ to be fixed across instances, the analysis for proving them can be easily extended to the varying box constrained case.}.
We denote the two \texttt{iUzawa-Net}s uniformly as $\mathcal{T}(y_d, f, u_a, u_b; \theta_\cT)$, which shall not cause confusion by context.
We take $L=6$ layers for the \texttt{iUzawa-Net-F} and $L=10$ layers for the \texttt{iUzawa-Net-S}.
The configurations for the neural sub-networks are identical for the \texttt{iUzawa-Net-F} and \texttt{iUzawa-Net-S}, which are detailed as follows.
Each of $\cS^k$, $\cA^k$, and $\cQ_S^k$ admits a lifting dimension of $m_p = 8$ and a maximum of $k_{\text{max}} = 8$ truncated Fourier modes in each coordinate.
We exploit $S=S^*$ and use a single FNO for both $\cS^k$ and $\cA^k$, with $3$ Fourier layers and the GeLU activation function.
We set $\gamma = 10^{-6}$ in $\cQ_S^k$.
For $\cQ_A^k$, all hidden layers admit a width of $64$, and we take $\tau = 10^{-4}$.
All inputs are padded to a resolution of $72$ for handling non-periodic cases.
These configurations lead to a total of $272{,}094$ trainable parameters for the \texttt{iUzawa-Net-F} and $45{,}349$ parameters for the \texttt{iUzawa-Net-S}.

We train the \texttt{iUzawa-Net-F} and \texttt{iUzawa-Net-S} by minimizing the discretized empirical loss function of \eqref{eq:relative-l2-err-def} with $\varepsilon_L = 10^{-8}$.
We employ the \texttt{AdamW} optimizer \cite{loshchilov2018decoupled} with a batch size of $64$ and train the networks for $300$ epochs.
The learning rate is initialized to $0.001$ and is decayed by a factor of $0.6$ every $30$ epochs.
The training process takes approximately $6.6 \times 10^{2}$ and $1.1 \times 10^{3}$ seconds for the \texttt{iUzawa-Net-F} and \texttt{iUzawa-Net-S}, respectively.

To evaluate the performance of the trained models, we compute the following absolute and relative $L^2$-errors of the predicted optimal control $\hat{u}$:
\begin{equation}\label{eq:acc-abs-rel-def}
    \varepsilon_{\text{abs}}(\hat{u}) = \left\|\hat{u} - u^* \right\|_{L^2(\Omega)}, \quad \varepsilon_{\text{rel}}(\hat{u}) = \frac{\varepsilon_{\text{abs}}(\hat{u})}{\left\| u^* \right\|_{L^2(\Omega)}}.
\end{equation}
Furthermore, for comparison with benchmark operator learning methods for approximating the solution operator $T$, we test the numerical accuracy of two standard FNOs in \cite{li2021fourier} for approximating the solution operator of \eqref{eq:elliptic-opt-ctrl}.
One FNO is constructed with $3$ Fourier layers, $10$ Fourier modes, and a lifting dimension of $8$; and the other employs $4$ Fourier layers, $12$ Fourier modes, and a lifting dimension of $15$.
The two FNOs contain $45{,}881$ and $295{,}682$ neural network parameters, which are slightly higher than the \texttt{iUzawa-Net-S} and the \texttt{iUzawa-Net-F}, respectively.
The numerical results are reported in \Cref{tab:elliptic-numerical-results-all}.
Additionally, \Cref{fig:elliptic-results} displays the computed control and the pointwise error relative to the reference solution for a randomly selected instance from the testing set.

\begin{table}[!ht]
    \centering
    \small
    \begin{tabular}{c c c c c c}
        \toprule
        \multirow{2}{*}{Method} & \multirow{2}{*}{$m$}
        & \multicolumn{2}{c}{$\varepsilon_{\mathrm{rel}}(\hat{u})$}
        & \multicolumn{2}{c}{$\varepsilon_{\mathrm{abs}}(\hat{u})$} \\
        \cmidrule(lr){3-4} \cmidrule(lr){5-6}
        & & Mean & SD & Mean & SD \\
        \midrule
        \multirow{3}{*}{\texttt{iUzawa-Net-F}}
        & $64$  & $2.00\times10^{-3}$ & $0.97\times10^{-3}$ & $4.05\times10^{-3}$ & $1.85\times10^{-3}$ \\
        & $128$ & $2.29\times10^{-3}$ & $0.99\times10^{-3}$ & $4.60\times10^{-3}$ & $1.76\times10^{-3}$ \\
        & $256$ & $2.38\times10^{-3}$ & $0.85\times10^{-3}$ & $4.99\times10^{-3}$ & $2.13\times10^{-3}$ \\
        \midrule
        \multirow{3}{*}{\texttt{iUzawa-Net-S}}
        & $64$  & $3.67\times10^{-3}$ & $3.02\times10^{-3}$ & $7.03\times10^{-3}$ & $3.74\times10^{-3}$ \\
        & $128$ & $3.73\times10^{-3}$ & $2.79\times10^{-3}$ & $7.15\times10^{-3}$ & $3.69\times10^{-3}$ \\
        & $256$ & $3.76\times10^{-3}$ & $2.62\times10^{-3}$ & $7.46\times10^{-3}$ & $3.96\times10^{-3}$ \\
        \midrule
        \multirow{3}{*}{$3$-layer FNO}
        & $64$  & $1.44\times10^{-2}$ & $1.11\times10^{-2}$ & $2.97\times10^{-2}$ & $2.10\times10^{-2}$ \\
        & $128$ & $1.43\times10^{-2}$ & $1.08\times10^{-2}$ & $2.95\times10^{-2}$ & $2.15\times10^{-2}$ \\
        & $256$ & $1.44\times10^{-2}$ & $1.06\times10^{-2}$ & $3.05\times10^{-2}$ & $2.18\times10^{-2}$ \\
        \midrule
        \multirow{3}{*}{$4$-layer FNO}
        & $64$  & $1.07\times10^{-2}$ & $0.88\times10^{-2}$ & $2.23\times10^{-2}$ & $1.58\times10^{-2}$ \\
        & $128$ & $1.12\times10^{-2}$ & $0.82\times10^{-2}$ & $2.29\times10^{-2}$ & $1.56\times10^{-2}$ \\
        & $256$ & $1.12\times10^{-2}$ & $0.79\times10^{-2}$ & $2.38\times10^{-2}$ & $1.62\times10^{-2}$ \\
        \bottomrule
    \end{tabular}
    \caption{Comparison of numerical results for solving \eqref{eq:elliptic-opt-ctrl}.}
    \label{tab:elliptic-numerical-results-all}
\end{table}

\begin{figure}
    \centering
    \hfill
    \begin{subfigure}{0.32\textwidth}
        \includegraphics[width=\textwidth]{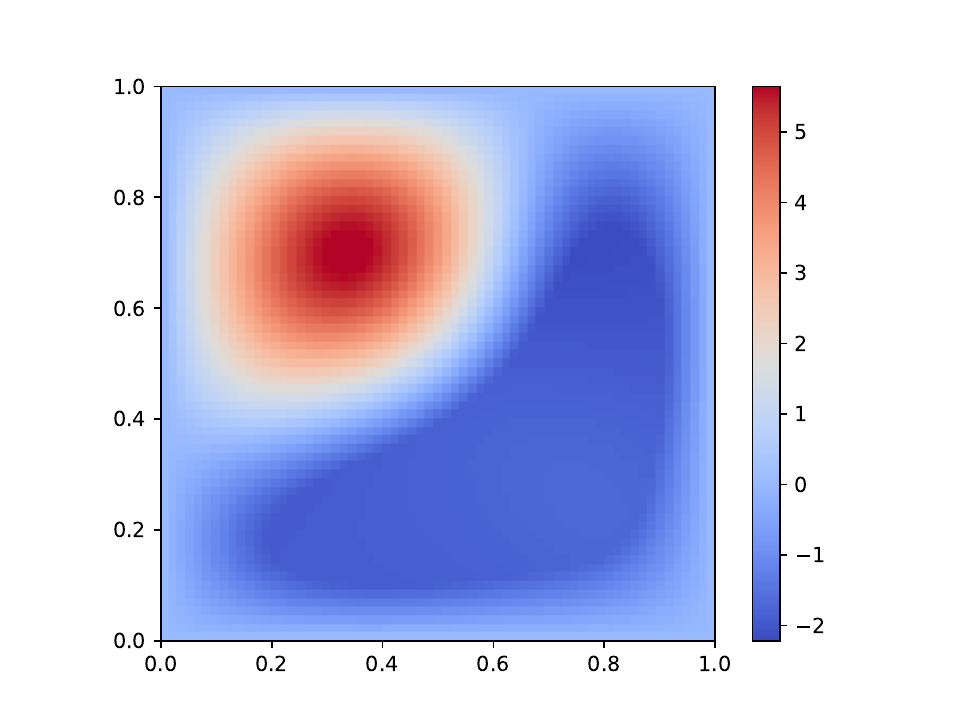}
        \subcaption{Exact solution $u^*$.}
    \end{subfigure}
    \hfill
    \begin{subfigure}{0.32\textwidth}
        \includegraphics[width=\textwidth]{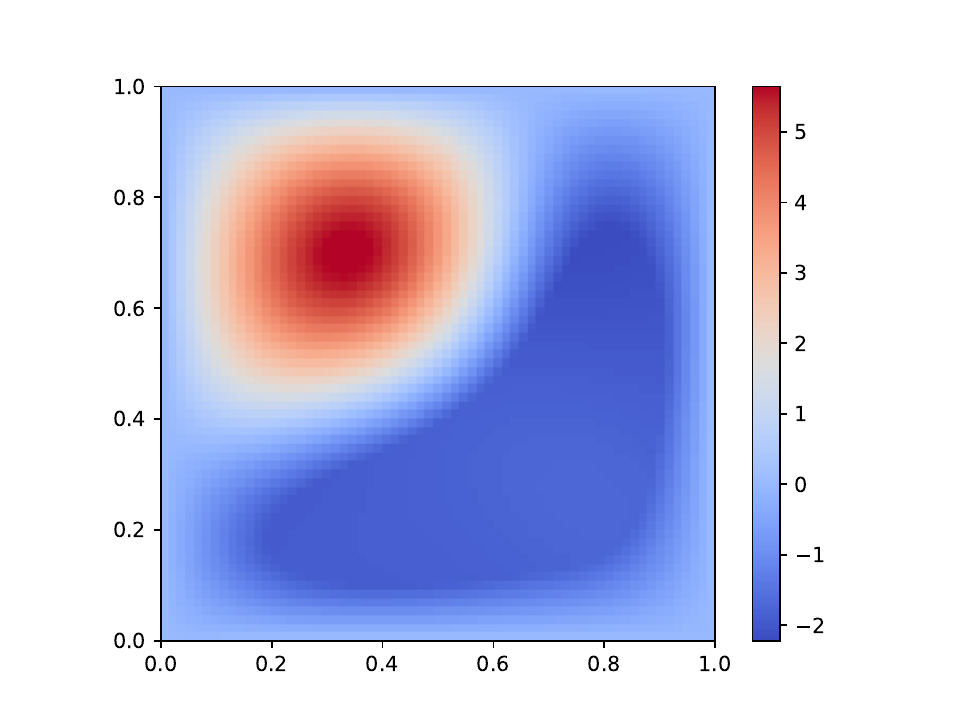}
        \subcaption{Computed solution $\hat{u}$.}
    \end{subfigure}
    \hfill
    \begin{subfigure}{0.32\textwidth}
        \includegraphics[width=\textwidth]{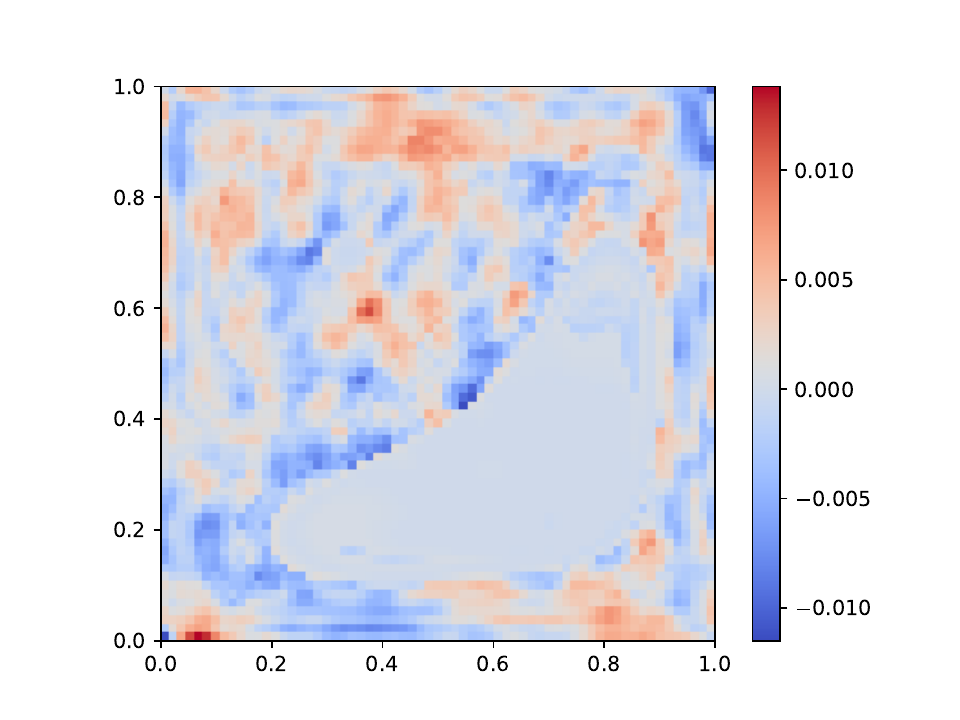}
        \subcaption{Difference $\hat{u} - u^*$.}
    \end{subfigure}
    \hfill
    \caption{Computed optimal control for a single instance of \eqref{eq:elliptic-opt-ctrl} with $m=64$.}
    \label{fig:elliptic-results}
\end{figure}

The results in \Cref{tab:elliptic-numerical-results-all} demonstrate that the \texttt{iUzawa-Net-F} and \texttt{iUzawa-Net-S} achieve satisfactory accuracy with few layers.
Compared with the FNOs, the proposed \texttt{iUzawa-Net}s yield significantly higher accuracy, highlighting the advantage of the optimization-informed network design.
Furthermore, the results confirm the excellent generalization capability of the \texttt{iUzawa-Net-F} and \texttt{iUzawa-Net-S} in a zero-shot super-resolution setting, i.e., a model trained on low-resolution data can output accurate predictions for high-resolution inputs without any retraining or fine-tuning.

To evaluate the computational efficiency, we compare the inference time of the trained \texttt{iUzawa-Net}s with the runtimes of four established FEM-based methods: the SSN method in \cite{porcelli2015preconditioning}, the primal-dual method in \cite{chambolle2011first,he2012convergence}, the inexact Uzawa method \eqref{eq:inexact-uzawa}, and a POD-based reduced-order method \cite{hinze2005proper,troltzsch2009pod}.
For the primal-dual method, the stepsizes of $3.5 \times 10^{2}$ and $1$ are taken in each primal and dual update, respectively.
The SSN method and the inexact Uzawa method do not require specific stepsize tuning.
For the inexact Uzawa method, we adopt the preconditioning strategy from \cite[Section 5]{song2019inexact}.
For the POD method, we first obtain the POD basis from $N_{\text{POD}}$ samples in the generated training data with respect to the $H^1_0(\Omega)$ inner product \cite{hinze2005proper}.
Then, the state and the adjoint state are approximated by the Galerkin projection onto the span of the first $\ell$ POD basis functions, while the control is represented from the optimality condition by $p$ \cite{hinze2005proper,troltzsch2009pod}. The resulting reduced optimality system is solved by the primal-dual active set method \cite{hintermuller2003primal}.
To ensure a fair comparison, matrices in the linear systems are pre-factorized prior to timing for all compared methods where applicable.
The termination criteria for each method are set as $\varepsilon_{\text{rel}} \leq \varepsilon_\text{rtol}$, where $\varepsilon_{\text{rel}}$ is defined in \eqref{eq:acc-abs-rel-def}. We take $\varepsilon_\text{rtol} = 2 \times 10^{-3}$, $\ell = 40$, and $N_{\text{POD}} = 1024$, whose numerical results match the accuracy achieved by the trained \texttt{iUzawa-Net-F} and \texttt{iUzawa-Net-S}.
The generation of POD basis takes $1.3$, $3.7$, and $17.9$ seconds for $m=64$, $128$, and $256$, respectively.
The \texttt{iUzawa-Net}s are evaluated one instance at a time without parallelism for fair comparison.

The average computation time per instance for various resolutions $m$ is reported in \Cref{tab:elliptic-computation-time}.
In addition, we present the number of iterations or layers required for each numerical method in \Cref{tab:elliptic-iter}.
The computational time of the \texttt{iUzawa-Net-F} and \texttt{iUzawa-Net-S} is below one millisecond for all resolutions (before rounding), which is lower than the computation time of the FEM-based full-space methods by factors between approximately $50$ and $1600$, and by a factor of $3$ to $30$ for POD methods.
Moreover, as the resolution refines, the computational cost of the \texttt{iUzawa-Net}s increases only mildly, in contrast to the significant growth of the FEM-based solvers.
In addition, the POD basis has to be computed for each resolution in the offline phase, while the \texttt{iUzawa-Net}s trained with $m = 64$ are applied to all the resolutions without retraining.

\begin{table}[!ht]
    \centering
	\small
    \begin{tabular}{c c c c c c c}
		\toprule
		$m$ & SSN & Primal-Dual & Inexact Uzawa & POD & \texttt{iUzawa-Net-F} & \texttt{iUzawa-Net-S} \\
		\midrule
		$64$ & $79.2$ & $25.1$ & $42.7$ & $1.1$ & $0.3$ & $0.5$ \\
		$128$ & $176.7$ & $218.7$ & $177.1$ & $3.8$ & $0.4$ & $0.6$ \\
		$256$ & $773.9$ & $966.0$ & $355.0$ & $18.3$ & $0.6$ & $1.0$ \\
		\bottomrule
	\end{tabular}
    \caption{Computation time (milliseconds) of the SSN method, the primal-dual method, the inexact Uzawa method, the POD method, the \texttt{iUzawa-Net-F}, and the \texttt{iUzawa-Net-S} for solving \eqref{eq:elliptic-opt-ctrl} under different resolutions $m$ (averaged over instances in the testing sets).}
    \label{tab:elliptic-computation-time}
    \normalsize
\end{table}

\begin{table}[!ht]
    \centering
	\small
    \begin{tabular}{c c c c c c c}
		\toprule
		$m$ & SSN & Primal-Dual & Inexact Uzawa & POD & \texttt{iUzawa-Net-F} & \texttt{iUzawa-Net-S} \\
		\midrule
		$64$ & $2.59$ & $10.77$ & $17.94$ & $2.36$ & $6$ & $10$ \\
		$128$ & $1.91$ & $10.78$ & $7.71$ & $2.43$ & $6$ & $10$ \\
		$256$ & $1.94$ & $10.78$ & $6.57$ & $2.60$ & $6$ & $10$ \\
		\bottomrule
	\end{tabular}
    \caption{Number of iterations or layers of the SSN method, the primal-dual method, the inexact Uzawa method, the POD method, the \texttt{iUzawa-Net-F}, and the \texttt{iUzawa-Net-S} for solving \eqref{eq:elliptic-opt-ctrl} under different resolutions $m$ (averaged over instances in the testing sets).}
    \label{tab:elliptic-iter}
    \normalsize
\end{table}

To further compare the two architectures in terms of their depth, we train the \texttt{iUzawa-Net-F} and the \texttt{iUzawa-Net-S} with $L \in \{2, 3, 6, 10\}$ layers.
The relative errors on the testing set with $m = 64$ are reported in \Cref{tab:elliptic-param-budget}.
For the \texttt{iUzawa-Net-F}, the error decreases as the number of layers increases with diminishing improvement.
For the \texttt{iUzawa-Net-S}, the same decreasing error effect is observed, although the number of parameters is independent of $L$.
These results show that weight tying is considerably more parameter-efficient if sufficiently many layers are allowed, while for a fixed number of layers, the \texttt{iUzawa-Net-F} attains higher numerical accuracy.

\begin{table}[!ht]
    \centering
    \small
    \begin{tabular}{c c c c @ {\hspace{2em}} c c c c}
        \toprule
        \multicolumn{4}{c}{\texttt{iUzawa-Net-F}} & \multicolumn{4}{c}{\texttt{iUzawa-Net-S}} \\
        \cmidrule(lr){1-4} \cmidrule(lr){5-8}
        \multirow{2}{*}{$L$} & \multirow{2}{*}{Parameters} & \multicolumn{2}{c}{$\varepsilon_{\mathrm{rel}}(\hat{u})$} & \multirow{2}{*}{$L$} & \multirow{2}{*}{Parameters} & \multicolumn{2}{c}{$\varepsilon_{\mathrm{rel}}(\hat{u})$} \\
        \cmidrule(lr){3-4} \cmidrule(lr){7-8}
        & & Mean & SD & & & Mean & SD \\
        \midrule
        $2$ & $90{,}698$ & $6.95 \times 10^{-3}$ & $4.53 \times 10^{-3}$ & $2$ & $45{,}349$ & $1.21 \times 10^{-2}$ & $1.12 \times 10^{-2}$ \\
        $3$ & $136{,}047$ & $4.14 \times 10^{-3}$ & $2.38 \times 10^{-3}$ & $3$ & $45{,}349$ & $6.15 \times 10^{-3}$ & $2.95 \times 10^{-3}$ \\
        $6$ & $272{,}094$ & $2.00 \times 10^{-3}$ & $0.97 \times 10^{-3}$ & $6$ & $45{,}349$ & $5.88 \times 10^{-3}$ & $3.21 \times 10^{-3}$ \\
        $10$ & $453{,}490$ & $1.63 \times 10^{-3}$ & $0.77 \times 10^{-3}$ & $10$ & $45{,}349$ & $3.67 \times 10^{-3}$ & $3.02 \times 10^{-3}$ \\
        \bottomrule
    \end{tabular}
    \caption{Relative errors $\varepsilon_{\mathrm{rel}}(\hat{u})$ on the testing set with $m = 64$ of the \texttt{iUzawa-Net-F} and the \texttt{iUzawa-Net-S} with different numbers of layers $L$ for solving \eqref{eq:elliptic-opt-ctrl}.}
    \label{tab:elliptic-param-budget}
    \normalsize
\end{table}

\subsection{Elliptic Optimal Control: Anisotropic Case}\label{sec:num-exp-ellip-ansio}
To further demonstrate the advantages of data-driven preconditioner training in the \texttt{iUzawa-Net}, we consider an ill-conditioned elliptic optimal control problem governed by an anisotropic elliptic operator.
Specifically, let $\Omega = (0, 1) \times (0, 1) \subset \bR^2$ and consider the following problem with a Neumann boundary condition:
\begin{equation}\label{eq:elliptic-opt-ctrl-illcond}
    \begin{aligned}
        & \underset{y\in L^2(\Omega), \, u\in L^2(\Omega)}{\min} \ &&\dfrac{1}{2} \|y - y_d\|_{L^2(\Omega)}^2 + \dfrac{\alpha}{2} \|u\|_{L^2(\Omega)}^2 + I_{U_{ad}}(u), \\
        & \qquad \ \text{s.t.} \ && \left\{
        \begin{aligned}
        & - \nabla \cdot (a \nabla y) + cy = u + f \quad &&\text{in}~\Omega, \\
        & \frac{\partial y}{\partial \mathbf{n}} = 0 \quad &&\text{on}~\partial\Omega,
        \end{aligned}
        \right.
    \end{aligned}
\end{equation}
where $a \in \bR^{2 \times 2}$ is a symmetric and positive definite matrix, $c>0$ is a constant, and $\mathbf{n}$ denotes the outward unit normal vector to $\partial \Omega$.
Other symbols are defined as in \Cref{sec:num-exp-ellip-iso}.
In this example, we set
\begin{equation*}
    a = \begin{pmatrix}
    1 & 0 \\
    0 & 100
\end{pmatrix}, \quad c = 1, \quad \text{~and~} \alpha = 0.01.
\end{equation*}
The large contrast in the eigenvalues of the coefficient matrix $a$ introduces strong anisotropy, making the problem ill-conditioned and numerically challenging.

As in \Cref{sec:num-exp-ellip-iso}, we treat $u_a$ and $u_b$ as problem parameters alongside $y_d$ and $f$.
The training and testing data are generated using the strategy described in \Cref{sec:num-exp-ellip-iso}, with the exception that $(y_d)_i$ is sampled from a Gaussian random field constructed to have non-vanishing boundary values:
\begin{equation*}
    (y_d)_i \sim \frac{\sigma}{\sqrt{2}} \left( \mathcal{GR}(0, (-\Delta_D + 9I)^{-3}) + \mathcal{GR}(0, (-\Delta_N + 9I)^{-3}) \right).
\end{equation*}
We generate a training set of size $N = 16384$ with resolution $m=64$.
The testing set consists of three groups with resolutions $m = 64, \, 128, \, 256$ and sizes $N = 2048, \, 64, \, 64$, respectively.

Two neural networks, \texttt{iUzawa-Net-F} and \texttt{iUzawa-Net-S}, are constructed using the same configuration as described in \Cref{sec:num-exp-ellip-iso}.
We train the \texttt{iUzawa-Net-F} and \texttt{iUzawa-Net-S} by minimizing the discretized empirical loss function of \eqref{eq:relative-l2-err-def} with $\varepsilon_L = 10^{-8}$, using the \texttt{AdamW} optimizer with a batch size of $64$ for $300$ epochs.
The learning rate is initialized as $0.002$ and decays by a factor of $0.6$ every $30$ epochs.
The total training time is approximately $6.7 \times 10^{2}$ and $1.1 \times 10^{3}$ seconds for the \texttt{iUzawa-Net-F} and \texttt{iUzawa-Net-S}, respectively.

We include in \Cref{fig:elliptic-illcond-results} the computed control and the pointwise difference relative to the reference solution for a randomly selected instance from the testing set.
The numerical errors of the computed solutions defined in \eqref{eq:acc-abs-rel-def} are reported in \Cref{tab:elliptic-illcond-numerical-results-all}, which demonstrate that both the \texttt{iUzawa-Net-F} and \texttt{iUzawa-Net-S} achieve satisfactory accuracy despite the anisotropy of the problem.

We compare the inference time of the \texttt{iUzawa-Net}s with the runtimes of the SSN method in \cite{porcelli2015preconditioning}, the primal-dual method in \cite{chambolle2011first,he2012convergence}, the inexact Uzawa method in \cite{song2019inexact}, and the POD method described in \Cref{sec:num-exp-ellip-iso}.
The termination criteria are set to be the same as those in \Cref{sec:num-exp-ellip-iso} with $\varepsilon_{\text{rtol}} = 4 \times 10^{-3}$.
We further take $\ell = 10$ and $N_{\text{POD}} = 1024$ to match the relative accuracy of POD methods and the trained neural networks.
For the primal-dual method, we observe that small stepsizes are necessary to ensure convergence, and we take $2.0$ for each primal update and $0.4$ for each dual update.
The generation of the POD basis takes $1.3$, $4.0$, and $19.5$ seconds for $m=64$, $128$, and $256$, respectively.
We again evaluate \texttt{iUzawa-Net}s one instance at a time to ensure a fair comparison with the other methods.

The results are reported in \Cref{tab:elliptic-illcond-computation-time}.
The performance of both the SSN method and the primal-dual method deteriorates significantly under the ill-conditioned setting.
The inexact Uzawa method, despite using a multigrid-based preconditioner, also suffers from the anisotropy.
However, the inference time of the two \texttt{iUzawa-Net}s is nearly identical to that reported in \Cref{sec:num-exp-ellip-iso}.
It is notable that the POD time is lower in the present experiment due to the reduced number of basis functions, however still being above the \texttt{iUzawa-Net}s.
This confirms that the superiority in efficiency of the \texttt{iUzawa-Net} is even more substantial for the ill-conditioned case.

\begin{figure}
    \centering
    \hfill
    \begin{subfigure}{0.32\textwidth}
        \includegraphics[width=\textwidth]{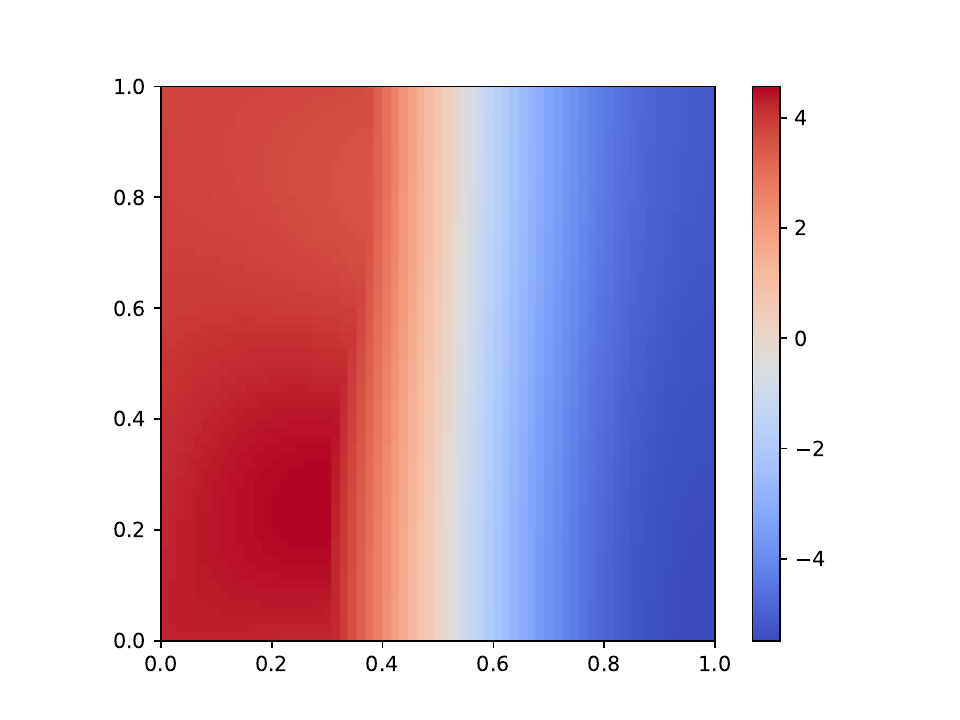}
        \subcaption{Exact solution $u^*$.}
    \end{subfigure}
    \hfill
    \begin{subfigure}{0.32\textwidth}
        \includegraphics[width=\textwidth]{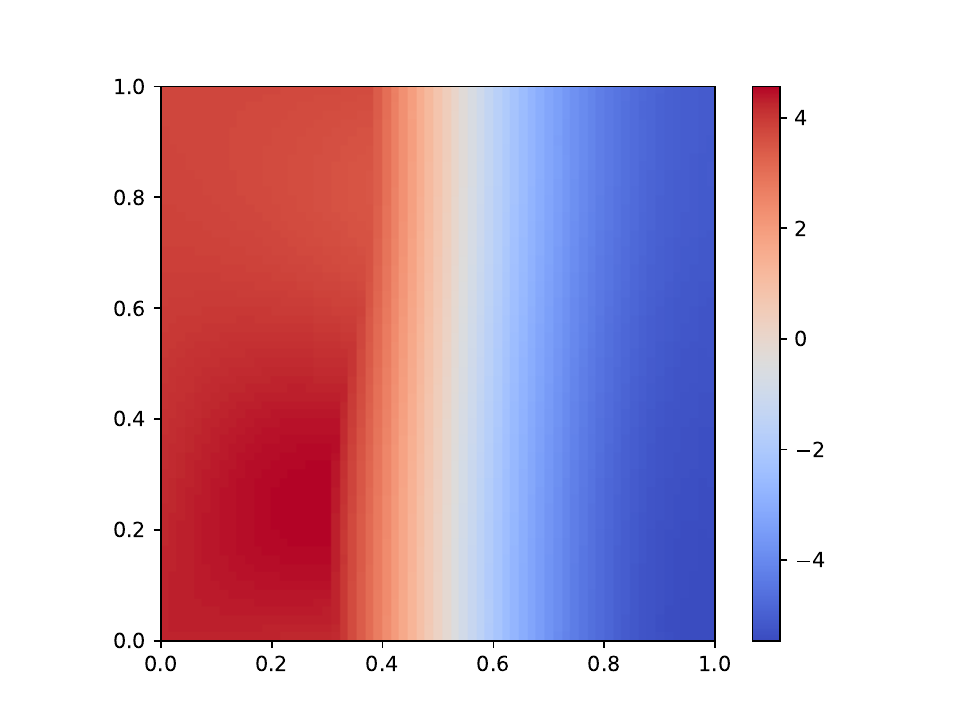}
        \subcaption{Computed solution $\hat{u}$.}
    \end{subfigure}
    \hfill
    \begin{subfigure}{0.32\textwidth}
        \includegraphics[width=\textwidth]{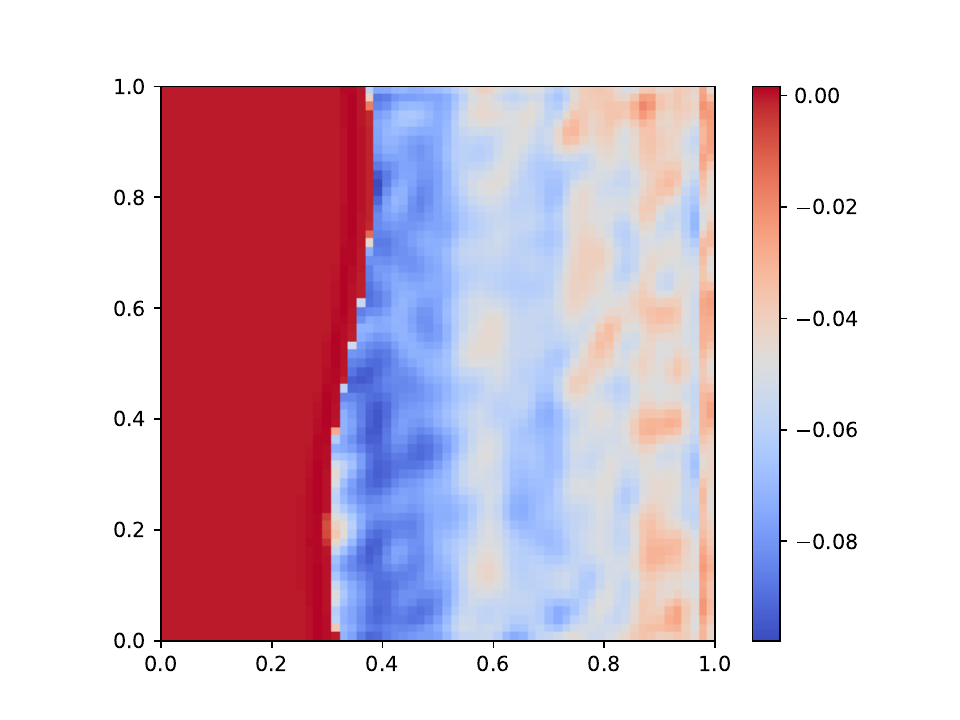}
        \subcaption{Difference $\hat{u} - u^*$.}
    \end{subfigure}
    \hfill
    \caption{Computed optimal control for a single instance of \eqref{eq:elliptic-opt-ctrl-illcond} with $m=64$.}
    \label{fig:elliptic-illcond-results}
\end{figure}

\begin{table}[!ht]
    \centering
    \small
    \begin{tabular}{c c c c c c}
        \toprule
        \multirow{2}{*}{Method} & \multirow{2}{*}{$m$}
        & \multicolumn{2}{c}{$\varepsilon_{\mathrm{rel}}(\hat{u})$}
        & \multicolumn{2}{c}{$\varepsilon_{\mathrm{abs}}(\hat{u})$} \\
        \cmidrule(lr){3-4} \cmidrule(lr){5-6}
        & & Mean & SD & Mean & SD \\
        \midrule
        \multirow{3}{*}{\texttt{iUzawa-Net-F}}
        & $64$  & $7.14 \times 10^{-3}$ & $1.02 \times 10^{-2}$ & $1.85 \times 10^{-2}$ & $1.82 \times 10^{-2}$ \\
        & $128$ & $5.52 \times 10^{-3}$ & $4.48 \times 10^{-3}$ & $1.50 \times 10^{-2}$ & $1.10 \times 10^{-2}$ \\
        & $256$ & $6.89 \times 10^{-3}$ & $6.68 \times 10^{-3}$ & $1.69 \times 10^{-2}$ & $1.36 \times 10^{-2}$ \\
        \midrule
        \multirow{3}{*}{\texttt{iUzawa-Net-S}}
        & $64$  & $4.71 \times 10^{-3}$ & $5.40 \times 10^{-3}$ & $1.12 \times 10^{-2}$ & $8.81 \times 10^{-3}$ \\
        & $128$ & $4.12 \times 10^{-3}$ & $2.96 \times 10^{-3}$ & $1.02 \times 10^{-2}$ & $7.72 \times 10^{-3}$ \\
        & $256$ & $4.35 \times 10^{-3}$ & $2.41 \times 10^{-3}$ & $1.05 \times 10^{-2}$ & $5.55 \times 10^{-3}$ \\
        \bottomrule
    \end{tabular}
    \caption{Numerical results for solving \eqref{eq:elliptic-opt-ctrl-illcond}.}
    \label{tab:elliptic-illcond-numerical-results-all}
\end{table}

\begin{table}[!ht]
    \centering
	\small
    \begin{tabular}{c c c c c c c}
		\toprule
		$m$ & SSN & Primal-Dual & Inexact Uzawa & POD & \texttt{iUzawa-Net-F} & \texttt{iUzawa-Net-S} \\
		\midrule
		$64$ & $4.29 \times 10^{3}$ & $1.14 \times 10^{3}$ & $171.1$ & $0.8$ & $0.3$ & $0.5$ \\
		$128$ & $2.14 \times 10^{4}$ & $8.71 \times 10^{3}$ & $1.17 \times 10^{3}$ & $2.2$ & $0.4$ & $0.6$ \\
		$256$ & $6.07 \times 10^{4}$ & $2.30 \times 10^{4}$ & $4.94 \times 10^{3}$ & $9.2$ & $0.6$ & $1.0$ \\
		\bottomrule
	\end{tabular}
    \caption{Computation time (milliseconds) of the SSN method, the primal-dual method, the inexact Uzawa method, the POD method, the \texttt{iUzawa-Net-F}, and the \texttt{iUzawa-Net-S} for solving \eqref{eq:elliptic-opt-ctrl-illcond} with different grid resolutions $m$ (averaged over instances in the testing sets).}
    \label{tab:elliptic-illcond-computation-time}
    \normalsize
\end{table}

Finally, we demonstrate the number of iterations (averaged over testing sets) required for the compared methods in \Cref{tab:elliptic-illcond-iter}.
For achieving the same level of accuracy, the \texttt{iUzawa-Net}s require significantly fewer layers than the number of outer iterations required by the primal-dual method (without preconditioner) and the inexact Uzawa method (with manually-designed multigrid-based preconditioner).
This verifies the numerical advantages of the data-driven and automatic preconditioner-training mechanism of the \texttt{iUzawa-Net}.
It is noteworthy that the SSN method requires fewer outer iterations due to its superlinear convergence property; however, this is at the cost of more computationally demanding inner iterations for solving the resulting algebraic systems, as reflected by the computation time in \Cref{tab:elliptic-illcond-computation-time}.

\begin{table}[!ht]
    \centering
	\small
    \begin{tabular}{c c c c c c c}
		\toprule
		$m$ & SSN & Primal-Dual & Inexact Uzawa & POD & \texttt{iUzawa-Net-F} & \texttt{iUzawa-Net-S} \\
		\midrule
		$64$ & $3.64$ & $209.23$ & $29.06$ & $3.37$ & $6$ & $10$ \\
		$128$ & $2.86$ & $214.19$ & $21.61$ & $3.53$ & $6$ & $10$ \\
		$256$ & $3.08$ & $215.75$ & $22.02$ & $3.69$ & $6$ & $10$ \\
		\bottomrule
	\end{tabular}
    \caption{Number of iterations or layers of the SSN method, the primal-dual method, the inexact Uzawa method, the POD method, the \texttt{iUzawa-Net-F}, and the \texttt{iUzawa-Net-S} for solving \eqref{eq:elliptic-opt-ctrl-illcond} under different resolutions $m$ (averaged over instances in the testing sets).}
    \label{tab:elliptic-illcond-iter}
    \normalsize
\end{table}

\subsection{Parabolic Optimal Control}

We now extend our investigation to the following parabolic optimal control problem defined on the spatial-temporal domain $\Omega_T := \Omega \times [0, T_f]$ with $\Omega := (0, 1) \times (0, 1) \subset \mathbb{R}^2$ and $T_f = 1$:
\begin{equation}\label{eq:parabolic-opt-ctrl}
    \begin{aligned}
        & \underset{y\in L^2(\Omega_T), \, u\in L^2(\Omega_T)}{\min} \ && \dfrac{1}{2} \|y - y_d\|_{L^2(\Omega_T)}^2 + \dfrac{\alpha}{2} \|u\|_{L^2(\Omega_T)}^2 + \beta \|u\|_{L^1(\Omega_T)} + I_{U_{ad}}(u), \\
        & \qquad \qquad \ \text{s.t.} \ && \left\{
        \begin{aligned}
        & \frac{\partial y}{\partial t} - \Delta y = u + f\quad &&\text{in}~\Omega_T, \\
        & y= 0 \quad&& \text{on}~\partial\Omega \times (0, T_f),  \\
        & y(0) = 0\quad &&\text{in}~\Omega,
        \end{aligned}
        \right.
    \end{aligned}
\end{equation}
where $\alpha > 0$, $\beta \geq 0$ are regularization parameters. The term $\beta \|u\|_{L^1(\Omega_T)}$ introduces sparsity in the optimal control, and the indicator functional $I_{U_{ad}}$ imposes a control constraint $u \in U_{ad}$, where the admissible set $U_{ad}$ is defined by
\begin{equation}
    U_{ad} = \{ u \in L^2(\Omega_T) \mid u_a(x, t) \leq u(x, t) \leq u_b(x, t) ~\text{a.e.~in~} \Omega \times [0, T_f] \}
\end{equation}
for some $u_a, u_b \in L^\infty(\Omega_T)$.

In this numerical example, we set the parameters as $\alpha = 0.01$, $\beta = 0.01$, $u_a = -6$, and $u_b = 6$. We then apply the \texttt{iUzawa-Net} to learn the solution operator that maps the problem data $(y_d, f)$ to the sparse optimal control $u^*$.
For data generation, the spatial domain $\Omega$ is discretized as a uniform grid with a resolution of $m \times m$, and the time interval $[0, T_f]$ is discretized as $m_T$ uniform time steps.
We generate a training set of $N=16384$ samples with $m = m_T = 32$ by drawing functions $\{(y_d)_i\}_{i=1}^N$ and $\{f_i\}_{i=1}^N$ independently from the GRF
\begin{equation*}
    (y_d)_i, \ f_i \sim \frac{\sigma}{\sqrt{2}} \left( \mathcal{GR}(0, (-\Delta_D + 9I)^{-3}) + \mathcal{GR}(0, (-\Delta_N + 9I)^{-3}) \right),
\end{equation*}
where $\Delta_D$ and $\Delta_N$ now denote the Laplacian operators on $\Omega_T$ with homogeneous Dirichlet and Neumann boundary conditions, respectively, and $\sigma = {0.05 \cdot (256)^{3/2}}$ for $(y_d)_i$ and $\sigma = {0.1 \cdot (256)^{3/2}}$ for $f_i$. 
For each data pair $((y_d)_i, f_i)$, the corresponding optimal control $u^*_i$ is computed using a high-fidelity SSN algorithm \cite{ulbrich2002semismooth}, providing reference solutions for training.
A testing set of $N=2048$ samples with $m = m_T = 32$ is generated following the same procedure.
To assess the model's zero-shot super-resolution generalization capability, we construct three additional testing sets of sizes $N=64, \, 64, \, 16$ with finer resolutions $m = m_T = 64, \, 128, \, 256$, respectively.

We construct two \texttt{iUzawa-Net}s, named \texttt{iUzawa-Net-F} and \texttt{iUzawa-Net-S}, and denote them uniformly by $\mathcal{T}(y_d, f; \theta_\cT)$.
The \texttt{iUzawa-Net-F} employs distinct weights for the neural networks $\cS^k$, $\cA^k$, $\cQ_A^k$, and $\cQ_S^k$ in each layer $k$, while the \texttt{iUzawa-Net-S} uses shared weights across all layers.
We take $L = 5$ layers for both the \texttt{iUzawa-Net-F} and the \texttt{iUzawa-Net-S}.
The architecture of the modules is specified as follows.
We adopt $3$ layers in the FNOs $\cS^k$ with a lifting dimension $m_p = 8$.
For each Fourier layer, the truncated Fourier modes are set to $k_{\text{max}} = 8$ in each coordinate, and the GeLU activation function is applied.
Since $S^* = \mathcal{R}_T \circ S \circ \mathcal{R}_T$, where $(\mathcal{R}_T v)(x, t) = v(x, T_f - t)$ denotes the time-reversal operator, we reuse $\mathcal{S}^k$ to define $\mathcal{A}^k = \mathcal{R}_T \circ \mathcal{S}^k \circ \mathcal{R}_T$ in each layer.
For $\cQ_S^k$, the lifting dimension is set to $m_p = 8$, the truncated Fourier modes are set to $k_{\text{max}} = 8$ in each coordinate, and $\gamma = 10^{-6}$.
We set the width of all hidden layers of $\cQ_A$ to $m_p = 64$.
We take $\tau = 10^{-4}$ for both the \texttt{iUzawa-Net-F} and the \texttt{iUzawa-Net-S}.
All inputs are padded to a resolution of $36$ for accommodating non-periodic boundary conditions.
The total numbers of trainable parameters are $3{,}360{,}225$ and $672{,}045$ for the \texttt{iUzawa-Net-F} and the \texttt{iUzawa-Net-S}, respectively.
All the learnable parameters are initialized using the default settings of \texttt{PyTorch}.

The networks are trained by minimizing the discretized empirical loss function of \eqref{eq:relative-l2-err-def} with $\varepsilon_L = 10^{-8}$.
We use the \texttt{AdamW} optimizer for $300$ epochs with a batch size of $32$. The learning rate is initialized to $5 \times 10^{-4}$ and is decayed by a factor of $0.6$ every $30$ epochs.
The training process takes approximately $2.4 \times 10^{3}$ and $2.3 \times 10^{3}$ seconds for the \texttt{iUzawa-Net-F} and \texttt{iUzawa-Net-S}, respectively.

After training, the performance of the trained models is evaluated on the testing sets using the following absolute and relative $L^2$-errors:
\begin{equation*}
    \varepsilon_{\text{abs}}(\hat{u}) = \left\|\hat{u} - u^* \right\|_{L^2(\Omega_T)}, \quad \varepsilon_{\text{rel}}(\hat{u}) = \frac{\varepsilon_{\text{abs}}(\hat{u})}{\left\| u^* \right\|_{L^2(\Omega_T)}}.
\end{equation*}
The numerical results are summarized in \Cref{tab:parabolic-numerical-results-uzawa}, and the solution of a randomly chosen instance from the testing set is visualized in \Cref{fig:parabolic-results}.
These results affirm that both the \texttt{iUzawa-Net-F} and the \texttt{iUzawa-Net-S} achieve high accuracy for solving the parabolic optimal control problem \eqref{eq:parabolic-opt-ctrl}. Notably, the data presented in \Cref{tab:parabolic-numerical-results-uzawa} highlights the models' zero-shot generalization capability, as they perform well on spatial and temporal resolutions unseen during training.

\begin{table}[!ht]
    \centering
    \small
    \begin{tabular}{c c c c c c}
        \toprule
        \multirow{2}{*}{Method} & \multirow{2}{*}{$m=m_T$}
        & \multicolumn{2}{c}{$\varepsilon_{\mathrm{rel}}(\hat{u})$}
        & \multicolumn{2}{c}{$\varepsilon_{\mathrm{abs}}(\hat{u})$} \\
        \cmidrule(lr){3-4} \cmidrule(lr){5-6}
        & & Mean & SD & Mean & SD \\
        \midrule
        \multirow{4}{*}{\texttt{iUzawa-Net-F}}
        & $32$  & $3.43\times10^{-3}$ & $1.68\times10^{-3}$ & $7.19\times10^{-3}$ & $2.03\times10^{-3}$ \\
        & $64$  & $1.45\times10^{-2}$ & $3.59\times10^{-3}$ & $3.20\times10^{-2}$ & $1.03\times10^{-2}$ \\
        & $128$ & $2.00\times10^{-2}$ & $3.88\times10^{-3}$ & $4.74\times10^{-2}$ & $1.46\times10^{-2}$ \\
        & $256$ & $2.19\times10^{-2}$ & $3.56\times10^{-3}$ & $5.96\times10^{-2}$ & $2.21\times10^{-2}$ \\
        \midrule
        \multirow{4}{*}{\texttt{iUzawa-Net-S}}
        & $32$  & $4.96\times10^{-3}$ & $2.83\times10^{-3}$ & $1.02\times10^{-2}$ & $3.17\times10^{-3}$ \\
        & $64$  & $1.52\times10^{-2}$ & $4.12\times10^{-3}$ & $3.32\times10^{-2}$ & $9.90\times10^{-3}$ \\
        & $128$ & $2.02\times10^{-2}$ & $3.79\times10^{-3}$ & $4.79\times10^{-2}$ & $1.44\times10^{-2}$ \\
        & $256$ & $2.25\times10^{-2}$ & $3.30\times10^{-3}$ & $6.07\times10^{-2}$ & $2.11\times10^{-2}$ \\
        \bottomrule
    \end{tabular}
    \caption{Numerical results of \texttt{iUzawa-Net-F} and \texttt{iUzawa-Net-S} for solving \eqref{eq:parabolic-opt-ctrl}.}
    \label{tab:parabolic-numerical-results-uzawa}
\end{table}

\begin{figure}[!ht]
    \centering
    \hfill
    \begin{subfigure}{0.32\textwidth}
        \includegraphics[width=\textwidth]{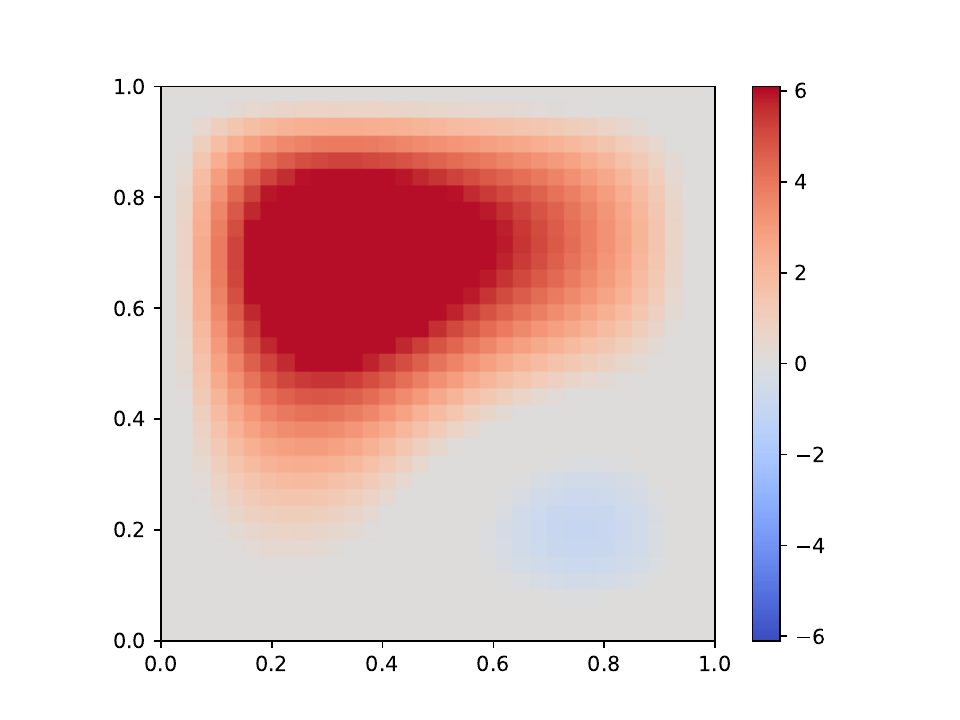}
        \subcaption{Exact $u^*$ at $t=0.25$.}
    \end{subfigure}
    \hfill
    \begin{subfigure}{0.32\textwidth}
        \includegraphics[width=\textwidth]{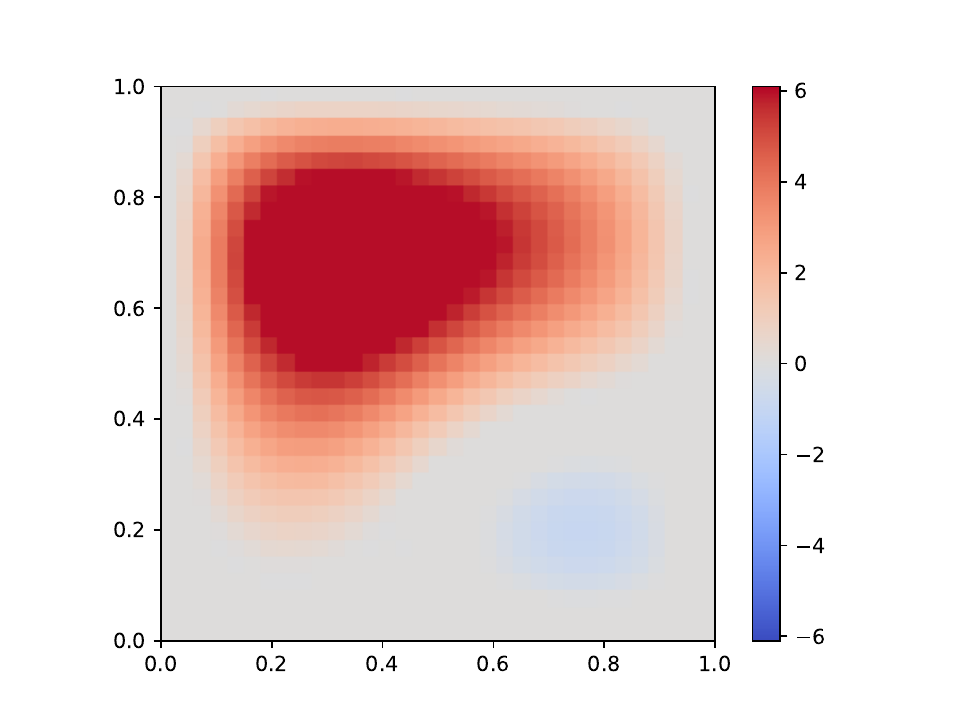}
        \subcaption{Computed $\hat{u}$ at $t=0.25$.}
    \end{subfigure}
    \hfill
    \begin{subfigure}{0.32\textwidth}
        \includegraphics[width=\textwidth]{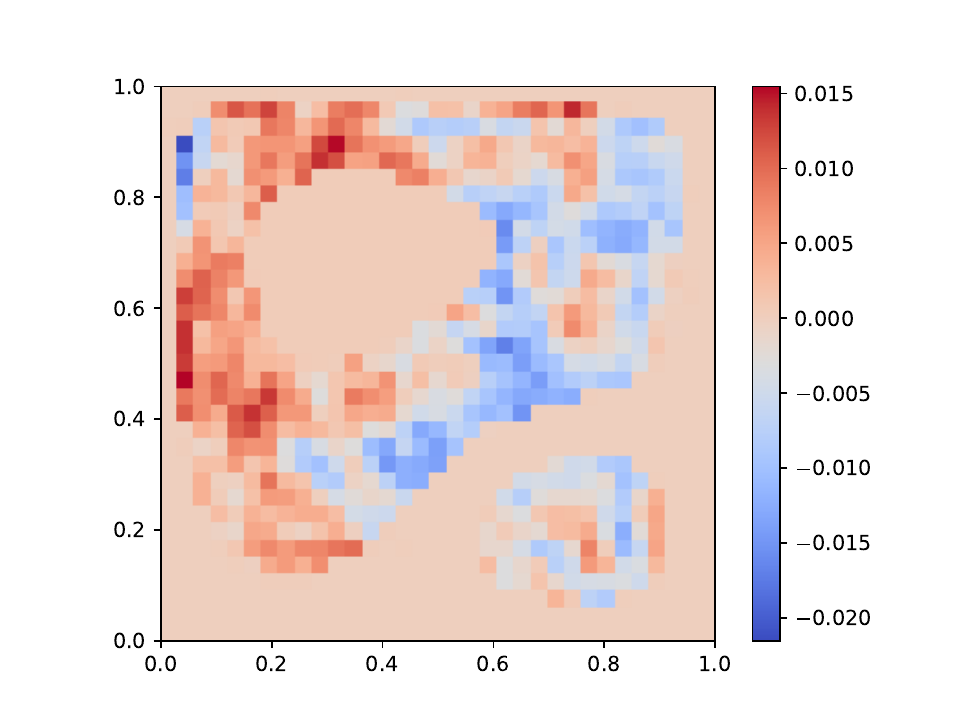}
        \subcaption{$\hat{u} - u^*$ at $t=0.25$.}
    \end{subfigure}
    \hfill

    \hfill
    \begin{subfigure}{0.32\textwidth}
        \includegraphics[width=\textwidth]{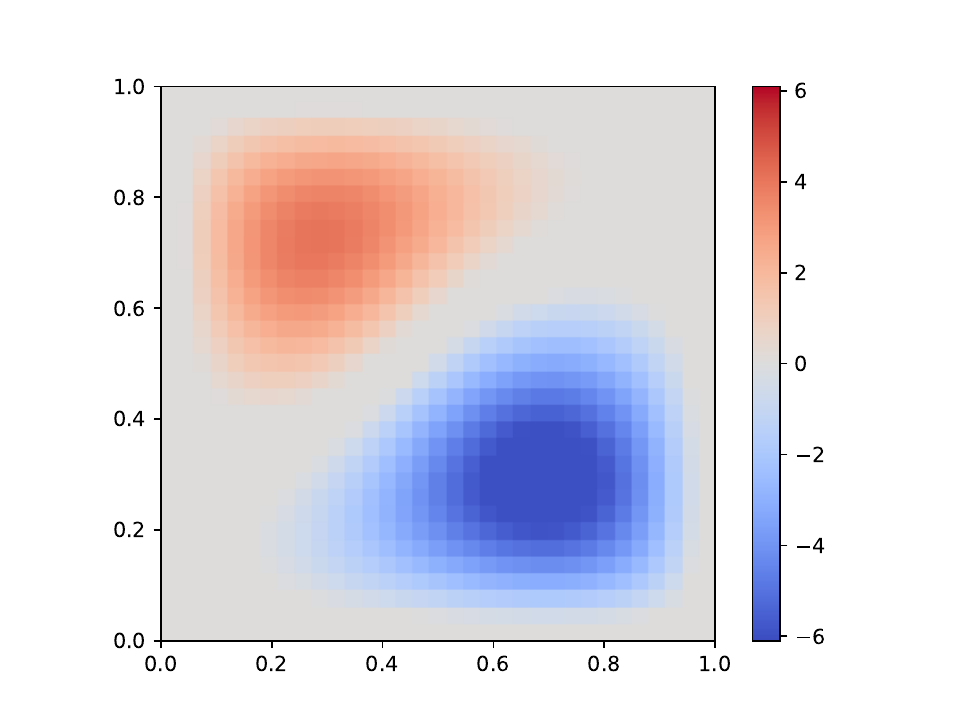}
        \subcaption{Exact $u^*$ at $t=0.5$.}
    \end{subfigure}
    \hfill
    \begin{subfigure}{0.32\textwidth}
        \includegraphics[width=\textwidth]{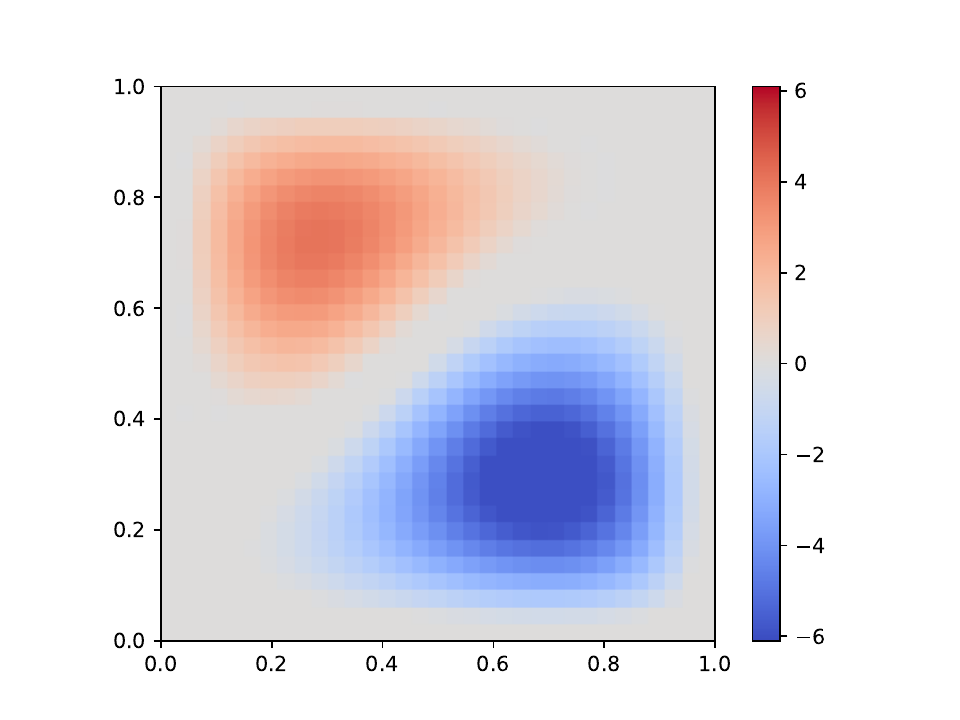}
        \subcaption{Computed $\hat{u}$ at $t=0.5$.}
    \end{subfigure}
    \hfill
    \begin{subfigure}{0.32\textwidth}
        \includegraphics[width=\textwidth]{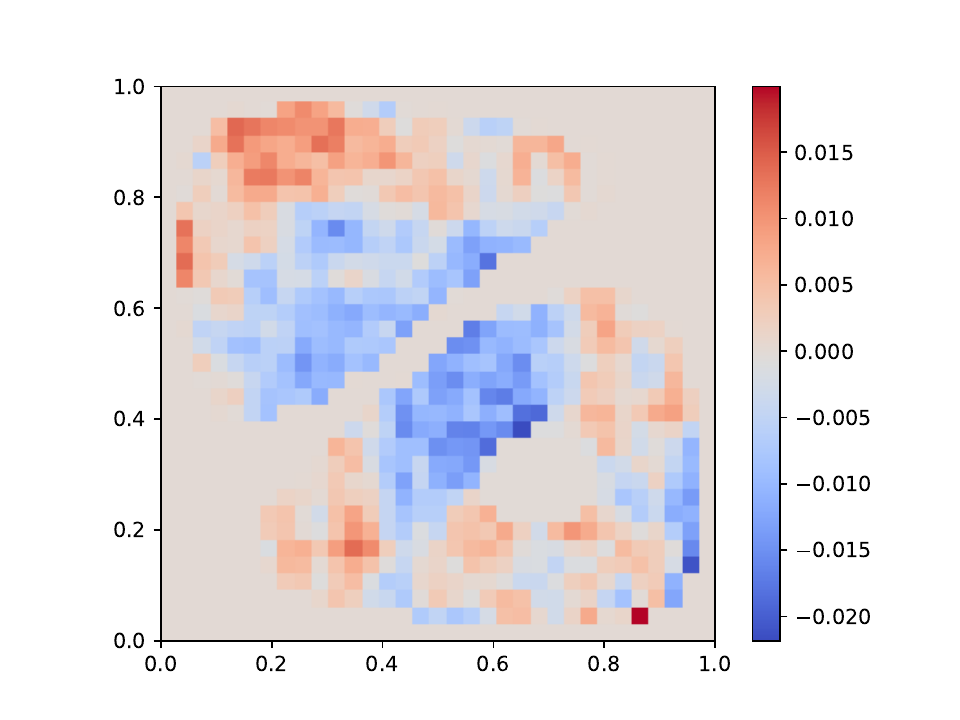}
        \subcaption{$\hat{u} - u^*$ at $t=0.5$.}
    \end{subfigure}
    \hfill

    \hfill
    \begin{subfigure}{0.32\textwidth}
        \includegraphics[width=\textwidth]{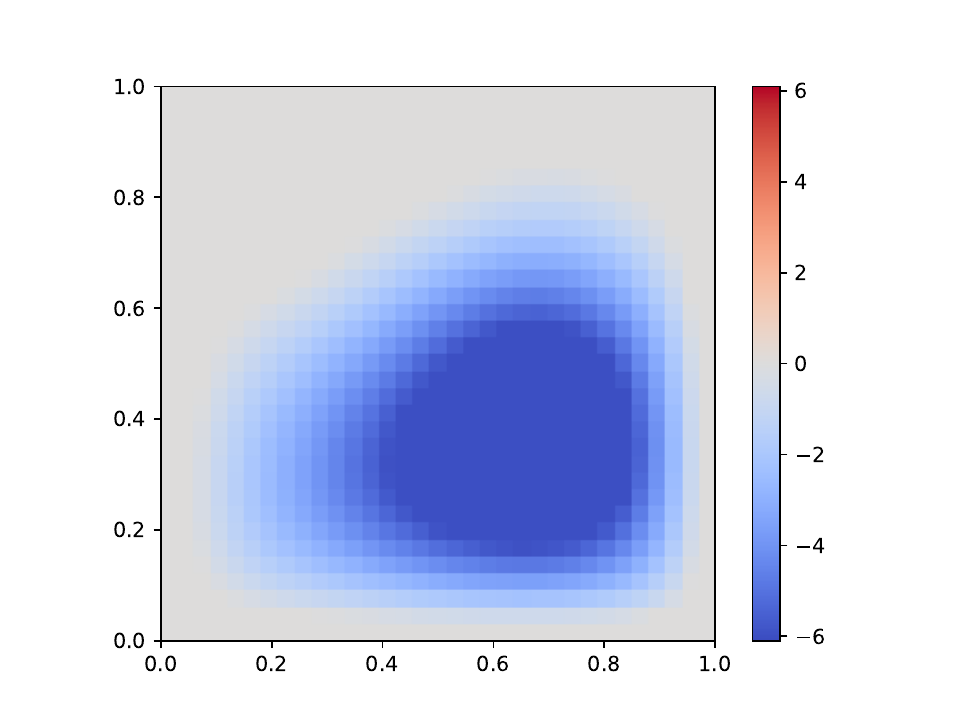}
        \subcaption{Exact $u^*$ at $t=0.75$.}
    \end{subfigure}
    \hfill
    \begin{subfigure}{0.32\textwidth}
        \includegraphics[width=\textwidth]{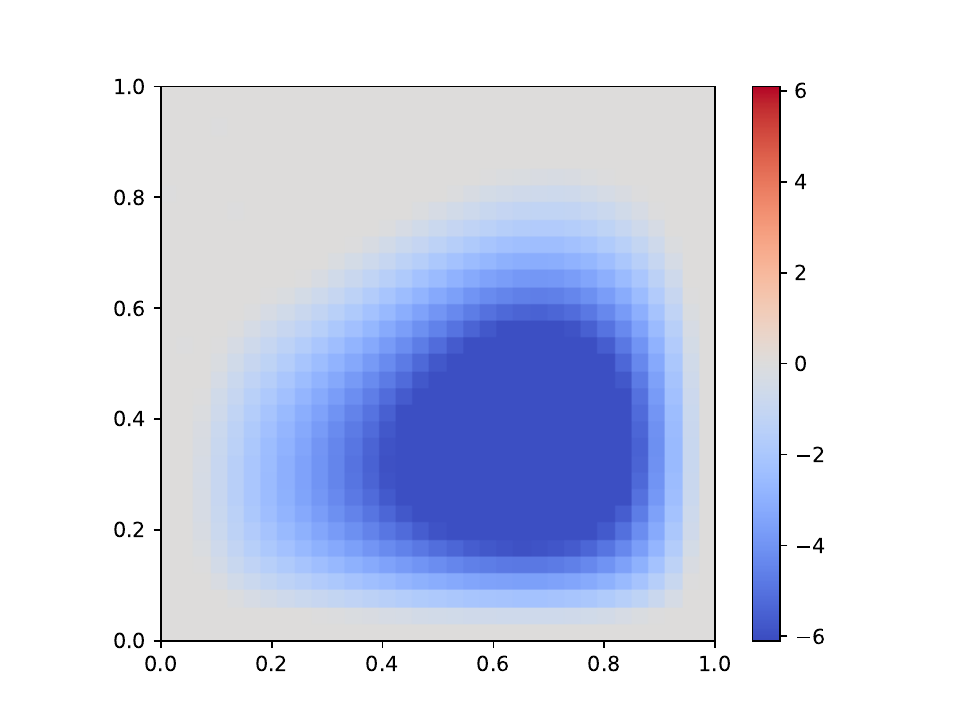}
        \subcaption{Computed $\hat{u}$ at $t=0.75$.}
    \end{subfigure}
    \hfill
    \begin{subfigure}{0.32\textwidth}
        \includegraphics[width=\textwidth]{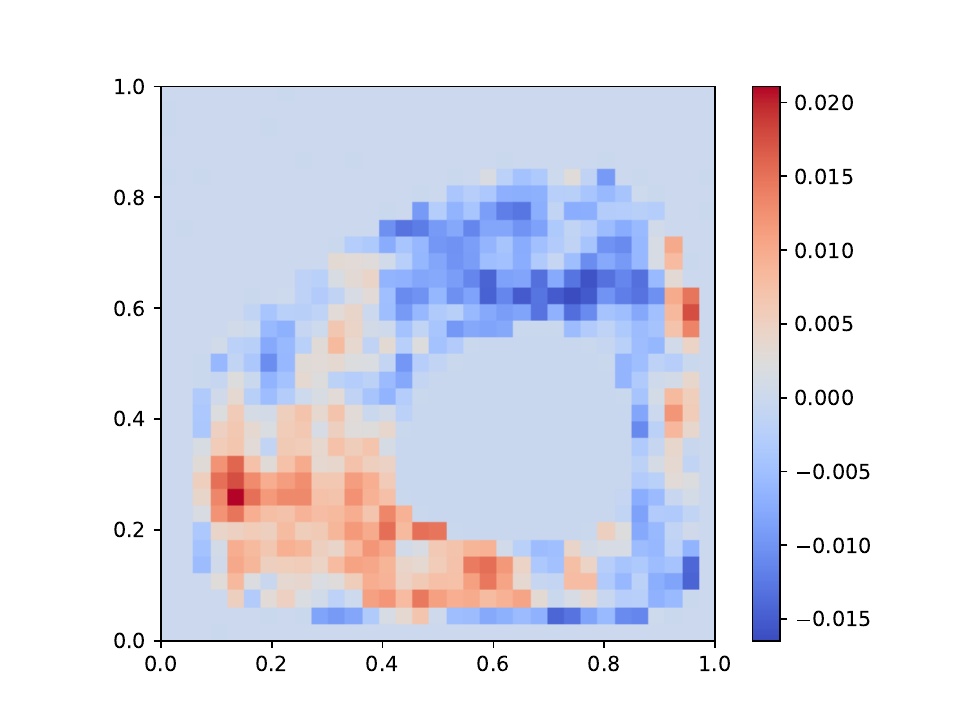}
        \subcaption{$\hat{u} - u^*$ at $t=0.75$.}
    \end{subfigure}
    \hfill
    \caption{Computed optimal control for a single instance of \eqref{eq:parabolic-opt-ctrl} with $m=m_T=32$.}
    \label{fig:parabolic-results}
\end{figure}

Finally, we compare the computational efficiency between the trained \texttt{iUzawa-Net}s and an SSN method, the primal-dual method in \cite{chambolle2011first,he2012convergence}, an inexact Uzawa method, and a POD-based reduced-order method \cite{hinze2005proper,troltzsch2009pod}.
For the SSN method, we employ a modified version of the implementation in \cite[Section 4]{wang2024augmented}, solving the resulting linear systems with a preconditioned conjugate gradient method and multigrid cycles.
The inexact Uzawa method is a straightforward extension of the approach in \cite{song2019inexact} to the parabolic setting.
For the POD method, after generating the spatial POD basis, the reduced state and adjoint equations are discretized by the implicit Euler method, and the control is recovered from the optimality condition.
To ensure a fair comparison, matrices in the linear systems are pre-factorized prior to timing where applicable.
The algorithms are terminated once any iteration $\hat{u}$ satisfies $\varepsilon_{\text{rel}}(\hat{u}) \leq \varepsilon_\text{rtol}$.
We set $\varepsilon_\text{rtol} = 3 \times 10^{-3}$ and $\ell = 40$ to match the accuracy achieved by the \texttt{iUzawa-Net}s at the training resolution $m = m_T = 32$.
For the POD method, we set $N_{\text{POD}} = 1024$ for resolutions $m = m_T = 32$, $64$, and $128$, and $N_{\text{POD}} = 32$ for $m = m_T = 256$ due to the prohibitive computational and memory cost required for larger datasets\footnote{As a reference, taking $N_{\text{POD}}=32$ leads to a decrease of $11$--$14$\% in accuracy for resolutions below $128$ compared with $N_{\text{POD}}=1024$.}.
The basis generation takes $4.4$, $60.1$, $651.0$, and $648.2$ seconds for $m = m_T = 32$, $64$, $128$, and $256$, respectively.

The computation time of different numerical methods is reported in \Cref{tab:parabolic-computation-time}.
The results clearly indicate that the \texttt{iUzawa-Net-F} and \texttt{iUzawa-Net-S} significantly outperform the traditional numerical methods across all tested resolutions by a factor of $250$ to $5{,}500$, and they are faster than the POD method by a factor of $20$ to $110$.
Among the FEM-based algorithms, the POD method is faster than the full-space methods, and all their computation time grows rapidly with $m$ and $m_T$.
This further validates the \texttt{iUzawa-Net} as a highly efficient solver for nonsmooth optimal control problems, underscoring the advantage of our approach where a forward-pass neural network architecture replaces a computationally costly iterative procedure.
We also observe that the inference time of the \texttt{iUzawa-Net}s grows only mildly with respect to the number of grid points.

\begin{table}[t]
    \centering
	\small
    \begin{tabular}{c c c c c c c}
		\toprule
		$m=m_T$ & SSN & Primal-Dual & Inexact Uzawa & POD & \texttt{iUzawa-Net-F} & \texttt{iUzawa-Net-S} \\
		\midrule
		$32$ & $3.27 \times 10^{3}$ & $355.4$ & $180.0$ & $16.0$ & $0.7$ & $0.7$ \\
		$64$ & $7.81 \times 10^{3}$ & $2.08 \times 10^{3}$ & $457.6$ & $95.0$ & $1.6$ & $1.6$ \\
		$128$ & $3.63 \times 10^{4}$ & $8.75 \times 10^{3}$ & $2.69 \times 10^{3}$ & $722.8$ & $7.1$ & $7.1$ \\
        $256$ & $2.79 \times 10^{5}$ & $7.79 \times 10^{4}$ & $2.85 \times 10^{4}$ & $5.70 \times 10^{3}$ & $50.6$ & $50.8$ \\
		\bottomrule
	\end{tabular}
    \caption{Computation time (milliseconds) of the SSN method, the primal-dual method, the inexact Uzawa method, the POD method, the \texttt{iUzawa-Net-F}, and the \texttt{iUzawa-Net-S} for solving \eqref{eq:parabolic-opt-ctrl} with different resolutions $m$ and $m_T$ (averaged over instances in the testing sets).}
    \label{tab:parabolic-computation-time}
    \normalsize
\end{table}

\subsection{Concluding Remarks on Numerical Results}

The optimization-informed \texttt{iUzawa-Net}s achieve satisfactory accuracy, which is considerably higher than that of the standard FNOs with comparable numbers of parameters.
Once trained, the \texttt{iUzawa-Net}s are significantly faster than all the compared traditional algorithms, and their computation time remains nearly unchanged in the ill-conditioned case.
As the grid is refined, the computation time of the traditional algorithms grows rapidly, while that of the \texttt{iUzawa-Net}s increases only mildly.
Moreover, the \texttt{iUzawa-Net}s are trained only once on a coarse grid and applied to finer grids without retraining, whereas the POD basis has to be regenerated for each resolution.
Notably, the advantage of the \texttt{iUzawa-Net} is more significant for the time-dependent parabolic problem.
Finally, the \texttt{iUzawa-Net} is particularly suitable for real-time scenarios, where fast online solutions take absolute priority over achieving the highest possible accuracy or the lowest offline cost.

\section{Conclusions}\label{sec:conclusion}

We propose the \texttt{iUzawa-Net}, an optimization-informed deep neural network architecture that unrolls an inexact Uzawa method, for a general class of parameterized nonsmooth optimal control problems of linear PDEs. The proposed \texttt{iUzawa-Net} is capable of learning the solution operator that maps parameters defining the problem to the corresponding solution, effectively replacing iterations of a traditional solver with a forward pass of the neural network and hence making real-time solutions possible. Theoretically, we prove both universal approximation properties and the asymptotic $\varepsilon$-optimality with respect to network depth under the weight-tying setting. 
Numerical experiments for solving nonsmooth elliptic and parabolic optimal control problems demonstrate that the \texttt{iUzawa-Net} both achieves high accuracy with few layers and exhibits excellent generalization capability in zero-shot super-resolution scenarios.
With real-time solvability, ease of implementation, structural interpretability, and theoretical guarantees, the \texttt{iUzawa-Net} is a new and powerful tool for solving optimal control problems of PDEs.

\bigskip

\bibliographystyle{siamplainmc}
{\small
\bibliography{references}
}

\end{document}